\documentclass[12pt,reqno]{amsart}
\usepackage[letterpaper,margin=1.1in]{geometry}

\usepackage{amsmath,amssymb,amsthm,verbatim,color,hyperref}

\newtheorem{theorem}[subsection]{Theorem}
\newtheorem{lemma}[subsection]{Lemma}
\newtheorem{corollary}[subsection]{Corollary}
\newtheorem{prop}[subsection]{Proposition}

\theoremstyle{definition}
\newtheorem{definition}[subsubsection]{Definition}
\newtheorem{remark}[subsection]{Remark}
\newtheorem{example}[subsection]{Example}

\newcommand{\haus}{\mathcal{H}}
\newcommand{\leb}{\mathcal{L}}
\newcommand{\spt}{\mathrm{spt}}
\newcommand{\diam}{\mathrm{diam}}

\newcommand{\reg}{\mathrm{reg}}

\newcommand{\setint}{\mathrm{int}}

\newcommand{\graph}{\mathrm{graph}}

\newcommand{\eps}{\epsilon}
\newcommand{\proj}{\mathrm{proj}}

\newcommand{\sing}{\mathrm{sing}}
\newcommand{\cF}{\mathcal{F}}

\newcommand{\R}{\mathbb{R}}
\newcommand{\Z}{\mathbb{Z}}

\newcommand{\del}{\partial}

\newcommand{\cN}{\mathcal{N}}

\newcommand{\cP}{\mathcal{P}}
\newcommand{\cI}{\mathcal{I}}

\newcommand{\cT}{\mathcal{T}}
\newcommand{\cIH}{\mathcal{IVT}}
\newcommand{\cD}{\mathcal{D}}
\newcommand{\cIV}{\mathcal{IV}}

\newcommand{\cIVT}{\mathcal{IVT}}
\newcommand{\OmegaRef}{{\Omega^{(0)}}}
\newcommand{\OmegaRefN}{{\Omega^{(0)}_0}}
\newcommand{\bY}{\mathbf{Y}}
\newcommand{\sint}{\mathrm{int}}

\title[Regularity in locally polyhedral domains]{Regularity of free boundary minimal surfaces in locally polyhedral domains}
\author{Nick Edelen and Chao Li}
\address{ \noindent Nick Edelen: 
	\newline University of Notre Dame, 255 Hurley Bldg, Notre Dame, IN 46556, USA
	\newline
	\textit{E-mail address: nedelen@nd.edu} 
	\newline \newline \indent Chao Li: 
	\newline Princeton University, Fine Hall, 304 Washington Road, Princeton, NJ 08540, USA
	\newline \textit{E-mail address: chaoli@math.princeton.edu} }

\numberwithin{equation}{section}

\begin{document}

\begin{abstract}
We prove an Allard-type regularity theorem for free-boundary minimal surfaces in Lipschitz domains locally modelled on convex polyhedra.  We show that if such a minimal surface is sufficiently close to an appropriate free-boundary plane, then the surface is $C^{1,\alpha}$ graphical over this plane.  We apply our theorem to prove partial regularity results for free-boundary minimizing hypersurfaces, and relative isoperimetric regions.
\end{abstract}
\maketitle

\tableofcontents

\section{Introduction}

We are interested in the regularity of free-boundary minimal hypersurfaces $M$ inside piecewise-smooth domains $\Omega$.  These surfaces $M$ arise variationally as critical points of area or capillary type functions among surfaces in $\Omega$ whose boundaries lie in $\del\Omega$ but are otherwise free to vary.  The existence and regularity of free-boundary minimal surfaces has been extensively studied by Courant \cite{Courant1977minimal}, Lewy \cite{Lewy1951boundarybehavior,Lewy1951partiallyfree}, Hildebrandt \cite{Hildebrandt1969boundary}, Nitsche \cite{HildebrandtNische1979minimal,Nitsche1985stationary}, Gruter \cite{Gr} and Jost \cite{GrJo}, Taylor \cite{Taylor1976structure,Taylor:cap}, Struwe \cite{struwe1984freeboundary}, among others.  

When $\partial \Omega$ is at least $C^2$, Gruter-Jost \cite{GrJo} proved an Allard-type regularity theorem which says that if $M$ is sufficiently varifold close to a free-boundary plane (or, equivalently in this case, when the density ratio of $M$ near a boundary point is sufficiently close to $1/2$), then nearby $M$ is a $C^{1,\alpha}$ graph over this plane.  Gruter \cite{Gr} used this regularity theorem to prove that the (optimal) dimension of the singular set at the free-boundary for area-minimizing hypersurfaces satisfies the same codimension $7$ bound as in the interior (contrast this with Hardt-Simon's \cite{HaSi:boundary} result that showed every area-minimizing hypersurface is entirely regular near its Dirichlet boundary).

In this paper we prove an Allard-type regularity for $M$ when $\Omega$ is locally modeled on any convex polyhedral cone, and hence is only piecewise $C^2$.  The archtype of such domain is a $C^2$ perturbation of a convex wedge (i.e. intersection of two half-spaces).  Existence and regularity of free-boundary minimal $M$ in non-smooth $\Omega$ has a history dating back to the 19th century, when Gergonne \cite{gergonne1816questions} and Schwartz \cite{schwarz1872fortgesetzte} formulated and solved the question of determining minimal surfaces with partially free boundary in a standard cube in $\R^3$.  Since then, there has been a rich history of investigation on free boundary minimal surfaces in various geometric and physical scenarios, where polyhedral domains naturally arise. For instance, when $\Omega$ is a wedge region in $\R^3$, with opening angle $\theta_0$, a range of geometric and regularity properties of free boundary minimal surfaces were studied in \cite{HildebrandtSauvigny1997minimalI,HildebrandtSauvigny1997minimalII,HildebrandtSauvignyminimalIII,HildebrandtSauvignyminimalIV}. See also \cite{HildebrandtNische1979minimal}, where $\theta_0=2\pi$. 

In these past results, as well as in \cite{Jost1987continuity,GruterHildebrandtNitsche1981boundary}, it was observed that boundary regularity of minimal surfaces depends on the local structure of $\Omega$. For instance, if $\Omega$ is a wedge in $\R^3$ as above, and $\theta_0=2\pi$, then an area minimizing surface $M$ may have branch points at its free boundary (see \cite{HildebrandtNische1979minimal}). More generally, as pointed in \cite{HildebrandtSauvigny1997minimalII}, when $\theta_0\in (\pi, 2\pi)$, an area-minimizing surface in $\Omega$ may contain an interval on the edge $\{0\}\times \R$, and thus fail to be a regular surface meeting $\partial \Omega$ orthogonally (called the edge-creeping phenomenon). For a beautuful local description of two-dimensional minimal surfaces in wedges using the Weierstrass representation, see \cite{Brakke1992minimal}.

In polyhedral domains, similar regularity questions also naturally appear in geometric problems with other type of boundary conditions, including capillary surfaces \cite{Si:capillary} \cite{Taylor:cap} as well as general soap bubbles \cite{Gromov2014dirac}. These surfaces are crucial in Gromov's geometric comparison principle for scalar curvature proved by the second named author \cite{Li:3,Li:n}.

Simon \cite{Si:capillary} implemented tools of geometric measure theory in the investigation of capillary surfaces in domains with corners. For a free boundary minimal surface, he noted that \textit{local convexity} of the model domain is a sufficient condition for the existence of a nontrivial tangent cone at the corner. On the other hand, cusp type singularities may occur at the corner if local convexity is violated (see \cite{Si:capillary}\cite{ConcusFinn1974capillary}\cite{ConcusFinn1996capillary} and the references therein).
\vspace{5mm}

Our main result is a local Allard-type regularity theorem, which says loosely that if $\Omega \cap B_1$ is a sufficiently small $C^2$ perturbation of a convex polyhedral cone $\OmegaRef$, and if $M \llcorner B_1$ is sufficiently varifold close to an appropriate free-boundary plane in $\OmegaRef$, then $\spt M \cap B_{1/2}$ is a $C^{1,\alpha}$ perturbation of this plane.
\begin{theorem}\label{thm:teaser}
Let $\Omega^{n+1} = \Omega_0 \times \R$ be a polyhedral cone domain: a dilation-invariant intersection of finitely-many closed half-spaces, with non-empty interior.  Let $F_i : B_1 \to B_1$ be a sequence of $C^2$ diffeomorphisms which limit to $\mathrm{id}$ in $C^2(B_1)$.  Let $M_i$ be a sequence of integral varifolds in $B_1$ which are stationary free-boundary in $F_i(\Omega)$ (or have bounded mean curvature tending to $0$), such that $M_i \to [\Omega_0 \times \{0\}]$ as varifolds in $B_1$.


Then for $i \gg 1$, we can write $F_i^{-1}(\spt M) \cap B_{1/2}$ as $C^{1,\alpha}$ graphs over $\Omega_0 \times \{0\}$, with $C^{1,\alpha}$ norm tending to zero.

\end{theorem}
Theorem \ref{thm:teaser} is stated for codimension-one varifolds in Euclidean space, but versions hold for higher-codimension varifolds and other ambient manifolds; see Section \ref{sec:codim}.

We note that there are two classes of free-boundary planes: if $\Omega = W^2 \times \R^{n-1}$ is a wedge, then one has ``horizontal'' planes containing $W^2 \times \{0\}$, and ``vertical'' planes containing $\{0\}\times \R^{n-1}$.   Our Theorem holds only for horizontal planes, and in general an Allard-type regularity fails for vertical planes (Example \ref{ex:vertical-plane}).  Likewise, regularity as in Theorem \ref{thm:main} can fail when $\Omega$ is non-convex (Example \ref{ex:non-convex}).

We also remark that, by considering the corresponding linear Neumann problem, we expect $C^{1,\alpha}$ regularity to be sharp for general $\Omega$.  For example, when $\Omega$ is a wedge with opening angle $\theta_0$, then we would expect no better regularity than $C^{1,\alpha}$ for $\alpha = \frac{\pi}{\theta_0}-1$.  If the opening angle is $\leq \pi/2$, then one might expect $C^{2,\alpha}$ regularity.  In the special case when $\Omega = W^2\times\R$ is a $3$-dimensional wedge, then Theorem \ref{thm:teaser} was proven for minimal graphs over the free-boundary plane $W^2 \times \{0\}$ by \cite{HildebrandtSauvignyminimalIV}.


\vspace{5mm}
Our regularity theorem implies a partial regularity for codimension-one minimizing currents in domains which are locally modeled on polyedral cone domains.  The interior dimension bound of course follows from classical interior regularity -- our contributions are the estimates on the boundary.  See Section \ref{sec:minz} for exact definitions and the proof.
\begin{theorem}[Partial regularity]\label{thm:partial-reg}
	Let $\Omega$ be a locally polyhedral $C^2$ domain in a complete $C^3$ Riemannian manifold $N^{n+1}$.  Suppose $T = T \llcorner \sint \Omega$ is a free-boundary area-minimizing integral current in $\Omega$, or that $T = \del[U] \llcorner \sint\Omega$ is a solution to the relative isoperimetric problem, i.e. $T$ is area-minimizing in the class of $T' = \del[U'] \llcorner \sint\Omega$ where $U'$ is relatively open in $\Omega$ and $|U'| = |U|$.

	
	Then $\dim (\sing T\cap \sint \Omega)\le n-7$, and:
	\begin{enumerate}
		\item $\dim(\sing T\cap \partial \Omega)\le n-2$, for general $\Omega$;
		\item $\dim(\sing T\cap \partial \Omega)\le n-3$, if the dihedral angles of $\Omega$ are $\leq \pi/2$ (in which case $\Omega$ is actually a domain-with-corners);
		\item $\dim(\sing T\cap \partial \Omega)\le n-7$, if the dihedral angles of $\Omega$ are $= \pi/2$. 
	\end{enumerate}
\end{theorem}
We emphasize that $\Omega$ need only be locally a $C^2$ perturbation of a convex polyhedral domain, and does not need not be (locally) convex itself.

We think it should be true that $\dim (\sing T \cap \del \Omega) \leq n-7$ when the dihedral angles are $\leq \pi/2$.  The stumbling block is proving an appropriate Neumann eigenvalue bound for domains contained in an octant of $S^2$.  In this spirit, when the dihedral angles of $\Omega$ are $\leq \pi/2$ our proof in fact gives $\dim(\sing M \setminus \del_{n-2}\Omega) \leq n-7$, where $\del_{n-2}\Omega$ consists of points where $\Omega$ is locally modelled on domains with at least $n-2$ dimensions of symmetry.  On the other hand, it is plausible that the $n-2$ bound for general $\Omega$ is sharp, as there are free-boundary planes in $3$-dimensional polyhedral cone domains which we think may be minimizing (Example \ref{ex:possible-minz-plane}).


\vspace{5mm}

Our strategy is to prove a certain ``excess decay'' inequality \eqref{eqn:outline-1}.  The two key difficulties we encounter are a lack of reflection principle, and a lack of a single boundary model.  This means that, unlike regularity with smooth boundaries (such as Allard \cite{All:boundary}, Gruter \cite{Gr}), we do not have any easy characterization of low-density tangent cones, nor do we have a nice monotone quantity defined at all scales and all points along the boundary.  Our proof therefore relies comparatively little on monotonicity, in contrast to \cite{All:boundary}, \cite{Gr} who used it as a key tool to obtain among other things good effective graphical approximations.

Instead, we prove a trace-like inequality for the first variation, and use Moser iteration to prove good lower ahlfors regularity and non-concentration estimates.  The key to make our excess decay argument work is a sharp $L^\infty$-$L^2$ bound, which implies that, even at the scale of the excess, the $L^2$ norm cannot ``accumulate'' near the boundary.  The main technical challenge is establishing a first-variation control, and the corresponding trace-like inequality.  To do this we induct on the dimension of the cone, basing our argument on work of Simon \cite{Si:capillary} who considered $2$-dimensional capillary surfaces in a $3$-dimensional wedge.

While $\Omega$ resembles a fixed cone-type, we can prove excess decay by comparison with the linear problem, in the spirit of DeGiorgi's original proof of interior regularity.  However, inevitably at certain finite scales $\Omega$ will cease to look like a given cone, and here this argument breaks down (e.g. in a wedge, every big ball looks like the wedge, but as the radius shrinks you may look like a half-space, or an interior point).  By an inductive argument on the strata, we show that $\Omega$ looks like one of finitely-many cone types, away from finitely many scales.  By giving up some constant, we can chain together our excess decays at each cone model to obtain a global excess decay down to $0$.  This last step is very general, and would apply to any regularity problem in a ``stratified'' model.

It would be interesting to see if any of our techniques carry over to regularity problems with other boundary conditions, for example the capillary problem.  De Philippis-Maggi \cite{DephilippisMaggi2015regularity} proved a (partial-)regularity theorem for energy-minimizing capillary surfaces in $C^{1,1}$ domains, and the second author \cite{Li:3} proved regularity of energy-minimizing $2$-dimensional capillary surfaces in locally polyhedral Lipschitz domains.  However very little is known for capillary surfaces which are only stationary for area, nor is the sharp regularity known for energy-minimizing capillary surfaces in higher dimensions.  We mention here that recently Kagaya-Tonegawa \cite{KagayaTonegawa2017contact} proved a monotonicity formula for stationary capillary surfaces.

\textbf{Acknowledgements: } The authors thank Rick Schoen, Luca Spolaor, Brian White, and Neshan Wickramasekera for helpful conversations. The authors also thank the anonymous referee for suggestions on improvements of the paper. C.L. wants to thank the Institute for Advanced Study, where part of this work was carried out. N.E. was supported in part by NSF grant DMS-1606492.  C.L. was supported in part by NSF grant DMS-2005287.

\section{Notation and preliminaries}

\addtocontents{toc}{\setcounter{tocdepth}{1}}

\subsection{Notation}

We work in $\R^{n+1}$.  Given a subset $A \subset \R^{n+1}$, we define $d_A(x) = \inf_{a \in A} |x - a|$ to be the usual Euclidean distance to $A$.  If $A = \emptyset$, we define $d_A(x) = \infty$.  Write $B_r(A) = \{ x : d_A(x) < r \}$ for the open $r$-tubular neighborhood of $A$, and write $\overline{A}$ for the closure of $A$.  $B_r(x)$ is the open $r$-ball centered at $x$.  We write $B_r^p(x) \equiv B_r(x) \cap \R^p \times \{0^{n+1-p}\}$, and $S^{p-1} \equiv \del B_1^p$ for the unit sphere in $\R^p \times \{0^{n+p-1}\}$.  Define the dilation/translation operator $\eta_{y, r}(x) = (x - y)/r$.  Given a subset $\Omega \subset \R^{n+1}$, write $\Omega_{y, r} = \eta_{y, r}(\Omega)$.

Given vectors $v, w \in \R^{n+1}$, then $v \cdot w$ denotes the usual Euclidean inner product, and $|v| = \sqrt{v\cdot v}$ is the usual Euclidean length.  Given linear maps $A, B : \R^{n+1} \to \R^{n+1}$, we define the inner product $A\cdot B = \sum_i A(e_i) \cdot B(e_i)$, summed over any choice of orthonormal basis $e_i$ of $\R^{n+1}$, and correspondingly set $|A|^2 = A \cdot A$.  Unless otherwise stated, $e_i$ will denote the standard basis of $\R^{n+1}$.

Given a subspace $V \subset \R^{n+1}$, we let $\pi_V$, $\pi^\perp_V$ be the orthogonal projections to $V$, $V^\perp$ respectively.  Given subpsaces $V$, $W$ (not necessarily of the same dimension), then we set $V \cdot W = \pi_V \cdot \pi_W$.  We remark that if $V$ and $W$ do have the same dimension, then
\[
\frac{1}{2} |\pi_V - \pi_W|^2 = V^\perp \cdot W = V \cdot W^\perp.
\]
If $U \subset p + V$ and $f : U \to V^\perp$, we write $\graph_{p+V}(f) = \{ x + f(x) : x \in U \subset p + V \}$.  For ease of notation, if $p + V = \R^n \times \{0\}$, and $g : U \subset \R^n \times \{0\} \to \R$, then we interpret $\graph_{\R^n\times\{0\}}(g) \equiv \graph_{\R^n\times\{0\}}(g e_{n+1})$.

Every constant written as $c$ or $c_i$ will be $\geq 1$, and every constant $\delta_i$ or $\eps_i$ will be $\leq 1$.  Unindexed constants may change line-to-line.

\subsection{Polyhedral cone domains}

Here we define our model domains.

\begin{definition}\label{def:omegaref}
A \emph{polyhedral cone domain} $\Omega$ in $\R^{n+1}$ is a closed, dilation-invariant domain with non-empty interior, that can be written as the intersection of finitely-many closed half-spaces.
\end{definition}

Take a polyhedral cone domain $\Omega$ as defined above, and write $\Omega = \cap_{j=1}^N H_j$ for $H_j = \{ x : x\cdot \nu_j \leq 0\}$ being closed half-spaces, and $\nu_j$ unit vectors.  Here $\cap_\emptyset$ is understood to be all of $\R^{n+1}$.  If $\Omega$ contains a line $L$, then $L \perp \nu_j$ for each $j$, and hence after a rotation we can write $\Omega = \Omega' \times \R$, where $\Omega'$ is a polyhedral cone domain in $\R^n$.  We can therefore repeat this process and find a maximal integer $m \in \{0, \ldots, n+1\}$ so that we can write
\begin{equation}\label{eqn:omega-decomp}
\Omega = O(\Omega_0^l \times \R^m)
\end{equation}
for some $O \in SO(n+1)$, and $l+m = n+1$.  Equivalently, we have $O(\{0\}\times \R^m) = \cap_j \del H_j$ and $O(\Omega_0 \times \{0\}) = (\cap_j \del H_j)^\perp$.  Under this decomposition $\Omega$ is said to be \emph{$m$-symmetric}.  Note that $\Omega_0$ is $0$-symmetric, in the sense that it contains no lines.

We say $\Omega$ has dihedral angles $\leq \pi/2$ (resp. $=\pi/2$) if: given any pair of half-spaces $H_i, H_j$ such that $\del H_i \cap \del H_j \cap \Omega$ is a relatively open subset of $\del H_i \cap \del H_j$, then we have $\nu_i \cdot \nu_j \leq 0$ (resp. $ = 0$).  We may refer to $\Omega$ having dihedral angles $\leq \pi/2$ as \emph{non-obtuse}.

\begin{example}
If $L : \R^{n+1} \to \R^{n+1}$ is a linear isomorphism, then $L([0, \infty)^l \times \R^{m})$ is a polyhedral cone domain.  In fact \cite[Theorem 1.1]{Coxter1934discrete} implies that any non-obtuse polyhedral cone domain is a simplicial prism, and hence takes the form $L([0, \infty)^l \times\R^m)$ for some $l, m$ and $L$ as above.
\end{example}

\begin{example}
When $l = 0$, $\Omega = \R^{n+1}$.  When $l = 1$, $\Omega$ is a half-space.  When $l = 2$, $\Omega$ is a wedge $W^2 \times \R^{n-1}$, with interior angle $< \pi$.  The convex wedge $W^2 \times \R^{n-1}$ is the archtype of polyhedral cone domain, and $W^2 \times \R^{n-2}$ is the archtype of free-boundary minimal surface in this domain.
\end{example}

Given $x \in \Omega = \cap_{j=1}^N H_j$, there are indices $I \subset \{1, \ldots, N\}$ so that we can write
\begin{equation}\label{eqn:poly-x-decomp}
x \in \bigcap_{j \in I} \del H_j \cap \bigcap_{j \not\in I} \sint H_j.
\end{equation}
If we let $r < d(x, \cup_{i \not\in I} \del H_i) \equiv \inf \{ |x \cdot \nu_i| : i \not \in I\}$, then
\begin{equation}\label{eqn:poly-tangent}
\Omega \cap B_{r}(x) = \cap_{i \in I} H_i \cap B_{r}(x) = \left( x+ \cap_{i \in I} H_i \right) \cap B_{r}(x).
\end{equation}
We therefore define the polyhedral cone domain $T_x\Omega := \cap_{i \in I} H_i$ to be the \emph{tangent domain} of $\Omega$ at $x$.  Alternatively, one could realize $T_x\Omega$ as the limit (in the local Hausdorff sense) of dilations $\lim_{r \to 0} \frac{1}{r} (\Omega - x)$.  Trivially $T_0\Omega = \Omega$.

We define the \emph{$i$-th stratum} $\del_i\Omega$ to be the set of points $x \in \Omega$, such that the tangent domain $T_x\Omega$ is \emph{at most} $i$-symmetric.  Equivalently, if $x$ and $I$ are as in \eqref{eqn:poly-x-decomp}, then $x \in \del_i\Omega \iff \cap_{j \in I} \del H_j$ is at most $i$-dimensional. We have the trivial fibration
\[
\del_0\Omega \subset \ldots \subset \del_n\Omega = \del\Omega \subset \del_{n+1}\Omega = \Omega.
\]
For notational convenience define $\del_{-1}\Omega = \emptyset$.  If $\Omega$ is $m$-symmetric, in which case $\Omega = O(\Omega_0 \times \R^m)$ for some $0$-symmetric $\Omega_0$, then we have $\del_m \Omega = O(\{0\}\times \R^m)$, $\del_i \Omega = \emptyset$ for all $i < m$, and $T_x\Omega = \Omega$ for all $x \in \del_m\Omega$.

We shall use the following quantification of \eqref{eqn:poly-tangent}:
\begin{lemma}\label{lem:Bref}
Given $\Omega$ a polyhedral cone domain, there is a number $B(\Omega) \in (0, 1/4)$ so that given $x \in \del_i\Omega$, we have
\[
B_{2Bd}(x) \cap \Omega = B_{2Bd}(x) \cap (x + T_x\Omega), \quad d = d(x, \del_{i-1}\Omega).
\]
\end{lemma}

\begin{proof}
Write $\Omega = \cap_{j=1}^N H_j$ as before, with $H_j = \{ x : x \cdot \nu_j \leq 0\}$.  Fix an $i$.  Let $\cI_i$ be the set of subcollections $I \subset \{1, \ldots, N\}$ with the property that $V_I := \cap_{j \in I} \del H_j$ is an $i$-dimensional plane, and such that $I$ is ``maximal'' in the sense that $\del H_k \cap V_I$ is $(i-1)$-dimensional for every $k \not\in I$.  Take $I \in \cI_i$.  Let
\[
\alpha_I = \min \{ |\pi_{V_I}(\nu_j)| : j \not\in I\}.
\]
We have $\alpha_I > 0$, as otherwise $\nu_k \perp V_I$ for some $k \not\in I$, contradicting maximality of $I$.

Given $x \in V_I$ such that $d(x, V_I \cap \del H_k) \geq 1$ for every $k \not \in I$, we have
\begin{align*}
\min_{k \not \in I} d(x, \del H_k)
&= \min_{k \not\in I} |x \cdot \nu_k| \\
&\geq \alpha_I \min_{k \not\in I} \frac{|x \cdot \pi_{V_I}(\nu_k)|}{|\pi_{V_I}(\nu_k)|} \\
&= \alpha_I \min_k d(x, V_i\cap \del H_k) \\
&\geq \alpha_I.
\end{align*}
Let $\alpha_i = \min_{I \in \cI_i} \alpha_I$, which is positive since $\cI_i$ is finite.

When $i = 0$ there is nothing to show, likewise if $d(x, \del_{i-1}\Omega) = 0$.  Take $x \in \del_i\Omega \setminus \del_{i-1}\Omega$, and after scaling we can WLOG suppose that $d(x, \del_{i-1}\Omega) \geq 1$.  Write
\begin{equation}\label{eqn:poly-x-decomp2}
x = \bigcap_{x \in I} \del H_j \cap \bigcap_{x \not\in I} \sint H_j
\end{equation}
as in \eqref{eqn:poly-x-decomp}.  By our choice of $x$ we have that $\cap_{j \in I} \del H_j$ is $i$-dimensional.  If $k \not\in I$, then necessarily $\del H_k \cap (\cap_{j \in I} \del H_j)$ is $(i-1)$-dimensional, as otherwise we would have $\del H_k \cap (\cap_{j \in I} \del H_j) = \cap_{j \in I} \del H_j \ni x$, contradicting our decomposition \eqref{eqn:poly-x-decomp2}.  This implies that if $y \in \del H_k \cap (\cap_{j \in I} \del H_j)$, then $y \in \del_{i-1}\Omega$, and hence
\[
\min_{k \not \in I} d(x, V_I \cap \del H_k) \geq d(x, \del_{i-1}\Omega) \geq 1.
\]

Therefore $I \in \cI_i$, and $x \in V_I$ satisfies $d(x, V_I \cap \del H_k) \geq 1$ for every $k\not\in I$.  We deduce by our earlier computations
\[
\min_{k \not\in I} d(x, \del H_k) \geq \alpha_i,
\]
and hence 
\[
B_{\alpha_i}(x) \cap \Omega = B_{\alpha_i}(x) \cap ( \cap_{j \in I} H_j) = B_{\alpha_i}(x) \cap (x + T_x\Omega)
\]
as in \eqref{eqn:poly-tangent}.  Taking $B = \frac{1}{2} \min_{i = 1, \ldots, n+1} \alpha_i$ proves the Lemma.
\end{proof}

Let us define the density of $\Omega$ as
\begin{equation}\label{eqn:theta-omega}
\Theta_\Omega = \omega_{n+1}^{-1} \haus^{n+1}(\Omega \cap B_1) \equiv \omega_l^{-1} \haus^l(\Omega_0 \cap B_1),
\end{equation}
where $\omega_n = \haus^n(B_1(0^n))$ is the $n$-dimensional volume of the unit $n$-ball.  By convexity, we have the monotonicity
\begin{equation}\label{eqn:theta-mono}
\Theta_{T_x\Omega} \geq \frac{\haus^{n+1}(\Omega \cap B_r(x))}{\omega_{n+1} r^{n+1}} \quad \forall r > 0,
\end{equation}
and hence we have the following lower-semi-continuity: if $x_i \to x$, then 
\begin{equation}\label{eqn:theta-lsc}
\liminf_i \Theta_{T_{x_i}\Omega} \geq \Theta_{T_x\Omega}.
\end{equation}

Let $\cN_\Omega$ be the cone consisting of \emph{outer normals} for $\Omega$:
\[
\cN_\Omega = \cup \{ \nu \in \R^{n+1} : \Omega \subset \{ y : y \cdot \nu \leq 0 \} \}.
\]
By convexity, $\cN_\Omega \neq \emptyset$.

If we decompose $\Omega = O(\Omega_0^l \times \R^m)$ as in \eqref{eqn:omega-decomp}, then define $\cP_\Omega$ to be the collection of \emph{horizontal $n$-planes} of the form $O(\R^l \times W^{m-1})$ for $W^{m-1}$ an $(m-1)$-plane in $\R^m$.  We observe the trivial inclusion
\[
\cP_{\Omega} \subset \cP_{T_x\Omega} \quad \forall x \in \Omega.
\]
In particular, if $x \in \del_i\Omega$ and $y \in \Omega \cap B_{2B d(x, \del_{i-1}\Omega)}(x)$, then
\begin{equation}\label{eqn:pv-inclusion}
\cP_{T_x\Omega} \subset \cP_{T_y\Omega}.
\end{equation}

\subsection{Curved polyhedral cones}

Since a general polyhedral domain will at finite scales only look like the perturbation of a polyhedral cone domain, in our local regularity we must allow for domains with a little bit of curvature (captured by the map $\Phi$).  Moreover, since the model polyhedral cone will in general change as one moves along any given stratum, we allow for small changes in the model domain itself (captured by $\Psi$).
\begin{definition}\label{def:omega}
Let $\OmegaRef$ be a polyhedral cone domain in $\R^{n+1}$, as per Definition \ref{def:omegaref}.  Given $\eps \in (0, 1)$, we define $\cD_\eps(\OmegaRef)$ as the set of domains $\Omega$ satisfying
\[
\Omega \cap B_1 = \Phi(\Psi(\OmegaRef)) \cap B_1,
\]
where $\Psi$ is a linear isomorphism satisfying $|\Psi - \mathrm{Id}| \leq \eps$, and $\Phi : B_2 \to \R^{n+1}$ is a $C^2$ diffeomorphism satisfying
\begin{equation}\label{eqn:def-phi}
\Phi(0) = 0, \quad D\Phi|_0 = \mathrm{Id}, \quad |\Phi - \mathrm{id}|_{C^2(B_2)} \leq \eps.
\end{equation}
\end{definition}

\begin{remark}
Since $\eps < 1$, we have $\Phi(B_2) \supset B_1$.  In fact we could equivalently have asked for $\Phi$ to be a diffeomorphism $B_1 \to B_1$, but we shall see our definition is slightly more convenient to work with.
\end{remark}

Take $\Omega = \Phi(\Psi(\OmegaRef)) \cap B_1$ as in Definition \ref{def:omega}, and given $x \in \Omega$ let us write $x = \Phi(\Psi(z))$.  We define $\del_i\Omega = \Phi(\Psi(\del_i\OmegaRef)) \cap B_1$, so that $\del_i\Omega$ consists of the points near which $\Omega$ is diffeomorphic to some polyhedral cone that is at most $i$-symmetric.

There is a well-defined polyhedral cone domain $T_x\Omega = \lim_{r \to 0} \frac{1}{r} (\Omega - x)$ which we will call the tangent domain, and in fact we can write $T_x\Omega = (D\Phi|_{\Psi(z)} \circ \Psi)T_z\OmegaRef$.  It follows by scaling that if $\Omega \in \cD_\eps(\OmegaRef)$, then $\frac{1}{R} \Omega \cap B_1 \in \cD_\eps(\OmegaRef)$ for every $R \geq 1$, and hence $T_0\Omega \in \cD_\eps(\OmegaRef)$ also.

Since $T_x\Omega$ is a polyhedral cone domain, we can define density $\Theta_{T_x\Omega}$ and cone of outer normals $\cN_{T_x\Omega}$ as before.   We say a vector field $X$ is \emph{tangential to $\Omega$} if $X(x) \cdot V = 0$ for all $V \in \cN_{T_x\Omega}$, and for all $x \in \Omega$.  Similar to \eqref{eqn:theta-lsc}, $\Theta$ obeys the following lower-semi-continuity:
\begin{lemma}
Suppose $x_i \to x \in B_1$, and $\Omega_i \in \cD_{\delta_i}(\OmegaRef)$ for some $\delta_i \to 0$.  Then we have
\begin{equation}\label{eqn:theta-lsc-omega}
\liminf_i \Theta_{T_{x_i}\Omega_i} \geq \Theta_{T_x\OmegaRef}.
\end{equation}
\end{lemma}

\begin{proof}
Write $\Omega_i = \Phi_i(\Psi_i(\OmegaRef)) \cap B_1$ as per Definition \ref{def:omega}, and let $z_i = \Psi_i^{-1}(\Phi_i^{-1}(x_i))$.  We have $T_{x_i}\Omega_i = (D\Phi_i|_{\Psi_i(z_i)} \circ \Psi_i)(T_{z_i}\OmegaRef)$, and therefore
\begin{align*}
\Theta_{T_{x_i}\Omega}
&\geq \omega_{n+1}^{-1}(1-c(n)\delta_i) \haus^{n+1}(T_{z_i}\OmegaRef \cap B_{1-c(n)\delta_i}) \\
&\geq (1-c(n)\delta_i) \Theta_{T_{z_i}\OmegaRef}.
\end{align*}
\eqref{eqn:theta-lsc-omega} then follows from the lower-semi-continuity \eqref{eqn:theta-lsc}.
\end{proof}

\subsection{Varifolds}

Our notion of weak surface will be a varifold.  For a more detailed background see \cite{simon:gmt} or \cite{All}.  Recall that an integral $n$-varifold in an open set $U$ is a Radon measure $M$ on $U \times Gr(n, n+1)$ of the form 
\[
M(\phi(x, S)) = \int_{\tilde M} \phi(x, T_xM) \theta(x) d\haus^n
\]
for some countably $n$-rectifiable set $\tilde M$, and some non-negative Borel measurable function $\theta : \tilde M \to \Z$.  Here $Gr(n, n+1)$ denotes the set of unoriented $n$-planes in $\R^{n+1}$.  We write $\cIV_n(U)$ for the space of integral $n$-varifolds in $U$.  If $S$ is an $n$-dimensional, $C^1$ submanifold of $\R^{n+1}$, we write $[S]$ for the obvious varifold induced by $S$.

Given $M \in \cIV_n(U)$, the mass measure $\mu_M = \pi_\sharp M$ is the pushforward under the projection $\pi : U \times Gr \to U$, so that $\mu_M = \haus^n \llcorner \theta \llcorner \tilde M$.  Given a $C^1_c(U)$ vector field $X$, generating a $1$-parameter family of diffeomorphisms $\phi_t : U \to \R^{n+1}$, the first-variation of $M$ along $X$ is the derivative
\[
\delta M(X) := \frac{d}{dt}|_{t = 0} (\phi_t)_\sharp M = \int \mathrm{div}_M(X) d\mu_M,
\]
where $div_M(X) = \sum_i e_i \cdot D_{e_i} X$, for any choice of ON basis $\{e_i\}_i$ of $T_x\tilde M$.  Relatedly, given a function $h \in C^1_c(U)$, we write $\nabla h = \pi_{T_x \tilde M}(Dh)$ for the tangential derivative of a function $h$ along $M$.

$M$ is said to have locally finite first variation if $\delta M$ is a bounded operator on every $W \subset \subset U$.  In this case we may decompose
\[
\delta M(X) = -\int X \cdot H d\mu_M + \int X \cdot \eta d\sigma
\]
where $H$ is the generalized mean curvature of $M$, $\sigma$ the generalized boundary measure, and $\eta$ the generalized boundary conormal.

We define the density ratio of $M$ in a ball $B_r(x) \subset U$ as
\[
\theta_M(x, r) := \frac{\mu_M(B_r(x))}{\omega_n r^n}.
\]
If $M$ has locally finite first variation, has zero generalized boundary, and $||H_M||_{L^\infty(U; \mu_M)} \leq \Lambda$, then $e^{\Lambda r}\theta_M(x, r)$ is increasing in $0 < r < d(x, \del U)$ \cite{All}.  In particular, the density at a point
\[
\theta_M(x) := \lim_{r \to 0} \theta_M(x, r)
\]
is a well-defined upper-semi-continuous function, satisfying $\theta_M(x) \geq 1$.

Given a $C^1$ domain-with-corners $\Omega \subset \R^{n+1}$, we define the set of \emph{integral varifolds in $B_1$ with free-boundary in $\Omega$}, denoted $\cIH_n(\Omega, B_1)$, to be the set $M \in \cIV_n(B_1)$ satisfying the conditions that $M = M \llcorner \pi^{-1}(\mathrm{int}\Omega)$, and 
\[
\delta M(X) = -\int X \cdot H^{tan}_M d\mu_M
\]
for all $X \in C^1_c(B_1)$ tangential to $\Omega$, for some $H_M^{tan} \in L^1_{loc}(B_1, \R^{n+1}; \mu_M)$.  In other words, $M$ has mean curvature but no boundary ``tangential'' to $\del\Omega$ (made precise in Theorem \ref{thm:first-var}).

\subsection{Excess}

Our mechanism to establish regularity is the decay of an appropriate excess quantity.  Given $\Omega = \Phi(\Psi(\OmegaRef)) \cap B_1$ as in Definition \ref{def:omega}, and an $n$-plane $V$, the full $L^2$ excess is:
\begin{align*}
E_\delta(\Phi, M, p + V, x, r) 
&= \max \left\{ r^{-n-2} \int_{B_r(x)} d_{p + V}^2 d\mu_M , \right. \\
&\quad \left. \delta^{-1} r^2 ||H_M||^2_{L^\infty(B_1; \mu_M)} + \delta^{-1} r^2 |D^2 \Phi|_{C^0(B_{2r}(x))} \right\}.
\end{align*}
We will often abbreviate
\[
E(\Phi, M, p + V, x, r) = E_1(\Phi, M, p + V, x, r),
\]
and may also omit the $\Phi$ or $M$ when there is no ambiguity.  At scales for which $\Omega$ resembles a fixed model cone $\OmegaRef$, we will prove decay on $E_\delta$ by reducing the problem to a decay estimate of the linearized problem.

When traversing cone types, we will find it convenient to work with the ``total'' excess:
\begin{align*}
E^\infty(M, p + V, x, r) &= r^{-2} \sup_{z \in \spt M \cap B_r(x)} d_{p + V}(z)^2 \\
&\equiv r^{-2} \sup_{z \in \spt M \cap B_r(x)} |\pi_V^\perp(z - p)|^2 \\
E^W(M, V, x, r) &= r^{-n} \int_{B_r(x)} |\pi_{T_z M} - \pi_V|^2 d\mu_M(z) \\
E^{tot}(\Phi, M, p + V, x, r) &= \max\{ E^\infty(M, p + V, x, r),  E^W(M, V, x, r), E(\Phi, M, p + V, x, r)\}.
\end{align*}

Implicit in the definition of excess is the requirement that $\Phi(0) = 0$, $D\Phi|_0 = \mathrm{Id}$, so even though $E$ is formally scale-invariant one must be a little careful: in a general ball $B_r(x) \cap \Omega$ will not look like a cone, and even when it does, if $x \neq 0$ then $\Phi$ will no longer be the right map.  This is made precise in the following section.


\subsection{Changing cone type} \label{sec:changing-type}

The key fact that we will use in our regularity theorem is that at any point, and at any appropriately small scale, $\Omega$ looks like one of finitely-many polyhedral cone domains.
\begin{lemma}\label{lem:B}
Let $\OmegaRef$ be a polyhedral cone in $\R^{n+1}$.  There is a finite set of polyhedral cone domains $\cT(\OmegaRef) := \{ T_z\OmegaRef : z \in \OmegaRef \}$, and constants $B(\OmegaRef) \in (0, 1/4)$, $\eps_B(n)$, $c_B(n)$, so that the following holds.

Given any $\Omega = \Phi(\Psi(\OmegaRef)) \cap B_1 \in \cD_\eps(\OmegaRef)$, for $\eps \leq \eps_B$, and given $x = \Phi(\Psi(z)) \in \del_i\Omega$, take $r \leq \min\{ B d(x, \del_{i-1}\Omega), 1-|x| \}$, and define $\Omega_{x, r} = \frac{1}{r}(\Omega - x)$.  Then the following holds:
\begin{enumerate}
\item \label{item:poly1} We have
\begin{equation}\label{eqn:poly1}
(1-c_B\eps) d(x, \del_{i-1}\Omega) \leq d(z, \del_{i-1}\OmegaRef) \leq (1+c_B\eps) d(x, \del_{i-1}\Omega),
\end{equation}

\item \label{item:poly2} Given $y \in \Omega \cap B_r(x)$, then for every $V \in \cP_{T_x\Omega}$ there is a $W \in \cP_{T_y\Omega}$ such that $|\pi_V - \pi_W| \leq c(n) |D^2 \Phi|_{C^0(B_2)} |x - y|$.

\item \label{item:poly3} There is a $T_z\OmegaRef \in \cT$, a linear isomorphism $\beta : \R^{n+1} \to \R^{n+1}$, a $C^2$ diffeomorphism $\alpha : B_2 \to \R^{n+1}$, satisfying
\begin{equation}\label{eqn:poly3-1}
|\beta - \mathrm{Id}| \leq c_B \eps,
\end{equation}
and $\alpha(0) = 0$, $D\alpha|_0 = \mathrm{Id}$, and 
\begin{equation}\label{eqn:poly3-2}
(1-c_B\eps)|D^2\alpha|_{C^0(B_{2})} \leq r |D^2\Phi|_{C^0(B_{2r}(x))} \leq (1+c_B\eps)|D^2 \alpha|_{C^0(B_{2})}, 
\end{equation}
so that
\begin{equation}\label{eqn:poly3-3}
\Omega_{x,r} \cap B_1 = \alpha(\beta(T_z\OmegaRef)) \cap B_1,
\end{equation}
In particular, $\Omega_{x, r} \in \cD_{c_B\eps}(T_z\OmegaRef)$.

\item \label{item:poly4} In the notation of part \ref{item:poly3},
\begin{align}
\frac{1}{2} E(\Phi, M, p + V, x, r) 
&\leq E(\alpha, (\eta_{x, r})_\sharp M, \eta_{x,r}(p) + V, 0, 1) \label{eqn:poly4} \\
&\leq 2 E(\Phi, M, p + V, x, r), \nonumber
\end{align}
and the same with $E^{tot}$ in place of $E$.
\end{enumerate}
\end{lemma}

\begin{remark}\label{rem:scaling-E}
If $x = 0$, then for any $r \leq 1$ we have $r^{-1}\Omega \cap B_1 = \Phi_{0, r}(\Psi(\OmegaRef)) \cap B_1$, where $\Phi_{0, r}(y) = r^{-1}\Phi(r y)$, and therefore we have the exact scaling
\[
E(\Phi, M, p + V, 0, r) = E(\Phi_{0,r}, (\eta_{0, r})_\sharp M, r^{-1}p + V, 0, 1).
\]
\end{remark}

\begin{proof}
For $\eps_B(n)$ sufficiently small, $\Phi \circ \Psi$ is a $(1+c(n)\eps)$-bi-Lipschitz equivalence:
\[
|\Phi(\Psi(y)) - \Phi(\Psi(y')) - (y - y')| \leq c(n) \eps|y - y'|,
\]
from which \ref{item:poly1} follows directly.  In particular note that $\Psi^{-1}\Phi^{-1} B_r(x) \subset B_{2B d(z, \del_{i-1}\OmegaRef)}(z)$.

We prove \ref{item:poly2}.  Observe that for a fixed subspace $U^n$, and linear isomorphisms $A, B : \R^{n+1} \to \R^{n+1}$, the map $A \mapsto \pi_{A(U)}$ is well-defined and analytic in $A$, and satisfies
\begin{equation}\label{eqn:B-map}
|\pi_{A(U)} - \pi_{B(U)}| \leq c(n) |A - B|
\end{equation}
for $A, B$ satisfying $|A - Id| + |B - Id| < \eps(n)$.  Write $y = \Phi(\Psi(w))$.  From \ref{item:poly1} and the inclusion \eqref{eqn:pv-inclusion}, we have
\[
\cP_{T_z\OmegaRef} \subset \cP_{T_w\OmegaRef},
\]
and therefore we can take $W = D\Phi|_w \circ D\Phi^{-1}|_x V$.  We then estimate, provided $\eps_B(n)$ is sufficiently small, 
\[
|\pi_V - \pi_W| \leq c(n)\left| D\Phi|_w \circ D\Phi^{-1}|_x - Id\right| \leq c(n) \left|D\Phi^{-1}|_x - D\Phi^{-1}|_y\right| \leq c(n) |x - y| |D^2\Phi|_{C^0(B_1)}.
\]
This proves \ref{item:poly2}.

We prove \ref{item:poly3}, \ref{item:poly4}.  Let $\cT = \{T_z \OmegaRef\}_{z \in \OmegaRef}$, and $B$ be as in Lemma \ref{lem:Bref}.  Let us define
\[
\beta = D\Phi|_{\Psi(z)} \circ \Psi, \quad \alpha(y) = \frac{1}{Br}( \Phi(x + r D\Phi^{-1}|_x y) - x).
\]
The bound \eqref{eqn:poly3-1} and $\alpha(0) = 0$, $D\alpha|_0 = \mathrm{Id}$ follow trivially.  The Hessian bound in \eqref{eqn:poly3-2} follows from our definition of $\alpha$ and the estimate $(1-c(n)\eps) |v| \leq |D\Phi|_y v| \leq (1 + c(n)\eps) |v|$.  Ensuring $\eps_B(n)$ is small, we have 
\[
r \leq 2 B d(z, \del_{i-1}\OmegaRef),
\]
and therefore
\begin{align*}
(\alpha \circ \beta)( B_2 \cap T_x\OmegaRef) \cap B_1 
&= (\alpha \circ \beta)( B_2 \cap \OmegaRef_{x, r}) \cap B_1 \\
&= \Omega_{x, r} \cap B_1.
\end{align*}
This proves \ref{item:poly3}.  Lastly, \ref{item:poly4} follows by \eqref{eqn:poly3-2} and scaling, ensuring $\eps_B(n)$ is sufficiently small.
\end{proof}

\section{Main theorem}\label{sec:main}

Our main Theorem \ref{thm:main} is the following Allard-type regularity result, which says loosely that whenever an integral varifold $M$ has free-boundary in $\Omega = \Phi(\OmegaRef = \OmegaRefN \times \R)$, and: $\Phi$ is sufficiently close to the identity, $M$ has sufficiently small mean-curvature, and $M$ is sufficiently varifold close to the ``horizontal'' plane $\R^n \times \{0\}$, then $\spt M$ is a $C^{1,\alpha}$ perturbation of $\OmegaRefN$.

In general $\Omega$ is curved, and so $\spt M$ will not be graphical over a particularly ``nice'' subdomain of $\R^n$.  Instead, it is more convenient and precise to look at $\Phi^{-1}(\spt M)$, which will be a graph over $\R^n \cap \OmegaRef \cap B_{1/32}$.  We do not lose anything in our estimates by doing this, as even before this transformation we must use $|D^2\Phi|_{C^0}$ to control the tilting of tangent planes of $\spt M$.

For various reasons we in fact want to allow not only the diffeomorphism $\Phi$ to change, but also the reference domain $\OmegaRef$ (mainly because different points in $\Phi(\OmegaRef)$ will be modeled on different polyhedral cone domains, even when staying in the same stratum; see Section \ref{sec:changing-type}).  For this reason we in fact consider domains of the form $\Phi(\Psi(\OmegaRef))$, where $\Phi$ is a diffeomorphism close to $Id$, and $\Psi$ is a linear map close to $Id$.  Our constants $\delta, c, \alpha$ will be uniform in $\Phi, \Psi$, but in each particular case our ``reference'' polyhedral domain will be $\Psi(\OmegaRef)$.  It may be easier to parse Theorem \ref{thm:main} by considering the case when $\Psi = Id$, $p = 0$, $V = \R^n$, in which case $q = 0$, $W = \R^n$, $T_0\Omega = \OmegaRef$.

\begin{theorem}[Allard-type regularity]\label{thm:main}
Let $\OmegaRef$ be a polyhedral cone domain.  There are constants $\delta(\OmegaRef)$, $c(\OmegaRef)$, $\alpha(\OmegaRef)$ so that the following holds.  Let $\Omega = \Phi(\Psi(\OmegaRef)) \cap B_1 \in \cD_\delta(\OmegaRef)$, and take $M \in \cIH(\Omega, B_1)$.  Assume there is a $V \in \cP_{T_0\Omega}$, $p \in V^\perp$, so that
\begin{equation}\label{eqn:main-hyp1}
E := \int_{B_1} d_{p + V}^2 d\mu_M + ||H||_{L^\infty(B_1; \mu_M)}^2 + |D^2\Phi|_{C^0(B_2)}^2 \leq \delta^2
\end{equation}
and
\begin{equation}\label{eqn:main-hyp2}
\theta_M(0, 1) \leq (3/2) \Theta_{T_0\Omega}, \quad \spt M \cap B_{1/512} \neq \emptyset.
\end{equation}
Then if we set $q = \Phi^{-1}(p)$, $W = D\Phi^{-1}|_p V$, we can find a function $f : (q + W) \cap\Psi(\OmegaRef) \cap B_{1/128}(q) \to W^\perp$ satisfying
\begin{equation}\label{eqn:main-concl1}
|f|_{C^{1,\alpha}} \leq c E^{1/2}, 
\end{equation}
so that
\begin{equation}\label{eqn:main-concl2}
\Phi^{-1}(\spt M) \cap B_{1/256} \subset \Phi^{-1}(\spt M) \cap B_{1/128}(q) = \graph_{q + W}(f) \cap B_{1/128}(q).
\end{equation}
\end{theorem}

Some comments are in order.
\begin{remark}
Even though $\OmegaRef$ is convex, $\Omega$ need not be.
\end{remark}

\begin{remark}
$\alpha$ can in fact be chosen to be anything in some interval $(0, e - 1)$ (where $e = e(\OmegaRef) > 1$ as in Proposition \ref{prop:expansion} is determined by the Neumann eigenvalue expansion of $\Omega_0^l \cap \del B_1$), provided $\delta$ and $c$ are taken to depend on $\alpha$.  For example, when $\OmegaRef$ is $\R^{n+1}$ or a half-plane, then any $\alpha \in (0, 1)$ is admissible.  When $\OmegaRef = W^2 \times \R^{n-1}$ is a wedge with angle $\gamma < \pi$, then we can take $\alpha \in (0, \min\{1, \pi/\gamma - 1\})$.  See Remark \ref{rem:expansion} for more details.
\end{remark}

\begin{remark}
We state and prove Theorem \ref{thm:main} in codimension-one Euclidean space.  However when $l = 1$ our proof carries over verbatim to higher codimension and ambient manifolds, giving an alternate proof of \cite{GrJo}.  When $l \geq 2$, the proof carries over except for two estimates in Section \ref{sec:first-var}, which continue to hold if one knows a priori that $\spt M$ is contained in some $(n+1)$-dimensional submanifold.  See Section \ref{sec:codim} for details.
\end{remark}

\begin{remark}
If one assumes $\theta_M(0, 1) \leq (1+\delta) \Theta_{T_0\Omega}$ then Theorem \ref{thm:main} holds for varifolds which are only rectifiable, but have a lower density bound $\theta_M \geq 1$ $\mu_M$-a.e.  This requires only minor modifications of the proof. (Specifically, in the contradiction arguments of Proposition \ref{prop:graph} and Lemma \ref{lem:sharp-mass-bound}, the choice of constant $\delta_1$ in Corollary \ref{cor:decay}, and the choice of constants in the induction argument of of Theorem \ref{thm:total-decay}.).  In a similar vein, Theorem \ref{thm:main} also holds if we assume $H_M \in L^p(\mu_M)$, for $p > n$ instead of $p = \infty$.  In this case our constants would depend on $p$ also.
\end{remark}

Our regularity Theorem \ref{thm:main} requires $V$ to be ``horizontal,'' in the sense that if $\Psi(\OmegaRef) = \OmegaRefN \times \R$, then $V \supset \OmegaRefN$.  When $V$ is instead ``vertical,'' in the sense that $V \supset \{0\} \times \R$, then regularity as in Theorem \ref{thm:main} can fail.  Below is an example illustrating this.
\begin{example}\label{ex:vertical-plane}
Let $W \subset \R^2$ be the wedge $\{ r e^{i\theta} : r \in [0, \infty), -\pi/6 \leq \theta \leq \pi/6 \}$, and let $\bY \subset \R^2$ be the cone over $\{ 1, e^{i 2\pi/3}, e^{-i 2\pi/3}\}$, consisting of three rays meeting at the origin at $120^0$.  Then for every $\eps > 0$, the integral varifold $M_\eps$ given by integrating over $(\eps + \bY ) \cap W$ is stationary with free-boundary in $W$.  As $\eps \to 0$, then $M_\eps$ converges as varifolds to the ``vertical'' plane $P = \{ x \geq 0, y = 0 \}$, but none of the $M_\eps$ are $C^1$ perturbations of $P$.

One can construct a similar example by restricted the tetrahedal cone to a $0$-symmetric domain in $\R^3$ consisting of the intersection of $4$ half-spces.
\end{example}
Example \ref{ex:vertical-plane} is a little contrived, but we expect one should be able to construct smooth counterexamples.  However, we would not expect these examples to be minimizing (in the sense of integral currents).  Relatedly, when $\OmegaRef$ is a wedge, or has dihedral angles $\leq \pi/2$, then the vertical planes are not minimizing (Lemma \ref{lem:minz-planes}).  For more general convex $\OmegaRef$, or when $\OmegaRef$ is a non-convex wedge, this may fail (see Examples \ref{ex:non-convex}, \ref{ex:possible-minz-plane}).

When $\Omega$ is non-convex, Theorem \ref{thm:main} can fail also, even when $M$ is in some sense minimizing.
\begin{example}\label{ex:non-convex}
\cite[Theorem 1]{HildebrandtSauvignyminimalIII} implies the following: Suppose $\Omega = W^2 \times \R$, where $W^2$ is a wedge with angle $> \pi$, and $\Gamma$ is any smooth curve in $\sint\Omega \cap \{ x^2 + y^2 = 1\}$, such that $\Gamma$ meets $\del\Omega$ only at its endpoints $\{p_1, p_2\}$, and the height function $x_3|_\Gamma$ has no maxima away from the endpoints.  Let $B_+ = \{ (x, y) \in \R^2 : x^2 + y^2 < 1, y > 0 \}$.  Then there is a map $F \in C^0(\overline{B_+}, \Omega) \cap W^{1,2}(\overline{B_+}, \Omega)$ such that
\begin{enumerate}
	\item $F$ minimizes the Dirichlet energy, and $F(B_+)$ is a smooth minimal surface;
	\item $F$ maps the semi-circle $\del B_+ \cap \{y > 0\}$ monotonically to $\Gamma$, $\pi_{\R^2 \times \{0\}} \circ F$ maps the interval $[-1, 1]$ monotonically into to $\del W$, and $\pi_{\R^2} \circ F$ maps $B_+$ diffeomorphically to $\sint\Omega \cap B_1$;
	\item writing $F(\overline{B_+}) \cap \{0\}\times \R = \{q\}$, then $F^{-1}(q) = [a_1, a_2]$ is an interval of positive length, and $F$ extends smoothly to $\overline{B_+} \setminus \{ -1, 1, a_1, a_2\}$;
	\item on $(-1, 1)\setminus [a_1, a_2]$, $F$ meets $\del\Omega$ orthogonally;
	\item on $(a_1, a_2)$, the unit normal of $F$ is horizontal.
\end{enumerate}

These items imply that $M = [F(\overline{B_+})] \in \cIVT_2(\Omega, \R^3\setminus \Gamma)$, and has zero mean curvature.  By choosing $\Gamma$ to be contained in a very thin slab $\R^2\times [-\eps, \eps]$, the maximum principle implies $F(\overline{B_+})$ is contained in this slab also.  Therefore we can arrange so that $M \llcorner B_{1/2}$ is arbitrarily varifold close to the multiplicity-one horiztonal plane $[\R^2] \llcorner B_{1/2}$, but $\spt M$ will never be graphical over $\R^2$ at $0$.
\end{example}

\subsection{Outline of proof}

Our strategy to prove Theorem \ref{thm:main} is to show the following excess decay estimate: for all $x \in \spt M \cap B_{1/16}$ there is a plane $V_x \in \cP_{T_x\Omega}$ so that
\begin{equation}\label{eqn:outline-1}
\sup_{z \in B_r(x) \cap \spt M} r^{-2} d(z, x + V_x)^2 \leq c(\OmegaRef) r^{2\alpha} E \quad \forall 0 < r < 1/4,
\end{equation}
where $\alpha(\OmegaRef) \in (0, 1)$.  From \eqref{eqn:outline-1} it follows easily that $\spt M \cap B_{1/32}$, and hence $\Phi^{-1}(\spt M) \cap B_{1/64}$, is a $C^{1,\alpha}$ graph with norm controlled by $c E^{1/2}$.

We prove \eqref{eqn:outline-1} in two steps.  In step one (Section \ref{sec:decay}), given any fixed model cone $\OmegaRef$, we prove a decay like \eqref{eqn:outline-1} with $x = 0$.  Loosely speaking we show that $\spt M \cap B_1$ resembles a $W^{1,2}$ harmonic function in $\OmegaRef$ with Neumann boundary conditions.  By understanding these linear solutions (by a Fourier expansion and an eigenvalue estimate), we can prove $C^{1,\alpha}$ decay.  This basic idea goes back to DeGiorgi, who proved interior regularity by a similar ``excess decay'' strategy, and is implemented in a fashion closer to our style in \cite{All}, \cite{Simon1}, \cite{EdSp}.

In step two (Section \ref{sec:reg}), we exploit the polyhedral structure of $\OmegaRef$ to show that for every $x \in B_{1/4}$, and $0 < r < 1/4$, we can find radii $r = r_0^+ \geq r_0^- \geq r_1^+ \geq r_1^- \ldots \geq r_{n+1}^+ \geq r_{n+1}^- = 0$, such that when $s \in [r_i^-, r_i^+]$ then $\Omega \cap B_s(x)$ is modelled on some $T_{z_i}\OmegaRef$, and $r_i^- / r_{i-1}^+ \geq 1/c(\OmegaRef)$.  In other words, outside of finitely many scales (controlled only by $\OmegaRef)$, $\Omega \cap B_s(x)$ is modelled on some polyhedral cone of the form $T_z\OmegaRef$.  Since there are only finitely many tangents $T_z\OmegaRef$, we can therefore inductively prove decay by our first step in each interval $[r_i^-, r_i^+]$, and extend decay from one interval to the next by enlarging our constant $c$ in \eqref{eqn:outline-1} by a controlled amount.

The key technique hurdle in proving both steps is to show that $M$ has controlled first variation $\delta M$.  Our hypotheses imply $\delta M$ is only bounded in directions tangential to $\Omega$, but we need to establish both that $\delta M$ is bounded in all directions, and an a priori trace-like estimate for $||\delta M||$ in terms of $||M||$ (Theorem\ref{thm:first-var}).  This is the main point where we use convexity of $\OmegaRef$.

A priori control on $\delta M$ gives good compactness for sequences of such $M$, and allows us to prove a sharp $L^\infty-L^2$ estimate, and a uniform lower density bound.  The $L^\infty$ estimate is crucial in Step 1, to show that $M$ doesn't ``concentrate'' near $\del\Omega$ at the scale of excess (and therefore can be well approximated by the interior ``graphical'' region).  The lower density bound means that in various (blow-up) limits the varifolds do not disappear, as they might in the non-convex case.

We elaborate on these steps below.  For simplicity, in our outline we will assume $\Omega = \OmegaRef$, and $H = 0$.

\textbf{Step 1: Decay towards a single cone model.}  We wish to prove an excess decay of the following type: if $M$ is sufficiently varifold close in $B_1$ to a horizontal plane $(p + V) \cap \Omega$, then there is a new horizontal plane $p' + V'$ and a $\theta(\Omega)$ so that
\begin{equation}\label{eqn:outline-2}
E(M, p' + V', 0, \theta) \leq \frac{1}{2} E(M, p + V, 0, 1).
\end{equation}
Here $E(M, p + V, x, r) = r^{-n-2} \int_{B_r(x)} d_{p + V}(z)^2 d\mu_M(z)$.  (Proposition \ref{prop:decay} is phrased in terms of the $L^2$ excess, which is a little more convenient, but in this setting the $L^\infty$ and $L^2$ excesses are the same; see \ref{cor:height-w12-bound}).  By iterating \eqref{eqn:outline-2}, we obtain an estimate like
\[
E(M, p'' + V'', 0, r) \leq c(\Omega) r^{2\alpha} E^\infty(M, p + V, 0, 1) \quad \forall d(\spt M, 0) \leq r \leq 1.
\]
This is our main decay estimate in Step 1.

We prove \eqref{eqn:outline-2} by contradiction.  We assume there is a sequence of $M_i$ such that $M_i \to [(p + V) \cap \Omega]$ in $B_1$ as varifolds, such that
\[
\inf_{p' + V'} E(M_i, p' + V', 0, \theta) \geq \frac{1}{2} E(M_i, p + V, 0, 1/2) =: E_i \quad ( \to 0 \text{ as } i \to \infty).
\]
On larger and larger sets $U_i \subset\subset \sint\Omega \cap (p + V)$ (for $U_i \to \sint\Omega \cap (p + V)$) we can write $\spt M_i \cap U_i = \graph_{p + V}(u_i)$, where $u_i$ satisfy the minimal surface equation.  By setting $v_i = E_i^{-1/2} u_i$, then the $v_i$ are locally bounded in $L^\infty$, and after passing to a subsequence we get convergence
\[
v_i \to v,
\]
where the \emph{Jacobi field} $v$ satisfies
\[
\Delta v  = 0, \quad \del_n v = 0, \quad \int_{B_1(0) \cap (p + V) \cap \Omega} v^2 \leq 1.
\]

We now prove two key technical facts: First, by the sharp $L^\infty$ bound \eqref{eqn:sharp-height}, we have strong convergence
\[
E_i^{-1} E(M_i, p + V, 0, r) \to r^{-n-2}\int_{B_r(0)\cap (p + V) \cap \Omega} v^2 \quad \forall 0 < r < 1/4.
\]
Second, the $W^{1,2}$ estimate of \eqref{eqn:W12-bound}, a sharp eigenvalue estimate for convex domains in the sphere (Theorem \ref{thm:evalue}), and a standard Fourier-type decomposition, we can expand $v$ as above like
\[
v = a + b \cdot x + O(r^{1+\alpha}),
\]
where $b$ lies in some direction of translational symmetry of $\Omega$.  In other words, the eigenvalue bound of Theorem \ref{thm:evalue} implies that any free-boundary plane is integrable through rotations.

We can now repeat the blow-up argument with $p_i + V_i = \graph_{p + V}( E_i a + E_i b\cdot x)$ in place of $p + V$, and obtain the Jacobi field 
\[
v' = v - a - b\cdot x = O(r^{1+\alpha}).
\]
Hence for $i >> 1$ we have
\[
E_i^{-1} E(M_i, p_i + V_i, 0, r) \leq 2 r^{-n-2} \int_{B_r(0) \cap (p + V) \cap\Omega} v^2 \leq c(\Omega) r^{2\alpha},
\]
which is a contradiction for $r(\Omega)$ small.

This general strategy is very robust, and has been implemented in many other contexts.  However in any given situation there are typically two key technical issues to address: strong convergence in norm of the non-linear problem to the linear problem; and decay of the linear problem.  We handle these in our situation by our sharp $L^\infty$ bound and our sharp eigenvalue estimate.

\textbf{Step 2: Decay across cone models.} We prove the general decay \eqref{eqn:outline-1} by using Step 1 and an inductive argument on the strata.  We first observe (Lemma \ref{lem:B}) that every point in $\Omega$ is locally modelled on some (other) polyhedral cone.  Precisely, there is a $B(\Omega)$ so that if $x \in \del_i\Omega$, and $r \leq d(x,\del_{i-1} \Omega)$, then
\[
B_{B r}(x) \cap \Omega = B_{B r}(x) \cap (x + T_x\Omega).
\]
(When $\Omega \in \cD_\delta(\OmegaRef)$ is only a perturbation of a polyhedral cone a similar statement holds.)

Now given $x \in \del_i\Omega$, we choose points $\tilde x_j \in \del_{i_j} \Omega$ and radii
\[
1 = r^-_{J+1} \geq r^+_J \geq r^-_J \geq \ldots \geq r_0^+ \geq r_0^- = 0,
\]
so that $\Omega \cap B_r(\tilde x_j)$ is modelled on some fixed polyhedral cone $T_{z_j}\OmegaRef$, when $r \in [r_j^-, r_j^+]$, and $r_j^-/r_{j-1}^+ \geq 1/c(\OmegaRef)$.  The degree of symmetry $i_j$ is strictly decreasing as $j$ increases, so $J \leq n$.  We will apply Step 1 to get decay $r_j^+ \to r_j^-$, and then give up a controlled number of scales to ``decay'' $r_j^- \to r_{j-1}^+$.

Some care must be taken in construting these points/radii, since for each $j$ we need that: 
\begin{enumerate}
\item $r_j^+ \leq B d(\tilde x_j, \del_{i_j - 1}\Omega)$, so that $\Omega \cap B_{r_j^+}(\tilde x_j)$ looks close to a polyhedral cone; 
\item $|x - \tilde x_j| \leq r_j^-$, so that we can apply Step 1 in $B_{r_j^+}(x_j)$ to get decay down to $r_j^-$;
\item $B_{r_j^+}(\tilde x_j) \subset B_{r_{j+1}^{-}/2}(\tilde x_{j+1})$, but $r_j^+ \geq c r_{j+1}^-$, so that we can control the mass and excess in $B_{r_j^+}(\tilde x_j)$ in terms of the mass/excess in $B_{r_{j+1}^{-}}(\tilde x_{j+1})$.
\end{enumerate}

Once we construct these $\tilde x_j$, $r_j^\pm$, we inductively prove a statement of the form: for each $j$, there is a $V_j \in \cP_{T_{\tilde x_j}\Omega}$ so that
\[
E(M, p_j + V_j, \tilde x_j, r) \leq \Lambda_j E r^{2\alpha} \quad \forall r_j^- \leq r \leq 1/4,
\]
and
\[
\theta_M(\tilde x_j, r) \leq (7/4) \Theta_{T_{\tilde x_j}\Omega} \quad \forall r_j^- \leq r \leq r_j^+.
\]
The excess control ensures $M \llcorner B_{r}(\tilde x_j)$ looks like a plane with some multiplicity, the mass control ensures this multiplicity is $\leq 1$, and the fact that $x \in B_{r_j^-/2}(\tilde x_j)$ ensures this multiplicity is $\geq 1$.  The proof of this statement follows in a fairly straightforward way from Lemma \ref{lem:B} and Step 1.  An extra argument by contradiction (Lemma \ref{lem:sharp-mass-bound}) is required to ensure the mass control at scale $B_{r_j^-}(\tilde x_j)$ carries over to scale $B_{r_{j-1}^+}(\tilde x_{j-1})$.

\textbf{Minimizers} Here we apply our regularity theorem to minimizing currents/isoperimetric sets of finite perimeter.  We classify low-dimensional minimizing cones, and prove a compactness theorem, which together with our regularity result implies a partial regularity theorem by standard dimension reducing techniques.

The compactness is fairly straightforward -- we just need to adapt an argument of Gruter \cite{Gr} to ensure mass cannot accumulate near the boundary.  The main step is classifying low-dimensional minimizing cones, which we prove by induction on the number of symmetries of $\Omega$.  The idea is as follows: Assume that any minimizing $T$ in $\Omega_0 \times \R^{m-1}$ is a horizontal plane.  Then if $T$ is minimizing in $\Omega_0\times \R^m$, by induction and our regularity theorem $T$ is a $C^{1,\alpha}$ surface away from $0$.  Under certain circumstances, then we can boost this to $C^{2,\alpha}$ regularity, and thereby adapt Simons' classical argument to prove $T$ must be planar.  Then, again in certain circumstances, a cut-and-paste argument implies this plane must be horizontal.

Unfortunately, even in low dimensions we start running into issues.  The barrier to adapting Simons' argument when $\Omega_0$ is more than $3$ dimensional is the lack of $C^{2,\alpha}$ regularity in $3$-dimensional cones.  We suspect this holds if the dihedral angles are at most $\pi/2$, but this requires a Neumann eigenvalue estimate for spherical domains which is not known.  On the flip side, for general dihedral angles it is not clear that every minimizing plane need be horizontal -- when $\Omega_0$ is only an intersection of $3$ half-spaces in $\R^3$, it is plausible there are non-horizontal minimizers.  Taken together, we get a codimension $2$ bound for the singular set in general domains, a codimension $3$ bound in domains with dihedral angles $\leq \pi/2$, and the (sharp) codimension $7$ bound in certain special cases (e.g when all dihedral angles are $=\pi/2$).

\section{First variation, mass control}\label{sec:first-var}

We prove in this section that the mass and total first variation of an $M \in \cIH_n(\Omega, B_1)$ are controlled by the mass, tangential mean curvature of $M$, and geometry of $\Omega$.  We must first prove monotonicity and the mass control, by cooking up an appropriate tangential vector field, and then we use this to prove control on the first variation.  Two important consequences are lower Ahlfors-regularity of $M$ (Corollary \ref{cor:lower-reg}), and $L^\infty-L^2$, $W^{1,2}-L^2$ bounds on excess (Corollary \ref{cor:height-w12-bound}).

In this section we fix $\OmegaRef$ a polyhedral cone domain (as per Definition \ref{def:omegaref}), and after a rotation there is no loss in assuming $\OmegaRef = \OmegaRefN^l \times \R^{m}$.  Note for the reader: in this section (and occasionally in Section \ref{sec:minz}) we will allow $m = 0$, but in all other sections we will assume $m \geq 1$.  See also Example \ref{ex:vertical-plane} and Lemma \ref{lem:minz-planes}.


\begin{lemma}\label{lem:monotonicity}
There are constants $\eps_{mn}(n)$, $c_{mn}(n)$, so that if
\[
\Omega \in \cD_{\eps}(\OmegaRef), \quad M \in \cIH_n(\Omega, B_1), \quad ||H^{tan}_M||_{L^\infty(B_1; \mu_M)} \leq \eps,
\]
for $\eps \leq \eps_{mn}$, then
\begin{equation}\label{eqn:monotonicity}
(1+c_{mn}(n)\eps \rho)^{n+1} \theta_M(0, \rho)
\end{equation}
is increasing in $\rho < 1$.  If $\eps = 0$, and $\sigma < \rho$, then we have the sharp monotonicity
\begin{equation}\label{eqn:sharp-mono}
\theta_M(0, \rho) - \theta_M(0, \sigma) = \int_{B_\rho(0)\setminus B_\sigma(0)} \frac{ |\pi_M^\perp(z)|^2}{|z|^{n+2}} d\mu_M(z).
\end{equation}
\end{lemma}

\begin{proof}
Define $Y(x) = D\Phi|_{\Phi^{-1}(x)}(\Phi^{-1}(x))$.  Then $Y$ is a $C^1$ vector field tangential to $\Omega$.  An easy computation, using \eqref{eqn:def-phi} and taking $\eps_{mn}(n)$ small, shows that 
\begin{equation}\label{eqn:Y-bounds}
|Y(x) - x| \leq c(n)\eps |x|^2, \quad |DY|_x - \mathrm{Id}| \leq c(n) \eps |x|.
\end{equation}
Let $\zeta$ be a smooth, decreasing approximation to $1_{(-\infty, 1)}$, such that $\spt \zeta \subset (-\infty, 1)$.  Take $\rho \in (0, 1)$, plug in the vector field $X(x) = \zeta(|x|/\rho) Y(x)$ into the first variation, and use \eqref{eqn:Y-bounds}, to deduce
\begin{align}\label{eqn:first-var-Y}
&\rho \frac{d}{d\rho} \int \zeta d\mu_M - n \int \zeta d\mu_M  - \rho \frac{d}{d\rho} \int \zeta |x^\perp|^2/|x|^2 d\mu_M \\
&\geq - c(n)\eps \int \zeta |x| d\mu_M - c(n) \eps \rho^2 \frac{d}{d\rho} \int \zeta d\mu_M.
\end{align}
Notice that the last terms on both the RHS and LHS are non-positive.  Setting $I(\rho) = \int \zeta d\mu_M$, and discarding the last term on the left, we get
\[
(1+c\eps\rho)\rho I' - (n - c\eps\rho) I \geq 0,
\]
where WLOG both constants $c = c(n)$ are the same.  This implies
\[
\frac{d}{d\rho} ( (1+c\eps)^{n+1} \rho^{-n} I(\rho)) \geq 0.
\]
Integrate in $\rho$, then take $\zeta \to 1_{(-\infty, 1]}$, to obtain the required conclusion \eqref{eqn:monotonicity}.

If $\eps = 0$, then in fact \eqref{eqn:first-var-Y} is an equality (with no errors on the right hand side).  Integrating up as before, but without discarding terms, gives \eqref{eqn:sharp-mono}.
\end{proof}

\begin{corollary}\label{cor:upper-reg}
Given any $\theta \in (0, 1)$, there are constants $\eps_m(\OmegaRef)$, $c_m(\OmegaRef, \theta)$ so that if
\[
\Omega \in \cD_{\eps_m}(\OmegaRef), \quad M \in \cIH_n(\Omega, B_1), \quad ||H^{tan}_M||_{L^\infty(B_1; \mu_M)} \leq \eps_m,
\]
then for every $x \in B_\theta$, and $0 < r < 1-|x|$, we have
\begin{equation}\label{eqn:upper-reg}
\theta_M(x, r) \leq c_m \theta_M(0, 1).
\end{equation}
\end{corollary}

\begin{proof}
Let $\sigma_n = \theta$, and set $\sigma_{i-1} = 1/4 + (3/4) \sigma_i \in (\sigma_i, 1)$ for $i = 1, \ldots, n$.  We prove by induction on $i$ that there is a $c_{m0}(\theta, \OmegaRef)$ so that given $x_i \in \del_i\Omega \cap B_{\sigma_i}$, then we have
\begin{equation}\label{eqn:ind-upper-reg}
\theta_M(x_i, r) \leq c_{m0}^{i+1}\mu_M(B_1) \quad \forall 0 < r < 1-|x_i|.
\end{equation}
This will establish \eqref{eqn:upper-reg} with $c_m = c_{m0}^{n+1}$.

Take $B(\OmegaRef)$ as in Lemma \ref{lem:B}, and let us ensure that $\eps_m(\OmegaRef)$ is sufficiently small so that $\eps_m \leq \eps_B$, and $c_B \eps_m \leq \eps_{mn}$.  Take $x_i \in \del_i\Omega \cap B_\theta$, and set $d_i = d(x_i, \del_{i-1}\Omega)$, $R_i = \min\{ Bd_i, 1-|x_i|\}$.  For $0 < r \leq R_i$, we have by Lemma \ref{lem:B} and monotonicity \eqref{eqn:monotonicity} that
\[
\theta_M(x_i, r) \leq c(n)\theta_M(x_i, R_i).
\]

First assume $R_i = 1-|x_i|$.  Then we have
\[
\theta_M(x_i, R_i) \leq (1-\sigma_i)^{-n} \mu_M(B_1) \leq c(n, \theta) \mu_M(B_1),
\]
and we are done.  Notice that if $i = m$, then $d_i = \infty$, and so this proves the inductive base case.

Assume now $R_i = Bd_i$.  If $d_i \geq (1/4)(1-|x_i|)$, then we can similarly estimate for $R_i \leq r \leq 1-|x_i|$:
\[
\theta_M(x_i, r) \leq \left( \frac{4}{B (1-\sigma_i)}\right)^{n}  \mu_M(B_1) \leq c(\OmegaRef, \theta) \mu_M(B_1),
\]
and we are done.  Let us assume now $d_i \leq (1/4)(1-|x_i|)$.  Since
\[
|x_i| + d_i \leq 1/4 + (3/4)|x_i| \leq \sigma_{i-1},
\]
we can find an $x_{i-1} \in \del_{i-1}\Omega \cap B_{\sigma_{i-1}}$ realizing $d_i \equiv d(x_i, \del_{i-1}\Omega)$.  Now by inductive hypothesis we have for $R_i \leq r \leq 1-|x_{i-1}| - d_i$:
\[
\theta_M(x_i, r) \leq ( 1 + 1/B)^{n} \theta_M(x_{i-1}, r + d_i) \leq c(\OmegaRef) c_{m0}^i \mu_M(B_1).
\]
On the other hand, since
\[
|x_{i-1}| + d_i \leq |x_i| + 2d_i \leq 1/2 + (1-1/2)\sigma_i \leq 1 - 1/c(n, \theta),
\]
if $r \geq 1-|x_{i-1}| - d_i$, then we have
\[
\theta_M(x_i, r) \leq c(n,\theta) \mu_M(B_1).
\]
This proves the inductive claim \eqref{eqn:ind-upper-reg}, and finishes the proof of Corollary \ref{cor:upper-reg}.
\end{proof}

\begin{theorem}\label{thm:first-var}
There is an $\eps(\OmegaRef)$ so that if
\[
\Omega \in \cD_{\eps}(\OmegaRef), \quad M \in \cIH_n(\Omega, B_1), \quad ||H^{tan}_M||_{L^\infty(B_1; \mu_M)} \leq \eps, \quad \mu_M(B_1) < \infty, 
\]
then $||\delta M||$ is a Radon measure on $B_1$, and for any $X \in C^1_c(B_1, \R^{n+1})$ we can write
\begin{gather}\label{eqn:deltaM-decomp}
\delta M(X) = -\int H^{tan}_M \cdot X d\mu_M + \int \eta \cdot X d\sigma,
\end{gather}
where $\sigma \perp \mu_M$ is a non-negative Radon measure supported in $\del\Omega$; $\eta(x) \in \cN_{T_x\Omega}$ for $\sigma$-a.e. $x$; and $||\delta M||(\del_{i}\Omega) = 0$ for all $i \leq n-2$.

Moreover, we have
\begin{equation}\label{eqn:deltaM-bounds}
\int \phi d||\delta M|| \leq c(\OmegaRef) \int |H^{tan}_M| \phi + |\nabla \phi| d\mu_M
\end{equation}
for all $\phi \in C^1_c(B_1)$ non-negative.
\end{theorem}

\begin{remark}
Notice that \eqref{eqn:deltaM-bounds} is scale-invariant.
\end{remark}

We prove Theorem \ref{thm:first-var} inductively on $l$.  The cases $l = 1$, $l = 2$, and $l \geq 3$ are handled separately.  We shall use several times the following relation.

Let $Q$ be a $C^2$ closed $p$-manifold, with $p \leq n$, and suppose the distance function $d$ to $Q$ is smooth on $B_1 \setminus Q$.  Suppose $\delta M$ is a Radon measure in $B_1 \setminus Q$, and we decompose $\delta M$ as in \eqref{eqn:deltaM-decomp} for $X$ supported away from $Q$.  Let $\tau$ be any constant vector, and $h \in C^1_c(B_1)$ a compactly supported, non-negative function.  Then we have for any $\rho > 0$:
\begin{equation}\label{eqn:delta-tau}
\frac{1}{\rho} \int_{B_\rho(Q)} (\nabla d \cdot \tau) h - \int \min(d/\rho, 1) h \eta \cdot \tau d\sigma = - \int (H^{tan}_M h + \nabla h) \cdot \tau \min(d/\rho, 1) d\mu_M.
\end{equation}

\vspace{5mm}

\begin{proof}[Proof when $l = 0, 1$]
When $l = 0$ there is nothing to show.  Take $l = 1$.  Assume $\eps(n)$ is sufficiently small, so that the distance function to $\del\Omega \equiv \Phi(\Psi(\{0\}\times \R^{m})) \cap B_1$ is smooth in $\Omega$.  From Theorem \ref{thm:first-var-free-boundary}, we get that $\delta M$ is a Radon measure on $B_1$, and we can write $\delta M$ as in \eqref{eqn:deltaM-decomp}, where $\sigma$ satisfies our conclusions, and $\eta = \pm \nu_\Omega$.  It suffices to verify that $\eta = + \nu_\Omega$.

Let $d$ be the distance function to $\del\Omega$, and $\tilde d$ the signed distance function, so that $\tilde d = d$ in $\Omega$.  Let $X = \phi D \tilde d$, for $\phi \in C^1_c(B_1)$ non-negative.  From Theorem \ref{thm:first-var-free-boundary}, we have
\[
\int \eta \cdot (-\nu_\Omega) \phi d\sigma = \int \eta \cdot D \tilde d \phi d\sigma = -\Gamma_1(\phi |Dd|^2) \leq 0.
\]
Since $\phi$ is arbitrary, we deduce $\eta \cdot \nu_\Omega \geq 0$ for $\sigma$-a.e. $x$.

Therefore $\eta = \nu_\Omega$, and (using Theorem \ref{thm:first-var-free-boundary}) we in fact have
\[
\int \phi d\sigma = \Gamma_1(\phi |Dd|^2) \leq \int |H^{tan}_M| \phi + |\nabla \phi| d\mu_M.
\]
Since $||\delta M|| = |H^{tan}_M| \mu_M + \sigma$, we obtain \eqref{eqn:deltaM-bounds}.
\end{proof}

\begin{proof}[Proof when $l = 2$]
In this case $\OmegaRefN \subset \R^2$ is simply a wedge, with angle $< \pi$.  For ease of notation let us write $\del\OmegaRefN = L_1 \cup L_2$, for $L_1, L_2$ rays extending from $0$, and corresponding let us decompose
\[
\del\Omega  = F_1 \cup F_2,
\]
where $F_i = \Phi(\Psi(L_i \times \R^{m})) \cap B_1$, so that $\del_{m}\Omega = F_1 \cap F_2$.

Away from $F_1 \cap F_2$, the outer normal $\nu_\Omega$ is well-defined on $\del\Omega$.  Write $\nu_i$ for the outward normal on $F_i$.  Observe that each $x \in F_1 \cap F_2$, we have
\[
\cN_{T_x\Omega} = \{ a_1 \nu_1(x) + a_2 \nu_2(x) : a_1, a_2 \in (0, \infty) \}.
\]

Write $d_0, d_1, d_2$ for the distance functions to $F_1\cap F_2, F_1, F_2$ respectively.  Convexity of $\OmegaRefN$ (and $\sint\OmegaRefN \neq\emptyset$) implies that by taking $\eps(\OmegaRefN)$ sufficiently small, we can ensure the $d_i$ are smooth in $\mathrm{int} \Omega \cap B_1$, and that
\begin{equation}\label{eqn:l2-case-DdDd}
D d_i \cdot Dd_j \geq 0 \text{ on } \mathrm{int} \Omega \cap B_1.
\end{equation}
Also by convexity, taking $\eps(\OmegaRefN)$ small, we can choose a vector $\tau \in \R^2\times\{0^{m}\}$ and an $\eps_0(\OmegaRefN) > 0$ so that
\[
\tau \cdot D d_i \geq \eps_0 > 0 \quad \text{ on } \mathrm{int} \Omega \cap B_1.
\]
This of course implies that
\begin{equation}\label{eqn:l2-case-tau-nu}
\tau \cdot \nu_\Omega \leq -\eps_0 \text{ on } \del\Omega \cap B_1 \setminus F_1 \cap F_2.
\end{equation}

The upper bounds \eqref{eqn:upper-reg} imply that for every $\theta < 1$, and any $x \in F_1 \cap F_2 \cap B_\theta$, we have
\[
\mu_M(B_\rho(x)) \leq c(\OmegaRef, \theta)\mu_M(B_1) \rho^n.
\]
In particular, we get
\begin{equation}\label{eqn:l2-case-mass}
\mu_M(B_\rho(F_1 \cap F_2) \cap B_\theta) \leq c(\OmegaRef, \theta)\mu_M(B_1) \rho.
\end{equation}

From the $l = 1$ case, $||\delta M||$ is a Radon measure on $B_1 \setminus (F_1 \cap F_2)$.  Take $\theta < 1$, and $h \in C^1_c(B_\theta)$ non-negative.  Let us apply formula \eqref{eqn:delta-tau} with our choice of $\tau$, and $d_0$ in place of $d$, and make use of \eqref{eqn:l2-case-DdDd}, \eqref{eqn:l2-case-tau-nu}, \eqref{eqn:l2-case-mass} to obtain:
\begin{align}
\eps_0 \int \min(d_0/\rho,1) h d\sigma  
&\leq \eps \int h + |\nabla h| d\mu_M + \frac{1}{\rho} \int_{B_\rho(F_1 \cap F_2)}h d\mu_M \\
&\leq c |h|_{C^1} \mu_M(B_\theta) + |h|_{C^0} \frac{1}{\rho} \mu_M(B_\rho(F_1 \cap F_2) \cap B_\theta) \\
&\leq c |h|_{C^1}.
\end{align}
for $c = c(\OmegaRef, \theta, \mu_M(B_1))$.  Taking $\rho \to 0$, we get that
\[
||\delta M||(B_\theta \setminus (F_1 \cap F_2)) < \infty,
\]
for all $\theta < 1$.

We can now apply Theorem \ref{thm:first-var-boundary} to deduce that $||\delta M||$ is a Radon measure on $B_1$, and we can decompose $\delta M$ as in \eqref{eqn:deltaM-decomp}, where $\sigma \perp \mu_M$ is non-negative, supported in $\del\Omega$; and for $\sigma$-a.e. $x$, we have $|\eta(x)| = 1$, and
\begin{enumerate}
	\item $\eta(x) = \nu_i(x)$ if $x \in F_i \setminus (F_1 \cap F_2)$);
	\item $\eta(x) \perp T_x (F_1 \cap F_2)$ if $x \in F_1 \cap F_2$.
\end{enumerate}
To verify the first half of our theorem, it will suffice to show that $\eta(x) \cdot \nu_i(x) \geq 0$ when $x \in F_1 \cap F_2$.

Let $\tilde d_i$ be the signed distance function to $F_i$, which coincides with $d_i$ on $\mathrm{int} \Omega \cap B_1$.  Take $\phi \in C^1_c(B_1)$, and let $X = \phi D\tilde d_i$.  By Theorem \ref{thm:first-var-boundary} and \eqref{eqn:l2-case-DdDd} we get
\[
\int_{F_1 \cap F_2} \eta \cdot (-\nu_i) \phi d\sigma = \int_{F_1 \cap F_2} \eta \cdot D\tilde d_i \phi d\sigma = -\Gamma_2(\phi Dd_0 \cdot Dd_i) \leq 0.
\]
Since $\phi$ is arbitrary, we deduce $\eta \cdot \nu_i \geq 0$.

To prove the second part, we observe that by our choice of $\tau$ and characterisation above, we have $\eta(x) \cdot \tau \leq -\eps_0$ for $\sigma$-a.e. $x$.  Therefore by the first variation,
\begin{align}
\eps_0 \int \phi d\sigma \leq - \int \eta \cdot \tau \phi d\sigma 
&= \int -H^{tan}_M \cdot \tau \phi - \nabla \phi \cdot \tau d\mu_M \\
&\leq \int |H^{tan}_M|\phi + |\nabla \phi| d\mu_M.
\end{align}
This completes the proof of case $l = 2$.
\end{proof}

\begin{proof}[Proof when $l \geq 3$]
By convexity of $\OmegaRefN$ and since $\sint\OmegaRefN \neq\emptyset$, ensuring $\eps(\OmegaRefN)$ is sufficiently small, we can find a vector $\tau \in \R^{n+1}$, and $\eps_0(\OmegaRefN) > 0$, so that
\[
\tau \cdot \nu \leq - \eps_0 \quad \forall \nu \in \cN_{T_x\Omega}, \quad \forall x \in \del\Omega \cap B_1.
\]

Suppose by inductive hypothesis Theorem \ref{thm:first-var} holds for $1, 2, \ldots, l - 1$ in place of $l$.  Since (using Lemma \ref{lem:B}) any point of $\del\Omega \setminus \del_m\Omega$ is locally modelled on some $\OmegaRefN' \times \R^{m+1}$, for $\OmegaRefN' \subset \R^{l-1}$, we have by induction that $||\delta M||$ is a Radon measure in $B_1 \setminus \del_m\Omega$, and a have a decomposition of $\delta M$ as in \eqref{eqn:deltaM-decomp} for any $X$ supported away from $\del_m\Omega$.

Ensuring $\eps(\OmegaRefN)$ is sufficiently small, we can assume the distance function $d$ to $\del_m \Omega$ is smooth in $\mathrm{int} \Omega \cap B_1$.  Analogous to the proof when $l = 2$, the mass bounds \eqref{eqn:upper-reg} imply that
\begin{equation}\label{eqn:l3-case-mass}
\mu_M(B_\rho(\del_m\Omega) \cap B_\theta) \leq c(\OmegaRef, \theta)\mu_M(B_1) \rho^2
\end{equation}
for all $\theta < 1$.  Therefore, as before, if $h \in C^1_c(B_\theta)$ is non-negative, we can apply \eqref{eqn:delta-tau} with our choice of $\tau$, $d$ to deduce
\[
\eps_0 \int \min(d/\rho, 1) h d\sigma \leq c |h|_{C^1}
\]
independently of $\rho$, which implies $||\delta M||(B_\theta \setminus \del_m\Omega) < \infty$.

Let $X \in C^1_c(B_\theta, \R^{n+1})$, for any $\theta < 1$.  Let $\phi(x) = \zeta(d/\rho)$, for $\zeta$ a smooth, compactly supported approximation to $1_{(-\infty, 1]}$.  We compute
\begin{align}
\left| \int div(\phi X) d\mu_M \right| 
&\leq \int |\phi'| |X| + |\phi| |DX| d\mu_M \\
&\leq c |X|_{C^1} \frac{1}{\rho} \mu_M(B_\rho(\del_m \Omega) \cap B_\theta) \\
&\leq c \rho  \to 0 \quad \text{ as } \rho \to 0 .
\end{align}
Therefore, since $||\delta M|| \llcorner (\del_m \Omega)^C$ is a Radon measure on $B_1$, we deduce that
\begin{align}
\delta M(X) 
&= \lim_{\rho \to 0} \left( \delta M(\phi X) + \delta M( (1-\phi) X) \right) \\
&= -\int_{B_1 \setminus \del_m\Omega} H^{tan}_M \cdot X + \int_{B_1 \setminus \del_m \Omega} \eta \cdot X d\sigma.
\end{align}
So in fact $||\delta M||$ is a Radon measure on $B_1$, and $||\delta M||(\del_m\Omega) = 0$, so that same decomposition \eqref{eqn:deltaM-decomp} holds for all $X$.  The inequality \eqref{eqn:deltaM-bounds} follows by the same computation as in the $l = 2$ case.
\end{proof}

Inequality \eqref{eqn:deltaM-bounds} is important for ensuring good compactness, but even more importantly it implies a Sobolev inequality, which allows us to prove a mean value inequality for subharmonic functions.
\begin{theorem}\label{thm:moser}
Let $M \in \cI_n(B_1)$ satisfy \eqref{eqn:deltaM-bounds} for all non-negative $\phi \in C^1_c(B_1)$, and additionally assume that
\begin{equation}\label{eqn:moser-hyp1}
\mu_M(B_1) \leq A < \infty, \quad ||H^{tan}_M||_{L^\infty(B_1; \mu_M)} \leq 1.
\end{equation}
Then the following holds: if $u \in C^1(B_1)$ is a non-negative function, and $a$ is a constant, such that
\begin{equation}\label{eqn:moser-hyp2}
\int \nabla \zeta \cdot \nabla u d\mu_M \leq a \int \zeta u + |\nabla \zeta| u + \zeta |\nabla u | d\mu_M \quad \forall \zeta \in C^1_c(B_1) \text{ non-negative},
\end{equation}
then for every $\sigma < 1$ we have the inequality
\begin{equation}\label{eqn:moser-concl}
\sup_{\spt \mu_M \cap B_\sigma} u \leq c(\OmegaRef, A, a, \sigma) \int_{B_1} u d\mu_M.
\end{equation}
\end{theorem}

\begin{proof}
By the same argument as in \cite{MiSi} (see Appendix \ref{sec:sobolev} for more details), inequality \eqref{eqn:deltaM-bounds} implies the Sobolev inequality: for all $h \in C^1_c(B_1)$ non-negative, then
\begin{equation}\label{eqn:sobolev-1}
\sup h \leq c(\OmegaRef) \int |H^{tan}_M|h + |\nabla h| d\mu_M
\end{equation}
if $n = 1$, and 
\begin{equation}\label{eqn:sobolev}
\left( \int h^{n/(n-1)} d\mu_M \right)^{(n-1)/n} \leq c(\OmegaRef) \int |H^{tan}_M| h + |\nabla h| d\mu_M
\end{equation}
for $n \geq 2$.  The bound \eqref{eqn:moser-concl} then follows from \eqref{eqn:moser-hyp1}, \eqref{eqn:sobolev} (or \eqref{eqn:sobolev-1} if $n =1$), and \eqref{eqn:moser-hyp2} by standard iteration methods.
\end{proof}

An immediate consequence is lower Ahlfors regularity in every ball centered on the support of $M$.
\begin{corollary}\label{cor:lower-reg}
There is an $\eps(\OmegaRef)$ so that given
\[
\Omega \in \cD_{\eps}(\Omega), \quad M \in \cIH_n(\Omega, B_1), \quad ||H^{tan}_M||_{L^\infty(B_1; \mu_M)} \leq \eps, \quad \theta_M(0, 1) \leq A,
\]
then for any $\theta < 1$, $x \in \spt M \cap B_\theta$ and $0 < r < 1-|x|$, we have
\begin{equation}\label{eqn:lower-reg}
\mu_M(B_r(x)) \geq r^n/c(\OmegaRef, A, \theta)
\end{equation}
\end{corollary}


\begin{proof}
By Theorem \ref{thm:first-var} $M$ satisfies the scale-invariant inequality \eqref{eqn:deltaM-bounds}, and therefore if $M' = (\eta_{x, r})_\sharp M$ then $M'$ satisfies \eqref{eqn:deltaM-bounds} also.  From \eqref{eqn:upper-reg} and scaling we have the bounds
\[
\mu_{M'}(B_1) \leq c(\Omega, A, \theta), \quad ||H^{tan}_{M'}||_{L^\infty(B_1; \mu_{M'})} \leq r\eps \leq 1.
\]
Therefore we can apply Theorem \ref{thm:moser} to $M'$ with $u = 1$, $\sigma = 1/2$, to get
\[
1 \leq c(\Omega, A, \theta) \mu_{M'}(B_1),
\]
which implies the required estimate \eqref{eqn:lower-reg}.
\end{proof}

Another important consequence of first variation control and Theorem \ref{thm:moser} are the following $L^\infty$ and $W^{1,2}$ estimates, which will be important in the blow-up argument to ensure good convergence to the linear problem.
\begin{corollary}\label{cor:height-w12-bound}
Given any $\theta \in (0, 1)$, there are constants $\eps_e(\OmegaRef)$, $c_{e}(\OmegaRef, A, \theta)$, so that if
\[
\Omega \in \cD_{\eps_e}(\OmegaRef), \quad M \in \cIH_n(\Omega, B_1), \quad ||H^{tan}_M||_{L^\infty(B_1; \mu_M)} \leq \eps_e, \quad \theta_M(0,1) \leq A,
\]
then given any $V \in \cP_{T_0\Omega}$, and $p \in V^\perp$, we have the height bound
\begin{equation}\label{eqn:sharp-height}
\sup_{z \in B_{\theta} \cap \spt M} d(z, p + V)^2 \leq c_{e} E(\Phi, M, p+V, 0, 1),
\end{equation}
and $W^{1,2}$ bound
\begin{equation}\label{eqn:W12-bound}
\int_{B_{\theta}} |\pi_{T_zM} - \pi_V|^2 d\mu_M(z) \leq c_{e} E(\Phi, M, p + V, 0, 1).
\end{equation}

\end{corollary}


\begin{proof}
Since $V \in \cP_{T_0\Omega}$, we can assume after rotation that $V = \R^n\times \{0\}$.  Write $\Omega = \Phi(\Psi(\OmegaRef)) \cap B_1$, and set $\tilde e_{n+1} = D\Phi_{\Phi^{-1}(x)}(e_{n+1})$.  Then $\tilde e_{n+1}$ is tangential to $\Omega$, and an easy computation gives that, for $x \in B_1$ and $\eps_{e}(n)$ sufficiently small, 
\begin{equation}\label{eqn:tilde-e-bounds}
|\tilde e_{n+1} - Dx_{n+1}| \leq c(n) |D^2\Phi|_{C^0(B_2)}, \quad |D\tilde e_{n+1}| \leq c(n) |D^2\Phi|_{C^0(B_2)}.
\end{equation}
Ensure also that $\eps_e \leq \eps$, the constant from Theorem \ref{thm:first-var}.

Given $\phi \in C^1(B_1)$, plug the vector field $X = \phi \tilde e_{n+1}$ into the first variation, to obtain
\[
\int \nabla (x_{n+1} - p_{n+1}) \cdot \nabla \phi \leq c(n) (||H^{tan}_M||_{L^\infty(B_1; \mu_M)} + |D^2\Phi|_{C^0(B_2)}) \int |\phi| + |\nabla \phi| d\mu_M.
\]
Now if we assume $\phi \geq 0$, and let
\[
u = (x_{n+1} - p_{n+1})^2 + ||H^{tan}_M||_{L^\infty(B_1;\mu_M)}^2 + |D^2 \Phi|_{C^0(B_2)}^2,
\]
then we can compute
\begin{align*}
&\int \nabla u \cdot \nabla \phi  d\mu_M \\
&= \int 2 \nabla (x_{n+1} - p_{n+1}) \cdot \nabla ((x_{n+1} - p_{n+1}) \phi) - 2 \phi |\nabla x_{n+1}|^2 d\mu_M \\
&\leq \int c(n) (||H^{tan}_M||_{L^\infty} + |D^2\Phi|_{C^0(B_2)}) ( |(x_{n+1} - p_{n+1}) \phi| + |\nabla ((x_{n+1} - p_{n+1}) \phi)|) - 2 |\nabla x_{n+1}|^2 \phi d\mu_M \\
&\leq c(n) \int u \phi + u |\nabla \phi| d\mu_M.
\end{align*}
We can then use Theorem \ref{thm:moser} to get the required \eqref{eqn:sharp-height}.

To prove \eqref{eqn:W12-bound}, we plug in the field $X = \phi^2 (x_{n+1} - p_{n+1}) \tilde e_{n+1}$ into the first variation, and use \eqref{eqn:tilde-e-bounds}, to obtain:
\begin{align*}
&\int \phi^2 |\nabla x_{n+1}|^2 d\mu_M \\
&\leq \int c(n) (\phi^2 + \phi |\nabla \phi|) ( |\nabla x_{n+1}| + |x_{n+1} - p_{n+1}|)( ||H^{tan}_M||_{L^\infty(B_1)} + |D^2 \Phi|_{C^0(B_1)}) d\mu_M \\
&\quad + 2 \phi \nabla \phi \cdot \nabla x_{n+1}|x_{n+1} - p_{n+1}| d\mu_M .
\end{align*}

Rearranging, we get
\[
\int \phi^2 |\nabla x_{n+1}|^2 d\mu_M \leq \int c(n) (\phi^2 + |\nabla \phi|^2) (||H^{tan}_M||^2_{L^\infty} + |D^2\Phi|^2_{C^0(B_2)} + |x_{n+1} - p_{n+1}|^2) d\mu_M .
\]
Choosing any fixed $\phi$ satisfying $\phi \equiv 1$ on $B_{\theta}$, and using the relation 
\[|\nabla x_{n+1}|^2|_z = (T_z M) : V^\perp = \frac{1}{2} |\pi_{T_zM} - \pi_V|^2,\]
we deduce \eqref{eqn:W12-bound}.
\end{proof}

The height bound \eqref{eqn:sharp-height} with the upper-Ahlfors regularity \eqref{eqn:upper-reg} imply that mass cannot accumulate near the boundary of our domain.
\begin{corollary}\label{cor:mass-non-conc}
Given $\tau, \theta \in (0, 1)$, there is an $\eps(\OmegaRef, \theta, A, \tau) \leq \eps_{e}(\OmegaRef)$ so that under the assumptions of Corollary \ref{cor:height-w12-bound}, if $M$ additionally satisfies
\[
E(\Phi, M, \Omega, p + V, 0, 1) \leq \eps^2
\]
for some $V \in \cP_{T_0\Omega}$, $p \in V^\perp$, then we have the non-concentration estimate
\begin{equation}\label{eqn:mass-non-conc}
\mu_M(B_\theta \cap B_\tau(\del\Omega)) \leq c(\OmegaRef, \theta, A) \tau.
\end{equation}
\end{corollary}

\begin{proof}
There is no loss in assuming $\tau \leq (1-\theta)/4$.  From the height bound \eqref{eqn:sharp-height}, ensuring $\eps(\OmegaRef, \theta, A, \tau)$ is sufficiently small, we have
\[
\spt M \cap B_\theta \cap B_\tau(\del\Omega) \subset B_{2\tau}( \del [ \Omega \cap (p + V)]).
\]
We can cover $B_\theta \cap B_{2\tau}(\del [\Omega \cap (p + V)])$ with balls $\{B_\tau(x_i)\}_{i=1}^Q$ such that $Q \leq c(\OmegaRef) \tau^{1-n}$, and each $B_\tau(x_i) \subset B_{(1+\theta)/2}$.  Therefore by the mass bounds \eqref{eqn:upper-reg} we get
\[
\mu_M(B_\theta \cap B_\tau(\del\Omega)) \leq \sum_{i=1}^Q \mu_M(B_\tau(x_i)) \leq c(\OmegaRef, A, \theta) \tau. \qedhere
\]
\end{proof}

\section{Eigenvalues and the linearized operator}

Let us fix here a polyhedral cone domain $\OmegaRef = \OmegaRefN^l \times \R^m \subset \R^{n+1}$, with $m \geq 1$.  We prove in this section appropriate decay of the linearized problem on the $n$-dimensional planar wedge $\OmegaRefN \times \R^{m-1} \subset \R^n$.  Recall that if $D$ is a manifold with possible Lipschitz boundary $\del D$, then $\phi$ is said to be a Neumann eigenvalue of $D$ with eigenvalue $\mu$ if it solves
\begin{equation}\label{eqn:neumann}
\int_D \nabla \phi \cdot \nabla \zeta + \mu \phi \zeta = 0 \quad \forall \zeta \in W^{1,2}(D).
\end{equation}
Provided $D$ is compact and $\del D$ Lipschitz, then standard elliptic theory guarantees the existence of a countable sequence of Neumann eigenvalues $0 = \mu_0 < \mu_1 \leq \ldots \to \infty$, and corresponding eigenfunctions $\phi_i \in W^{1,2}(D)$, so that the $\phi_i$ form an $L^2(D)$-ON basis.

We require the following Lichnerowicz eigenvalue bound for piecewise-convex domains in the sphere.  For smooth domains this is classical, while the result piecewise smooth domains follows from a straightforward approximation.
\begin{theorem}[Lichnerowicz eigenvalue bound]\label{thm:evalue}
Let $D$ be a piecewise-smooth, convex domain of $S^{n-1} \subset \R^n$ ($n > 1$), and let $\mu_1$ be the first (non-zero) Neumann eigenvaue, $\phi$ the first Neumann eigenfunction.  Then $\mu_1 \geq n-1$, and $\mu_1 = n-1$ if and only if the $1$-homogenous extension of $\phi$ is linear.
\end{theorem}

\begin{proof}
First assume $D$ is smooth and weakly-convex (in the sense that the second fundamental form of the boundary is $\geq 0$).  Let $\mu$ be a Neumann eigenvalue for $D$ in the sense of \eqref{eqn:neumann} and $u$ be the corresponding eigenfunction.  By standard elliptic regularity $u \in C^\infty(\overline{D})$, and satisfies $\del_n u=0$ along $\del D$.  Therefore we can integrate the Bochner formula
\[
\frac{1}{2} \Delta |\nabla u|^2 \geq (n-2) |\nabla u|^2 + \frac{1}{n-1} (\Delta u)^2 - \mu|\nabla u|^2, 
\]
using the convexity of $D$ and boundary condition on $u$, to obtain 
\[
(\mu - (n-1)) \frac{n-2}{n-1} \int_D |\nabla u|^2 \geq \int_D |\nabla^2 u|^2 - \frac{1}{n-1} (\Delta u)^2 .
\]
This proves $\mu \geq n-1$.  If $\mu = n-1$, then we have $\nabla_{ij}^2 u + u g_{ij} = 0$ ($g_{ij}$ being the spherical metric), which implies that if we set $\tilde u(x) = |x| u(x/|x|)$ to be the $1$-homogenous extension of $u$, then $D^2 \tilde u = 0$.  This proves Theorem \ref{thm:evalue} in the case when $D$ is smooth.

Now take $D$ to be convex and piecewise-smooth.  Let $\mu$ be the first non-zero Neumann eigenvalue of $D$, and $u$ the corresponding eigenfunction.  Write $\del D = \del^r D \cup \del^s D$, where $\del^r D$ is smooth.  Take any sequence $\eps_i \to 0$.  We can find smooth, weakly convex domains $D_i \subset D_{i+1} \subset \ldots \subset D$ such that $D_i = D$ outside $B_{\eps_i}(\del^s D)$.  Let $\mu_i$ be the first Neumman non-zero eigenvalues of $D_i$, and $u_i$ the eigenfunctions.

Normalize $u$, $u_i$ such that $||u||_{L^2(D)} = ||u_i||_{L^2(D_i)} = 1$.  It's easy to check, using the Raliegh quotient and standard Sobolev theory, that $\mu_i \to \mu$ and $u_i \to u$ in $L^2(D')$ for every $D' \subset D$ such that $D' \cap \del^s D = \emptyset$.  In fact, by standard elliptic regularity $u_i \to u$ in $C^k(D')$.  We deduce that
\[
(\mu - (n-1)) \frac{n-2}{n-1} \mu \geq \int_{D \setminus \del^s D} |\nabla^2 u|^2 - \frac{1}{n-1} (\Delta u)^2,
\]
and the rest of the proof follows as in the smooth case.
\end{proof}

For general Lipschitz cones, we have the following standard Fourier-type expansion.
\begin{lemma}\label{lem:expansion}
Let $D$ be a Lipschitz subdomain of $S^{n-1} \subset \R^n$ ($n \geq 2$), and let $CD$ be the cone over $D$.  Let $\mu_i, \phi_i$ be the Neumann eigenvalues, eigenfunctions of $D \subset S^{n-1}$.  Suppose $u \in W^{1,2}(CD \cap B_1)$ solves
\begin{equation}\label{eqn:lem-exp-hyp}
\int Du \cdot D\zeta = 0 \quad \forall \zeta \in C^1(CD) \cap C^1_c(B_1^n).
\end{equation}
Then we have the expansion in $W^{1,2}(CD \cap B_1)$:
\[
u(x = r\omega) = \sum_{i=0}^\infty a_i r^{\gamma_i} \phi_i(\omega), \quad \gamma_i = -(n-2)/2 + \sqrt{((n-2)/2)^2 + \mu_i}.
\]
Here we write $r = |x|$, and $\omega = x/|x| \in D$.

If $n = 1$, and $u$ satisfies \eqref{eqn:lem-exp-hyp} for $CD = \R$ or $[0, \infty)$ (resp.), then $u$ is linear or constant (resp.)
\end{lemma}

\begin{proof}
This is standard, however for the reader's convenience, and to emphasize that we require no further regularity of $u$ beyond $W^{1,2}$, we give a proof in Appendix \ref{sec:fourier}.
\end{proof}

Combining Theorem \ref{thm:evalue} with Lemma \ref{lem:expansion} gives the following characterization of the linear problem of our model.
\begin{prop}\label{prop:expansion}
Let $D_0 = \OmegaRefN \cap S^{l-1}$, and $D = (\OmegaRef \times \R^{m-1}) \cap S^{n-1}$.  If $l \geq 2$ let $\mu_1(D_0)$ be the first Neumann eigenvalue of $D_0$, otherwise let us define $\mu_1(D_0) = 2$.

Suppose $u \in W^{1,2}( (\OmegaRefN \times \R^{m-1}) \cap B_1)$ solves
\begin{equation}\label{eqn:expansion-hyp}
\int_{\OmegaRefN \times \R^{m-1}} Du \cdot D\zeta = 0 \quad \forall \zeta \in C^1_c(B_1^n).
\end{equation}
Then there is an increasing sequence $\{\alpha_i\}_{i=1}^\infty$, and constants $b \in \R$, $A \in \{0^l\}\times \R^{m-1}$, $a_i \in \R$, and $L^2(D)$-orthonormal Neumann eigenfunctions $\psi_i$, so that we have the expansion in $W^{1,2}( (\OmegaRefN \times \R^{m-1}) \cap B_1)$:
\[
u(x = r\omega) = b + A \cdot x + \sum_{\alpha_i \geq e}^\infty a_i r^{\alpha_i} \psi_i(\omega),
\]
which for every fixed $r$ is $L^2(D)$-orthogonal, and where
\[
e \geq \min\{-(l-2)/2 + \sqrt{ ((l-2)/2)^2 + \mu_1(D_0)} , 2\} > 1.
\]
\end{prop}

\begin{remark}\label{rem:expansion}
When $l = 0,1$, so $\OmegaRef$ is $\R^n$ or a half-space, then $|\alpha_i| \geq 2$.  When $l=2$, in which case $\OmegaRef$ is a wedge formed by two hyperplanes, then $|\alpha_i| \geq \min\{2, \pi/\gamma\}$, where $\gamma$ is the angle of the wedge.
\end{remark}

\begin{proof}
If $l = 0$ then $u$ is harmonic in the entire ball $B_1$, and then Proposition \ref{prop:expansion} follows by the usual Fourier expansion.  Consider now $l \geq 1$.  Let
\[
u(x = r\omega) = \sum_{i=0}^\infty a_i r^{\gamma_i} \phi_i(\omega) =: \sum_{i=0}^\infty u_i(x)
\]
be the expansion of Lemma \ref{lem:expansion}.  It suffices to show that if $\gamma_i < e$, then $u_i$ is either constant, or of the form $A \cdot x$ for some $A \in \{0^l\}\times \R^{m-1}$.

We have $\mu_0 = \gamma_0 = 0$, and $\phi_0 = \mathrm{const}$, so the first term $u_0$ is a constant.  Since $D$ is convex, $\gamma_1 \geq 1$.

Suppose $\gamma_1 = 1$, so that $u_1$ is $1$-homogenous.  Given $v \in \{0^l\}\times \R^{m-1}$, then $v \cdot D u_1$ is a $0$-homogenous solution to \eqref{eqn:expansion-hyp}, and hence by the previous paragraph must be constant.  We deduce that $u_1 = A \cdot x + \tilde u_1(x)$, for some $A \in \{0^l\}\times \R^{m-1}$, and some $1$-homogenous $\tilde u_1(x) : \OmegaRefN \cap B_1 \to \R$ solving
\begin{equation}\label{eqn:expansion-2}
\int_{\OmegaRefN \cap B_1} Du_i \cdot D \zeta = 0 \quad \forall \zeta \in C^1_c(B_1^l) .
\end{equation}
However, since $D_0$ is $0$-symmetric, by Lemma \ref{lem:expansion} there are no $1$-homogenous solutions to \eqref{eqn:expansion-2}, and so $\tilde u_1 = 0$.

Suppose $\gamma_i \in (1, 2)$.  Then for each $v$ as above, $v \cdot D u_i$ is a $\gamma_i - 1 \in (0, 1)$-homogenous solution to \eqref{eqn:expansion-hyp}, and hence by the previous two paragraphs must be zero.  We deduce $u_i(x) = u_i( \pi_{\R^l \times\{0^{m-1}\}}(x))$ is a function of $\R^l$ only, and hence solves \eqref{eqn:expansion-2}.  If $l = 1$ then $u_i$ must be constant, which is impossible since $\gamma_i > 1$.  If $l \geq 2$, then applying Lemma \ref{lem:expansion} to $u_i$, we deduce
\[
\gamma_i \geq -(l-2)/2 + \sqrt{ ((l-2)/2)^2 + \mu_1(D_0)}. \qedhere
\]
\end{proof}

\section{$L^2$ excess decay}\label{sec:decay}

In this section we work towards an excess decay theorem (Proposition \ref{prop:decay}), which gives a decay estimate of the $L^2$ excess $E$ when $M$ looks close to planar, and $\Omega$ looks close to a cone.  At a general point $x$, for most scales $\Omega$ will look like one of only finitely many cones, and so in the next Section we will be able to use Corollary \ref{cor:decay} on each model cone to prove decay on all scales.  As before, in this Section we fix $\OmegaRef$ a polyhedral cone domain of the form $\OmegaRefN^l \times \R^{m}$, where we assume $m \geq 1$.

\begin{prop}\label{prop:decay}
Let $e(\OmegaRefN)$ be the exponent bound of Proposition \ref{prop:expansion}.  For any $\theta \leq 1/10$, there are constants $\delta_0(\OmegaRef, \theta)$, $c_0(\OmegaRef)$, so that the following holds.  Let $\Omega = \Phi(\Psi(\OmegaRef))\cap B_1 \in \cD_{\delta_0}(\OmegaRef)$, and $M \in \cIH_n(\Omega, B_1)$, satisfy
\begin{equation}\label{eqn:decay-hyp}
E_{\delta_0}(\Phi, M, p + V, 0, 1) \leq E \leq \delta_0^2, \quad \theta_M(0, 1) \leq (7/4) \Theta_{T_0\Omega}, \quad d_{\spt M}(0) \leq \delta_0, 
\end{equation}
for some $p \in V^\perp$, $V \in \cP_{T_0\Omega}$.

Then there are $V' \in \cP_{T_0\Omega}$, $p' \in {V'}^\perp$, satsifying
\begin{equation}\label{eqn:decay-concl1}
|p - p'| + |\pi_V - \pi_{V'}| \leq c_0 E^{1/2},
\end{equation}
so that
\begin{equation}\label{eqn:decay-concl2}
E_{\delta_0}(\Phi, M, p' + V', 0, \theta) \leq c_0 \theta^{2(e-1)} E .
\end{equation}
\end{prop}

By iterating Proposition \ref{prop:decay}, we obtain the direct Corollary.  As a technical aside, instead of using the monotonicity formula \eqref{eqn:monotonicity} to iterate Proposition \ref{prop:decay}, one could use Lemma \ref{lem:sharp-mass-bound}.
\begin{corollary}\label{cor:decay}
There are constants $\delta_1(\OmegaRef)$, $c_1(\OmegaRef)$, $\beta(\OmegaRef)$ so that the following holds.  Let $\Omega = \Phi(\Psi(\OmegaRef)) \cap B_1 \in \cD_{\delta_1}(\OmegaRef)$, $M \in \cIH_n(\Omega, B_1)$, and suppose that
\begin{equation}\label{eqn:con-decay-hyp}
E(\Phi, M, p + V, 0, 1) \leq E \leq \delta_1^2, \quad \theta_M(0, 1) \leq (3/2) \Theta_{T_0\Omega}, \quad d_{\spt M}(0) \leq \rho 
\end{equation}
for some $p \in V^\perp$, $V \in \cP_{T_0\Omega}$, and some $0 \leq \rho \leq 1/2$.

Then there are $V' \in \cP_{T_0\Omega}$,  $p' \in {V'}^\perp \cap B_1$, satisfying
\begin{equation}\label{eqn:con-decay-concl1}
|p - p'| + |\pi_V - \pi_{V'}| \leq (c_1 E)^{1/2},
\end{equation}
such that
\begin{gather}\label{eqn:con-decay-concl2}
E^{tot}(\Phi, M, p' + V', 0, r) \leq  c_1 E r^{2\beta}   \quad \forall \rho \leq r \leq 1/2.
\end{gather}
\end{corollary}

\begin{remark}
The exact form of the upper bound $\Theta_{T_0\Omega}$ is flexible.  One could equally well as use $\Theta_{\OmegaRef}$ or $\omega_{n+1}^{-1}\haus^{n+1}(\Omega)$, provided we take $\delta_1$ sufficiently small.  However in the following section we will find it convenient to use $T_0\Omega$.
\end{remark}

\begin{proof}[Proof of Corollary \ref{cor:decay}]
Choose $\theta(\OmegaRef)$ sufficiently small so that $c_0 \theta^{2(e-1)} \leq 1/2$.  Ensure $\delta_1(\OmegaRef)$ is sufficiently small so that
\[
\delta_1 \leq \min\{ \eps_{mn}, \eps_e \}, \quad (3/2)(1+c_{mn}\delta_1)^{n+1} \leq 7/4 \quad c_e \delta_1 \leq 1/4,
\]
and $\delta_1 \leq \delta_0^2$ with our choice of $\theta$.  Here $\eps_{mn}(n), c_{mn}(n)$ are the constants from Lemma \ref{lem:monotonicity}, $c_e(\OmegaRef, \theta = 1/2, A = 2), \eps_e(\OmegaRef)$ are from Corollary \ref{cor:height-w12-bound}, and $\delta_0(\OmegaRef, \theta), c_0(\OmegaRef)$ are from Proposition \ref{prop:decay}.

First suppose $\rho > \delta_1$.  Then take $p' = p$, $V' = V$, and we have trivially
\[
E(\Phi, M, p' + V', 0, r) \leq \delta_1^{-n-4} r^{2\alpha} E \quad \forall \rho \leq r \leq 1.
\]
Provided we take $c_1 \geq \delta_1^{-n-4}$, this proves the Theorem.

Let us therefore assume $\rho \leq \delta_1$.  Define $r_i = \theta^i$, and let $I$ be the maximal non-negative integer for which $\rho \leq r_I \delta_1$.  If $\rho = 0$, then we set $I = \infty \equiv I+1$.  For $i = 0, 1, \ldots, I+1$, we define inductively a sequence $p_i \in B_1$, $V_i \in \cP_{T_0\Omega}$, so that
\begin{equation}\label{eqn:con-decay-1}
r_i^{-1} |p_{i+1} - p_i| + |\pi_{V_{i+1}} - \pi_{V_i}| \leq c_0 2^{-i/2} \delta_0^{-1} E^{1/2},
\end{equation}
and
\begin{equation}\label{eqn:con-decay-2}
E_{\delta_0}(\Phi, M, p_i + V_i, 0, r_i) \leq 2^{-i} \delta_0^{-2} E.
\end{equation}

We take $p_0 = p$, $V_0 = V$.  Suppose, by inductive hypothesis, we have constructed $p_i$, $V_i$ satisfying \eqref{eqn:con-decay-2} and (if $i \geq 1$) \eqref{eqn:con-decay-1}.  By our choice of $I$ and $\delta_1$, and Remark \ref{rem:scaling-E}, and monotonicity \eqref{eqn:monotonicity} we can apply Proposition \ref{prop:decay} to the rescaled varifold $(\eta_{0, r_i})_\sharp M$ in $\Omega_{0, r_i}$ to obtain a $\tilde p_{i+1}$, $\tilde V_{i+1}$.  Setting $p_{i+1} = r_i \tilde p_{i+1}$ and $V_{i+1} = \tilde V_{i+1}$, the required estimates \eqref{eqn:con-decay-1}, \eqref{eqn:con-decay-2} hold by scaling.  This proves the inductive step, and therefore the existence of the required $p_i$, $V_i$.

If $I < \infty$, then set $p' = p_I$, $V' = V_I$.  Otherwise, if $I = \infty$, then observe that \eqref{eqn:con-decay-1} implies the $p_i$, $V_i$ form Cauchy sequences, and so we can take $p' = \lim_{i\to \infty} p_i$, $V' = \lim_{i \to \infty} V_i$.  For each (finite) $i \leq I$, we have 
\begin{equation}\label{eqn:con-decay-3}
r_i^{-1} |p' - p_i| + |\pi_{V'} - \pi_{V_i}| \leq c(\OmegaRef) 2^{-i/2} E^{1/2} .
\end{equation}
In particular taking $i = 0$ gives \eqref{eqn:con-decay-concl1}.

Given $\rho < r \leq 1$, either $r \geq \delta_1$, in which case set $i = 0$, or we have a maximal $i \leq I$ for which $r \leq \delta_1 r_i$.  Using \eqref{eqn:con-decay-3} and \eqref{eqn:upper-reg}, we compute
\begin{align*}
E(\Phi, M, p' + V', 0, r)
&\leq (\theta \delta_1)^{-n-2} E(\Phi, M, p' + V', 0, r_i) \\
&\leq c E(\Phi, M, p_i + V_i, 0, r_i) + c |\pi_{V'} - \pi_{V_i}|^2 + c r_i^{-2} |p' - p_i|^2 \\
&\leq c E 2^{-i} \\
&= c E r_i^{2\beta} \\
&\leq c(\OmegaRef) E r^{2\alpha},
\end{align*}
where $\beta = \log(1/2)/\log(\theta)$.  This proves the required $L^2$ decay of \eqref{eqn:con-decay-concl2}, for $\rho \leq r \leq 1$.  To deduce the decay of $E^{tot}$ we use Remark \ref{rem:scaling-E} to apply Corollary \ref{cor:height-w12-bound} at each scale $\rho \leq r \leq 1/2$.
\end{proof}

\vspace{5mm}

The rest of the section is devoted to proving Proposition \ref{prop:decay}.  We first require a definition and some helper theorems.

\begin{prop}[Fine graphical approximation]\label{prop:graph}
Given any $\tau, \beta$, there is a $\delta(\tau, \beta, \OmegaRef)$ and $c(\OmegaRef)$ so that the following holds.  Let $\Omega = \Phi(\Psi(\OmegaRef)) \cap B_1 \in \cD_\delta(\OmegaRef)$, $M \in \cIH_n(\Omega, B_1)$, satisfy
\begin{equation}\label{eqn:graph-hyp}
E(\Phi, M, p + V, 0, 1) \leq E \leq \delta^2, \quad \theta_M(0, 1) \leq (7/4) \Theta_{T_0\Omega}, \quad d_{\spt M}(0) \leq \delta,
\end{equation}
for some $p \in V^\perp$ and $V \in \cP_{T_0\Omega}$.  Then the following holds:
\begin{enumerate}
\item there is a domain $U \subset p+V$, and $C^{1}$ function $u : U \to V^\perp$, satisfying
\begin{equation}\label{eqn:graph-concl1}
(p+V) \cap\Omega \cap B_{1/4} \setminus B_{2\tau}(\del\Omega) \subset U, \quad r^{-1}|u| + |Du| \leq \beta,
\end{equation}
and
\begin{equation}\label{eqn:graph-concl1.5}
M \llcorner B_{1/2} \setminus B_\tau(\del\Omega) = [\graph(u)] ;
\end{equation}

\item we have the estimates
\begin{equation}
\sup_U |u|^2 + \int_U |Du|^2 \leq c E,
\end{equation}
and
\begin{equation}\label{eqn:graph-concl2.5}
\sup_{z \in \spt \mu_M \cap B_{1/2}} d(z, p + V)^2 \leq c E;
\end{equation}

\item for any $\zeta \in C^1_c(p + V, V^\perp) \cap C^1_c(B_{1/4})$, we have
\[
\left| \int_U Du \cdot D\zeta  \right| \leq c |\zeta|_{C^1} \left( ||H||_{L^\infty(B_1; \mu_M)} + |D^2\Phi|_{C^0(B_2)} + \tau^{1/2} E^{1/2} \right).
\]

\end{enumerate}
\end{prop}

\begin{proof}
Suppose, towards a contradiction, there is a sequence $\delta_i \to 0$, $p_i \in V_i^\perp$, $V_i \in \cP_{T_0\Omega_i}$, $\Omega_i = \Phi_i(\Psi_i(\OmegaRef)) \cap B_1 \in \cD_{\delta_i}(\OmegaRef)$, and $M_i \in \cIH_n(\Omega_i, B_1)$, satisfying \eqref{eqn:graph-hyp}, but failing conclusion 1.  By the height bound \eqref{eqn:sharp-height} and our assumption $\spt M_i \cap B_{\delta_i} \neq \emptyset$, we have $d(p_i + V_i, 0) \to 0$, and by Theorem \ref{thm:first-var} we have that $||\delta M_i||$ is uniformly bounded on compact subsets of $B_1$.  We can therefore find a $V \in \cP_{\OmegaRef}$, and an $M \in \cIV_n(B_1)$, so that after passing to a subsequence, we get $p_i \to 0$, $V_i \to V$, and $M_i \to M$ on compact subsets of $B_1$.  Since $|\Phi_i - \mathrm{Id}|_{C^2(B_2)} \to 0$, $|\Psi_i - \mathrm{Id}| \to 0$, we have $\Theta_{T_0\Omega_i} \to \Theta_{T_0\OmegaRef}$.

Since $E(\Phi_i, M, p_i + V_i, 0, 1) \to 0$, we have $\spt M \subset p + V$ and $||\delta M||(\setint \OmegaRef \cap B_1) = 0$.  Therefore by the constancy theorem $M \llcorner \setint\OmegaRef = k[p + V] \llcorner (\setint\OmegaRef \cap B_1)$ for some constant $k$.  By the non-concentration estimate \eqref{eqn:mass-non-conc} we have $\mu_M(\del\OmegaRef) = 0$, and trivially $\spt \mu_M \subset \OmegaRef$, and so in fact $M = k[p+V] \llcorner (\setint\OmegaRef\cap B_1)$.  From lower-regularity \eqref{eqn:lower-reg}, $\mu_{M_i}(B_{1/2}) \geq 1/c(\OmegaRef)$, and hence $\mu_M(B_1) \geq 1/c(\OmegaRef)$.  This implies $k \geq 1$.  On the other hand,
\[
\mu_M(B_1) \leq \liminf_i \mu_{M_i}(B_1) \leq (3/2) \lim_i \Theta_{T_0\Omega_i} = (3/2) \Theta_{T_0\OmegaRef},
\]
and so $k \leq 1$.  We deduce $k = 1$.  Since for every $W \subset\subset \setint\OmegaRef$ and $i >> 1$ the $M_i \llcorner W$ have bounded mean curvature (tending to zero) and zero boundary, Allard's theorem \cite{All} implies the convergence is $C^1$ on compact subsets of $\setint\OmegaRef \cap B_1$.  This proves conclusion 1.

Let us now fix a $\tau, U$, and prove conclusions 2 and 3.  Ensuring $\delta(\OmegaRef)$ is sufficiently small, There is no loss in assuming $p + V = \R^{n}\times\{0\}$.  Given a function $f : D \subset (p + V) \to V^\perp$, we can extend $f$ to be defined in $D \times V^\perp$ by setting $\tilde f(x) = f(p + \pi_V(x - p))$.  Given $f, g : D \to V^\perp$, then at $\mu_M$-a.e. $x \in \spt M \cap (D \times V^\perp)$ we have the bound
\begin{equation}\label{eqn:graph-1}
|\nabla \tilde f \cdot \nabla \tilde g - D\tilde f \cdot D \tilde g| \leq |\pi_{T_xM} - \pi_V|^2 |D\tilde f ||D\tilde g|.
\end{equation}
In the special case when $f = u$ as defined on $U$, then 
\begin{equation}\label{eqn:graph-2}
|\nabla \tilde u|^2 = |\nabla x_{n+1}|^2 \leq |\pi_{T_xM} - \pi_V|^2.
\end{equation}
Combining \eqref{eqn:graph-1}, \eqref{eqn:graph-2} with \eqref{eqn:W12-bound}, and (without loss of generality) ensuring $\beta \leq \beta(n)$, we get
\[
\int_U |Du|^2 \leq (1+c(n) \beta^2) \int_{B_{1/2}} |D \tilde u|^2 d\mu_M \leq c(n) \int_{B_{1/2}} |\pi_{T_zM} - \pi_V|^2 d\mu_M(z) \leq c(\OmegaRef) E.
\]
This completes the $W^{1,2}$ estimate of conclusion 2; the $L^\infty$ estimates both follow from the sharp height bound \eqref{eqn:sharp-height}.

We prove conclusion 3.  Take $\zeta \in C^1_c(p+V, V^\perp) \cap C^1_c(B_{1/4})$.  Let $\phi$ be any function which is $1$ on $B_{1/2}$, and is supported in $B_1$.  By the height bound \eqref{eqn:sharp-height} we can assume that $\spt M \cap \spt \tilde \zeta \subset B_{1/2}$.  Plugging the field $X = \tilde \zeta \phi \tilde e_{n+1}$ into the first variation (where $\tilde e_{n+1}$ as in the proof of Corollary \ref{cor:height-w12-bound}), and noting that $X$ is tangential to $\Omega$, we obtain
\begin{equation}\label{eqn:graph-3}
\left| \int_{B_{1/2}} \nabla \tilde \zeta \cdot \nabla x_{n+1} d\mu_M \right| \leq c(n)|\zeta|_{C^1}( ||H^{tan}_M||_{L^\infty(B_1; \mu_M)} + |D^2 \Phi|_{C^0(B_2)}).
\end{equation}

Let us make some remarks.  First, provided $\beta(n)$ is sufficiently small, the Jacobian $Ju$ of $u$ satisfies the bounds
\begin{equation}\label{eqn:graph-4}
1 \leq Ju \leq 1 + c(n) |Du|^2.
\end{equation}
Second, ensuring $\delta(\tau, \OmegaRef)$ is sufficiently small, by the non-concentration estimate \eqref{eqn:mass-non-conc} we have
\[
\mu_M(B_\tau(\del\Omega) \cap B_{1/2}) \leq c(\OmegaRef) \tau.
\]
Last, using \eqref{eqn:W12-bound}, this implies that
\begin{align}
\left| \int_{B_{1/2} \cap B_{\tau}(\del\Omega)} \nabla \tilde\zeta \cdot e_{n+1} d\mu_M \right|
&\leq c(\OmegaRef) |\zeta|_{C^1} \tau^{1/2} \left( \int_{B_{1/2}} |\pi_{T_zM} - \pi_V|^2 d\mu_M(z) \right)^{1/2} \nonumber \\
&\leq c(\OmegaRef) |\zeta|_{C^1} \tau^{1/2} E^{1/2} \label{eqn:graph-5}
\end{align}

Using conclusions 1 and 2, equations \eqref{eqn:graph-3}, \eqref{eqn:graph-4}, \eqref{eqn:graph-5} we obtain
\begin{align*}
\left| \int_U D \zeta \cdot Du \right| 
&\leq c \beta |\zeta|_{C^1} \int_U |Du|^2 + \left| \int_{B_{1/2} \setminus B_\tau(\del\Omega)} D\tilde \zeta \cdot D\tilde u d\mu_M \right| \\
&\leq c |\zeta|_{C^1} \int_U |Du|^2 + c|\zeta|_{C^1} \int_{B_{1/2}} |\pi_{T_zM} - \pi_V|^2 d\mu_M(z) + \left| \int_{B_{1/2}\setminus B_\tau(\del\Omega)} \nabla \tilde \zeta \cdot e_{n+1} d\mu_M \right| \\
&\leq c |\zeta|_{C^1} E+ c |\zeta|_{C^1} \tau^{1/2} E^{1/2} + c |\zeta|_{C^1} (||H||_{L^\infty(B_1; \mu_M)} + |D^2\Phi|_{C^0(B_2)}) \\
&\leq c|\zeta|_{C^1} \tau^{1/2} E^{1/2} + c|\zeta|_{C^1} (||H||_{L^\infty(B_1; \mu_M)} + |D^2\Phi|_{C^0(B_2)}),
\end{align*}
for $c = c(\OmegaRef)$.  The last inequality follows because we can of course assume $\delta^2 \leq \tau$.  This completes conclusion 3, and the proof of Proposition \ref{prop:graph}.
\end{proof}

\begin{definition}
Consider the sequences $\beta_i, \delta_i \in \R$, $\Omega_i = \Phi_i(\Psi_i(\OmegaRef)) \cap B_1 \in \cD_{\delta_i}(\OmegaRef)$, $V_i \in \cP_{T_0\Omega_i}$, $p_i \in V_i^\perp$, $M_i \in \cIH_n(\Omega, B_1)$.  We say $(\OmegaRef, \Omega_i, M_i, p_i + V_i, \beta_i, \delta_i)$ is a blow-up sequence if:
\begin{enumerate}
\item $p_i \to 0$, $V_i \to \R^n \times \{0\}$, $\beta_i \to 0$, $\delta_i \to 0$,
\item $\mu_{M_i}(B_1) \leq (7/4) \Theta_{T_0\Omega_i}$, and $d_{\spt M_i}(0) \to 0$, 
\item $\limsup_i \beta_i^{-2} E_{\delta_i}(\Phi_i, M_i, p_i + V_i, 0, 1) < \infty$.
\end{enumerate}
\end{definition}

\begin{prop}\label{prop:blow-up}
Let $(\OmegaRef,  \Omega_i, M_i, p_i + V_i, \beta_i, \delta_i)$ be a blow-up sequence.  Let us write $p_i + V_i = \graph_{\R^n \times \{0\}}(q_i + \phi_i \cdot x)$, for $q_i \in \R$ and $\phi_i \in \{0^l\}\times \R^{m-1}$.  Then there is a $W^{1,2}$ function $v : (\OmegaRefN \times \R^{m-1}) \cap B_{1/4} \to \R$, such that:
\begin{enumerate}
\item $v$ is weakly harmonic with Neumann boundary:
\begin{equation}\label{eqn:blow-up-concl1}
\int_{(\OmegaRefN \times \R^{m-1}) \cap B_{1/4}} Dv \cdot D\zeta = 0 \quad \forall \zeta \in C^1_c(B_{1/4} \cap (\R^n \times \{0\})) ;
\end{equation}

\item $v$ has the $W^{1,2}$ bound
\begin{equation}\label{eqn:blow-up-concl2}
\int_{(\OmegaRefN \times \R^{m-1}) \cap B_{1/4}} v^2 + |Dv|^2 \leq c(\Omega) \left( \limsup_i \beta_i^{-2} E_{\delta_i}(\Phi_i, M_i, p_i + V_i, 0, 1) \right) ;
\end{equation}

\item on any compact set $U \subset\subset B_{1/4} \cap (\OmegaRefN \times \R^{m-1}) \setminus \del\OmegaRef$, we have $L^2$ convergence
\begin{equation}\label{eqn:blow-up-concl3}
\beta_i^{-1} u_i( x + q_i + \phi_i \cdot x) \to v(x);
\end{equation}

\item for any $\rho \leq 1/4$, we have the $L^2$ convergence
\begin{equation}\label{eqn:blow-up-concl4}
\beta_i^{-2} \int_{B_\rho} d_{p_i + V_i}^2 d\mu_M \to \int_{B_\rho \cap (\OmegaRefN \times \R^{m-1})} v^2 .
\end{equation}
\end{enumerate}
\end{prop}

\begin{proof}
Let $\tau_i \to 0$ sufficiently slowly, so that for each $i$ large we can apply Proposition \ref{prop:graph} to to deduce
\[
M_i \cap B_{1/2} \setminus B_{\tau_i}(\del\OmegaRef) = \graph_{p_i + V_i}(u_i),
\]
where $u_i$ is defined on some domain $U_i$ satisfying
\[
(p_i + V_i) \cap B_{1/4} \setminus B_{2\tau_i}(\del\OmegaRef) \subset U_i \subset p_i + V_i.
\]
Write $E_i = E_{\delta_i}(\Phi_i, M_i, p_i + V_i, 0, 1)$.

Fix any $U \subset\subset B_{1/4} \cap (\OmegaRefN \times \R^{m-1}) \setminus \del\OmegaRef$.  Then for sufficiently large $i$, $x\mapsto w_i(x) := u_i(x + q_i + \phi_i \cdot x)$ is well-defined, and parameterizes a subset of $\spt M_i$.  For $x \in U$, we have
\[
Dw_i(x) = (1+o(1)) Du_i (x + q_i + \phi_i \cdot x), \quad 1 \leq J(x) \leq 1 + o(1),
\]
where $J(x)$ is the Jacobian of $x \mapsto x + q_i + \phi_i \cdot x$.  It follows from Proposition \ref{prop:graph} that
\begin{equation}\label{eqn:blow-up-bd}
\int_U |w_i|^2 + |Dw_i|^2 \leq c E_i, \quad \sup_U |w_i|^2 \leq cE_i, 
\end{equation}
and for any $\zeta \in C^1_c(B_{1/4})$, 
\begin{equation}\label{eqn:blow-up-eqn}
\left| \int_U D\zeta \cdot Dw_i \right| \leq o(1)  |\zeta|_{C^1} E_i^{1/2}.
\end{equation}

If $v_i = \beta_i^{-1} w_i$, then \eqref{eqn:blow-up-bd} (and our definition of blow-up sequence) shows that the $v_i$ are bounded in $W^{1,2}(U)$, with a bound independent of either $i$ or $U$.  A diagonalization argument implies there is a $v \in W^{1,2}(B_{1/4} \cap (\OmegaRefN \times \R^{m-1}))$, satisfying the bound \eqref{eqn:blow-up-concl2}, so that for every $U$ as above $v_i \to v$ strongly in $L^2(U)$ and weakly in $W^{1,2}(U)$.

From \eqref{eqn:blow-up-bd} we have
\begin{equation}\label{eqn:blow-up-height}
\sup |v|^2 \leq c(\OmegaRef), 
\end{equation}
and from \eqref{eqn:blow-up-eqn} we get that $v$ satisfies the required \eqref{eqn:blow-up-concl1}.  The strong $L^2$ convergence \eqref{eqn:blow-up-concl4} follows from the $L^\infty$ bounds \eqref{eqn:blow-up-height}, \eqref{eqn:graph-concl2.5}, \eqref{eqn:mass-non-conc} and the fact
\[
\int_{B_{\rho} \setminus B_{\tau_i}(\del\OmegaRef)} d_{p_i + V_i}^2 d\mu_M = (1+o(1))\int_{U_i} |u_i|^2. \qedhere
\]
\end{proof}

\begin{definition}
Let us call any $v$ as obtained in Proposition \ref{prop:blow-up} a \emph{Jacobi field} on $\OmegaRefN \times \R^{m-1}$.
\end{definition}

We are now ready to prove Proposition \ref{prop:decay}.
\begin{proof}[Proof of Proposition \ref{prop:decay}]
Suppose, towards a contradiction, there is a sequence $\delta_i \to 0$, $V_i \in \cP_{T_0\Omega_i}$, $p_i \in V_i^\perp$, $\Omega_i = \Phi_i(\Psi_i(\OmegaRef)) \in \cD_{\delta_i}(\OmegaRef)$, and $M_i \in \cIH_n(\Omega_i, B_1)$, such that
\[
E_i:= E_{\delta_i}(\Phi_i, M_i, p_i + V_i, 0, 1) \leq \delta_i^2, \quad \mu_{M_i}(B_1) \leq (7/4)\Theta_{T_0\Omega_i}, \quad \spt M_i \cap B_{\delta_i} \neq \emptyset,
\]
but for which
\[
E_{\delta_i}(\Phi_i, M_i, p' + V', 0, \theta) \geq c_0 \theta^{2(e-1)} E_i,
\]
for all $p' + V'$ satisfying \eqref{eqn:decay-concl1}.  Here $c_0$ is a constant depending only on $\OmegaRef$ that we will choose later.

Let $\beta_i^2 = E_{\delta_i}(\Phi_i, M_i, p_i + V_i, 0, 1)$, so of course $\beta_i \to 0$ also.  From the height bound \eqref{eqn:sharp-height}, we can assume $p_i \to 0$, and after passing to a subsequence and rotating as necessary, we can assume $V_i \to \R^n \times \{0\}$.  Then $(\OmegaRef, \Omega_i, M_i, p_i + V_i, \beta_i, \delta_i)$ is a blow-up sequence, and we can apply Proposition \ref{prop:blow-up} to obtain a Jacobi field $v$, satisfying
\[
\int_{\OmegaRefN \times \R^{m-1}} Dv \cdot D\zeta = 0 \quad \forall \zeta \in C^1_c(B_{1/4} \cap (\R^n \times \{0\})),
\]
and the $W^{1,2}$ estimate
\[
\int_{(\OmegaRefN \times \R^{m-1}) \cap B_{1/4}} |v|^2 + |Dv|^2 \leq c(\OmegaRef).
\]
From Proposition \ref{prop:expansion}, we can expand in $W^{1,2}$
\begin{equation}\label{eqn:decay-jacobi-exp}
v(x = r\omega) = b + A \cdot x + \sum_i a_i r^{\alpha_i} \phi_i(\omega), 
\end{equation}
for some $A \in \{0^l\}\times \R^{m-1}$, and $\alpha_i \geq e > 1$, for $e(\OmegaRefN)$ as in \ref{prop:expansion}.  Using the $L^2( (\OmegaRefN \times \R^{m-1}) \cap S^{n-1})$-orthogonality of this expansion for every fixed $r$, we get
\begin{equation}\label{eqn:decay-jacobi-bd}
|b|^2 + |A|^2 + \sum_i \frac{a_i^2 (1/4)^{2\alpha_i + n}}{2\alpha_i + n} \leq c(\OmegaRef).
\end{equation}

Let us write
\[
p_i + V_i = \graph_{\R^n\times \{0\}} ( q_i + \phi_i \cdot x),
\]
for $q_i \in \R$, and $\phi_i \in \{0^l\}\times \R^{m-1}$.  Now define the new affine planes
\[
p_i' + V_i' = \graph_{\R^n \times \{0\}} (q_i + \beta_i b + (\phi_i + \beta_i A) \cdot x).
\]
Since $A \in \{0^l\}\times \R^{m-1}$, we have $V_i' \in \cP_\OmegaRef$, and we can take $p_i' \in {V_i'}^\perp$.  By definition of the $p_i'$, $V_i'$, and by considering the analytic maps
\[
\phi \mapsto \pi_\phi := \pi_{\graph_{\R^n\times\{0\}}(\phi \cdot x)}, \quad (q, \phi) \mapsto \pi^\perp_\phi(q e_{n+1})
\]
we have
\begin{equation}
|p_i - p_i'| + |\pi_{V_i} - \pi_{V_i'}| \leq c(n) \beta_i (|b| + |A|) \leq c_{d1}(\OmegaRef) \beta_i.
\end{equation}
Therefore, $(\OmegaRef, \Omega_i, M_i, p_i' + V_i', \beta_i, \delta_i)$ is a blow-up sequence also, and we can again use Proposition \ref{prop:blow-up} to obtain a new Jacobi field $v'$ (with, a priori, a slightly worse $W^{1,2}$ bound than $v$). 

There is a sequence $\tau_i \to 0$, so that we can write
\begin{equation}\label{eqn:decay-M-graph}
M_i \cap B_{1/2} \setminus B_{\tau_i}(\del\OmegaRef) = \graph_{p_i + V_i}(u_i) = \graph_{p_i' + V_i'}(u_i'),
\end{equation}
where, for any compact subset $U  \subset\subset B_{1/4} \cap (\mathrm{int} \OmegaRefN \times \R^{m-1})$, we have
\begin{equation}\label{eqn:decay-M-C1}
|u_i(x + q_i + \phi_i \cdot x)|_{C^1(U)} + |u_i'(x + q_i + \beta_i b + (\phi_i + \beta_i A)\cdot x)|_{C^1(U)} \to 0,
\end{equation}
and (from Proposition \ref{prop:blow-up}.3)
\begin{equation}\label{eqn:decay-u-lim}
\beta_i^{-1} u_i(x + q_i + \phi_i \cdot x) \to v, \quad \beta_i^{-1} u_i'(x + q_i + \beta_i p + (\phi_i + \beta_i A) \cdot x) \to v'
\end{equation}
in $L^2(U)$.

From \eqref{eqn:decay-M-graph}, \eqref{eqn:decay-M-C1} and since $|p_i| + |\phi_i| \to 0$, we have that for every $x \in U$, 
\begin{equation}\label{eqn:decay-u-u'}
|u_i'( x + q_i + \beta_i p + (\phi_i + \beta_i A) \cdot x) - u_i(x + q_i + \phi_i \cdot x) - \beta_i p - \beta_i  A\cdot x| \leq o(1) \beta_i,
\end{equation}
where $o(1) \to 0$ as $i \to \infty$.  By \eqref{eqn:decay-u-lim} and \eqref{eqn:decay-u-u'}, we deduce that
\[
v'(x) = v(x) - b - A \cdot x = \sum_i a_i r^{\alpha_i} \phi_i(x),
\]
where $a_i$ as in equation \eqref{eqn:decay-jacobi-exp}.

Therefore, using \eqref{eqn:decay-jacobi-bd}, $4 \theta \leq 1$ and $|\alpha_i| \geq 1+ \alpha$, we have:
\[
\int_{B_\theta \cap (\OmegaRefN \times \R^{m-1})} |v'|^2 = \sum_i \frac{a_i^2 \theta^{2\alpha_i + n}}{2\alpha_i + n} \leq c_{d2}(\OmegaRef)(4\theta)^{n+2e}.
\]
Provided we take $c_0$ larger than $c_{d1}$ and $8^{4+n} c_{d2}$, then by the strong $L^2$ convergence
\[
\beta_i^{-2} \int_{B_\theta} d_{p_i' + V_i'}^2 d\mu_M \to \int_{B_\theta \cap (\OmegaRefN \times \R^{m-1})} |v'|^2,
\]
we obtain a contradiction (we remind the reader that the $H$ and $\Phi$ terms of $E$ decay gratuitously).  This completes the proof of Proposition \ref{prop:decay}.
\end{proof}

\section{Regularity}\label{sec:reg}

The main theorem of this section is the following decay estimate, which we shall prove by induction on the boundary strata $\del_i \Omega$.  The idea is that we can use the decay of Proposition \ref{prop:decay} (or rather Corollary \ref{cor:decay}) whenever $\Omega$ resembles a polyhedral cone.  If we hit a radius at which $\Omega$ stops resembling a cone, then by recentering on a lower strata and dropping a controllable number of scales, we will start looking like a cone with an extra degree of symmetry.

Although the $L^2$, $W^{1,2}$, and $L^\infty$ distances to planes are all effectively comparable when $\Omega$ resembles a cone (and the plane lies in the ``good'' space $\cP_{T_0\Omega}$), as we traverse scales and cone-types it will be convenient to prove a decay on the $L^2$, $W^{1,2}$, and $L^\infty$ excesses simultaneously.  The proof is no more involved than proving a decay on the $L^2$ by itself.

As in the previous sections we fix here a polyhedral cone domain $\OmegaRef = \OmegaRefN^l \times \R^{m}$, having $m \geq 1$.

\begin{theorem}\label{thm:total-decay}
There are constants $c_3(\OmegaRef)$, $\delta_3(\OmegaRef)$, $\alpha(\OmegaRef)$ so that the following holds.  Let $\Omega = \Phi(\Psi(\OmegaRef)) \cap B_1 \in \cD_{\delta_3}(\OmegaRef)$, $M \in \cIH_n(\Omega, B_1)$ satisfy
\begin{equation}\label{eqn:total-decay-hyp}
E(\Phi, M, p + V, 0, 1) \leq E \leq \delta_3^2, \quad \theta_M(0, 1) \leq (3/2) \Theta_{T_0\Omega},
\end{equation}
for some $p \in V^\perp$, $V \in \cP_{T_0\Omega}$.

Then for every $x \in \spt M \cap B_{1/16}$, there is plane $V_x \in \cP_{T_x\Omega}$ so that
\begin{equation}\label{eqn:total-decay-concl1}
|\pi_V - \pi_{V_x}| \leq c_3 E^{1/2},
\end{equation}
and for all $0 < r < 1/4$:
\begin{equation}\label{eqn:total-decay-concl2}
E^{tot}(\Phi, M, x + V_x, x, r) \leq c_3 r^{2\alpha} E
\end{equation}
\end{theorem}

We require first a helper lemma.
\begin{lemma}\label{lem:sharp-mass-bound}
Given any $\eps > 0$, there is a $\delta_4(\OmegaRef, \eps)$ so that the following holds.  Let $\Omega = \Phi(\Psi(\OmegaRef)) \cap B_1 \in \cD_{\delta_4}(\OmegaRef)$,  $M \in \cIH_n(\Omega, B_1)$ satisfy
\begin{equation}\label{eqn:sharp-mass-hyp}
E(\Phi, M, p + V, 0, 1) \leq \delta_4^2, \quad \theta_M(0, 1) \leq (7/4) \Theta_{T_0\Omega},
\end{equation}
for some $p \in V^\perp$, and some $V \in \cP_{T_0\Omega}$.  Then given any $x \in \Omega \cap B_1$ and $r \geq \eps$, so that $B_r(x) \subset B_{1-\eps}$, we have
\begin{equation}\label{eqn:sharp-mass-concl}
\theta_M(x, r) \leq (3/2) \Theta_{T_x\Omega} .
\end{equation}
\end{lemma}

\begin{proof}
Proof by contradiction.  Suppose otherwise: there is a sequence $\delta_i \to 0$, domains $\Omega_i= \Phi_i(\Psi_i(\OmegaRef)) \cap B_1 \in \cD_{\delta_i}(\OmegaRef)$, and varifolds $M_i \in \cIH_n(\Omega_i, B_1)$, so that
\[
E(\Phi_i, M_i, p_i + V_i, 0, 1) \to 0, \quad \theta_{M_i}(0, 1) \leq (7/4) \Theta_{T_0\Omega_i},
\]
for some sequence $V_i \in \cP_{T_0\Omega_i}$, $p_i \in V_i^\perp$, but there are points $x_i \in \Omega_i \cap B_{1}$, and radii $r_i \geq \eps$ such that
\[
B_{r_i}(x_i) \subset B_{1-\eps}, \quad \theta_{M_i}(x_i, r_i) \geq (3/2) \Theta_{T_{x_i}\Omega_i}.
\]

By the height bound \eqref{eqn:sharp-height} there is no loss in assuming the $p_i$ are bounded, as otherwise we would have $\spt M_i \cap B_1 = \emptyset$ for large $i$.  We can therefore pass to a subsequence, and get convergence $p_i \to p' \in \R^{n+1}$, $V_i \to V' \in \cP_{T_0\OmegaRef}$, $x_i \to x \in B_1$, $r_i \to r \geq \eps$, so that $B_r(x) \subset B_{1-\eps}$, and and $M_i \to M'$, for some $M' \in \cIV_n(B_1)$.  As in the proof of Proposition \ref{prop:graph}, we must have $M' = [p' + V'] \llcorner (\setint \OmegaRef \cap B_1)$.

For a.e. $1-|x| > \rho > r$ we have by the lower-semi-continuity \eqref{eqn:theta-lsc-omega}:
\[
\mu_M(B_\rho(x)) = \lim_i \mu_{M_i}(B_\rho(x)) \geq \limsup_i \mu_{M_i}(B_{r_i}(x_i)) \geq (3/2)\Theta_{T_x\OmegaRef} \omega_n r^n.
\]
On the other hand, since $V \in \cP_\OmegaRef$, we can use monotonicity \eqref{eqn:theta-mono} to get
\[
\mu_M(B_\rho(x)) = \haus^n( (p + V) \cap \OmegaRef \cap B_\rho(x)) \leq \haus^n( (x+V) \cap \OmegaRef \cap B_\rho(x)) \leq \Theta_{T_x\OmegaRef} \omega_n \rho^n,
\]
which is a contradiction for $\rho$ sufficiently close to $r$.
\end{proof}

\begin{proof}[Proof of Theorem \ref{thm:total-decay}]

There is no loss in assuming $p \in B_1$.  Let $B(\OmegaRef)$ be as in Lemma \ref{lem:B}.  Observe that since the set $\{T_z\OmegaRef \}_z$ is finite, any constant that depends on some $T_z\OmegaRef$ can be made to depend only on $\OmegaRef$.  In particular, let us choose $\alpha$ by setting
\[
\alpha = \min\{ \beta(T_z\OmegaRef) \}_z
\]
where $\beta$ as in Corollary \ref{cor:decay}.  In the following $c$ will denote a generic constant $\geq 1$ depending only on $\OmegaRef$, which may increase from line to line.

Take $x \in \del_{m+i}\Omega \cap B_{1/16}$ (note $i \geq 0$).  Let $x_i = x$, and then for $j = i, \ldots, 1$ define $x_{j-1} \in \del_{m+j-1} \Omega$ to be the point realizing $d(x_j, \del_{m+j-1}\Omega)$.  Let us formally define $(B/2)|x_0 - x_{-1}| = 1/8$.

Let $i_0 = i$.  If $i = 0$ let $J = 0$.  Otherwise, define inductively $j = 0, 1, 2, \ldots, J$ by the conditions that $0 \leq i_{j+1} < i_j$, and
\[
|x_{i_j} - x_{i_{j+1}}| + B|x_{i_j} - x_{i_j - 1}| \leq (B/2) |x_{i_{j+1}} - x_{i_{j+1} - 1}|,
\]
but
\[
|x_{i_j} - x_k| + B|x_{i_j} - x_{i_j - 1}| > (B/2) |x_k - x_{k-1}| \quad \forall k = i_{j+1}+1, \ldots, i_j - 1.
\]
Note that since $|x| \leq 1/16$, $0 \in \del_m\Omega$, and by Lemma \ref{lem:B}(\ref{item:poly1}) we have (taking $\delta_3(n)$ small)
\[
|x_{i_j} - x_{k}| + B|x_{i_j} - x_{i_j - 1}| \leq (3/2) |x_{i_j}| \leq (3/2)(1+c_B\delta_3)|x| \leq 1/8, 
\]
for every $k \leq i_j - 1$, and in particular we have $i_J = 0$.

Let $r_0^- = 0$.  Define
\begin{align*}
r_j^+ &=  B |x_{i_j} - x_{i_j - 1}| \quad (j = 0, \ldots, J), \\
r_j^- &= 2 r^+_{j-1} + 2 |x_{i_{j-1}} - x_{i_j}| \quad (j = 1, \ldots, J).
\end{align*}
Note that since $i_J = 0$, we have $r_J^+ = 1/4$.  By construction we have the inclusions
\begin{equation}\label{eqn:reg-11}
x \in B_{r_{j-1}^-}(x_{i_{j-1}}) \subset B_{r_{j-1}^+}(x_{i_{j-1}}) \subset B_{r_j^-/2}(x_{i_j}) \subset B_1(0).
\end{equation}

Now for $j = 0, 1, \ldots, J-1$ we have
\begin{align*}
|x_{i_j} - x_{i_{j+1}}|
&\leq |x_{i_j} - x_{i_{j+1}+1}| + |x_{i_{j+1}+1} - x_{i_{j+1}}| \\
&\leq |x_{i_j} - x_{i_{j+1}+1}| + (2/B) ( |x_{i_j} - x_{i_{j+1}+1}| + B|x_{i_j} - x_{i_j - 1}|) \\
&\leq 2|x_{i_j} - x_{i_j - 1}| + (1+2/B)|x_{i_j} - x_{i_{j+1}+1}| \\ 
&\leq 2(1 + (1+2/B))|x_{i_j} - x_{i_j - 1}| + (1+2/B)^2 |x_{i_j} - x_{i_{j+1}+2}| \\
&\leq \ldots \\
&\leq c(n, B) |x_{i_j} - x_{i_j - 1}|.
\end{align*}
Therefore
\begin{align}
\frac{r_j^+}{r_{j+1}^-} 
&= \frac{ B |x_{i_{j}} - x_{i_j - 1}|}{2 |x_{i_j} - x_{i_{j+1}}| + 2B|x_{i_j} - x_{i_j - 1}|} \nonumber \\
&= \frac{1}{(2/B) \frac{|x_{i_j} - x_{i_{j+1}}|}{|x_{i_j} - x_{i_j - 1}|} + 1}  \nonumber \\
&\geq \frac{1}{c(\OmegaRef)}. \label{eqn:reg-8}
\end{align}

\vspace{5mm}

For ease of notation let us set $\tilde x_j = x_{i_j}$.  Then we shall prove for $j = J, J-1, \ldots, 0$ the following statement, which we call $(\dagger_j)$: There is a $\Lambda_j(\OmegaRef)$, $V_j \in T_{\tilde x_j}\Omega$, and $p_j \in V_j^\perp$ so that
\begin{equation}\label{eqn:reg-1}
E^{tot}(M, p_j + V_j, \tilde x_j, r) \leq \Lambda_j r^{2\alpha} E \quad \forall r_j^- \leq r \leq 1/4,
\end{equation}
and
\begin{equation}\label{eqn:reg-2}
\theta_M(\tilde x_j, r) \leq (7/4) \Theta_{T_{\tilde x_j}\Omega} \quad \forall r_j^- \leq r \leq r_j^+
\end{equation}
and
\begin{equation}\label{eqn:reg-3}
|p_j - p| + |\pi_{V_j} - \pi_{V}| \leq \Lambda_j^{1/2} E^{1/2} (r_j^+)^\alpha.
\end{equation}

\vspace{5mm}

Observe that \eqref{eqn:reg-1} implies that
\[
d(x, p_j + V_j) \leq c (r_j^-)^\alpha
\]
and hence when $j = 0$, we have $p_0 + V_0 = x + V_0$.  Since $x = \tilde x_0$, we have $V_0 \in T_x\Omega$, and \eqref{eqn:reg-3} implies that
\[
|\pi_{V_x} - \pi_V| \leq c(\OmegaRef) E^{1/2}.
\]
Therefore when $j = 0$, $(\dagger_0)$ implies our Theorem.

\vspace{5mm} 

To prove $(\dagger_j)$ we induct downwards on $j$.  If $j = J$, then let us set $\tilde x_{J+1} = 0$, $r_{J+1}^- = 1/4$, $\Lambda_{J+1} = 4^{n+4}$, $p_{J+1} = p$, $V_{J+1} = V$, and proceed as below.  Otherwise, let us assume by hypothesis that $(\dagger_{j+1})$ holds.  We prove $(\dagger_j)$.

We have
\begin{equation}\label{eqn:reg-14}
E^{tot}(\Phi, M, p_{j+1} + V_{j+1}, \tilde x_{j+1}, r) \leq \Lambda_{j+1} r^{2\alpha} E
\end{equation}
for all $r_{j+1}^- \leq r \leq 1/4$, and
\[
\theta_M(\tilde x_{j+1}, r_{j+1}^-) \leq (7/4)\Theta_{T_{\tilde x_{j+1}}\Omega}.
\]

Since $B_{r_j^+}(\tilde x_j) \subset B_{r_{j+1}^-/2}(\tilde x_{j+1})$ and $r_j^+ / r_{j+1}^- \geq 1/c(\OmegaRef)$, we can use \eqref{eqn:reg-5} and apply Lemma \ref{lem:sharp-mass-bound} to deduce
\[
\theta_M(\tilde x_j, r_j^+) \leq (3/2)\Theta_{T_{\tilde x_j}\Omega},
\]
and hence by monotonicity \eqref{eqn:monotonicity} we get
\begin{equation}\label{eqn:reg-4}
\theta_M(\tilde x_j, r) \leq (7/4) \Theta_{T_{\tilde x_j}\Omega} \quad \forall r \leq r_j^+.
\end{equation}

\eqref{eqn:reg-14} implies we can find a $q_j' \in p_{j+1} + V_{j+1}$ so that
\[
|x - q_j'| \leq \Lambda_{j+1}^{1/2} E^{1/2} (r_{j+1}^-)^{1+\alpha} \leq r_{j+1}^-,
\]
for $\delta_3(\Lambda_{j+1}, \OmegaRef)$ sufficiently small.  In particular, we have $|\tilde x_j - q_j'| \leq c r_{j}^+$.  As per Lemma \ref{lem:B}, choose $V_j' \in \cP_{T_{\tilde x_j}\Omega}$ such that
\begin{equation}\label{eqn:reg-6}
|\pi_{V_{j+1}} - \pi_{V_j'}| \leq c(\OmegaRef) E^{1/2} |\tilde x_j - \tilde x_{j+1}| \leq c  E^{1/2} r_j^+.
\end{equation}
Then, using \eqref{eqn:reg-11} and \eqref{eqn:reg-8}, we have
\begin{align}
&E^\infty (M, q_j' + V_j', \tilde x_j, r_j^+) \\
&= (r_j^+)^{-2} \sup_{z \in B_{r_j^+}(\tilde x_j)} |\pi_{V_j'}^\perp(z - q_j')|^2 \\
&= c |\pi_{V_j'} - \pi_{V_{j+1}}|^2 + c E^\infty(\Phi, M, p_{j+1} + V_{j+1}, \tilde x_{j+1}, r_{j+1}^-)  \\
&\leq c(1+\Lambda_{j+1}) (r_j^+)^{2\alpha} E,
\end{align}
and similarly
\begin{align}
E^W(M, V_j', \tilde x_j, r_j^+) + E(\Phi, M, q_j' + V_j', \tilde x_j, r_j^+) \leq c(1+\Lambda_{j+1}) (r_j^+)^{2\alpha} E.
\end{align}
By Lemma \ref{lem:B}, there is a $z \in \OmegaRef$, a linear isomorphism $\beta : \R^{n+1} \to \R^{n+1}$, and a $C^2$ mapping $\alpha : B_2 \to \R^{n+1}$ such that $\Omega_{\tilde x_j, r_j^+} \in \cD_{c \delta_3}(T_x\OmegaRef)$ and
\begin{align}
E^{tot}(\alpha, (\eta_{\tilde x_j, r_j^+})_\sharp M, \eta_{\tilde x_j, r_j^+}(q_j'), V_j', 0, 1) 
&\leq 2 E^{tot}(\Phi, M, q_j' + V_j', \tilde x_j, r_j^+) \nonumber \\
&\leq c (1+\Lambda_{j+1}) (r_j^+)^{2\alpha} E. \label{eqn:reg-5}
\end{align}

By \eqref{eqn:reg-4} and \eqref{eqn:reg-5}, provided $\delta_3(\OmegaRef, \Lambda_{j+1})$ is sufficiently small, we can apply Corollary \ref{cor:decay} to deduce: there is a $V_j \in T_{\tilde x_j}\Omega$, $p_j \in V_j^\perp$, so that
\begin{equation}\label{eqn:reg-15}
E^{tot}(\Phi, M, p_j + V_j, \tilde x_j, r) \leq c(1+\Lambda_{j+1}) r^{2\alpha} E \quad \forall r_j^- \leq r \leq r_j^+,
\end{equation}
(recall that $|x - \tilde x_j| \leq r_j^-$) and
\begin{equation}\label{eqn:reg-7}
|\pi_{V_j} - \pi_{V_j'}| \leq \left( c(1+\Lambda_{j+1}) E \right)^{1/2} (r_j^+)^\alpha.
\end{equation}

\vspace{5mm}

Combining \eqref{eqn:reg-6}, \eqref{eqn:reg-7}  we get
\begin{equation}\label{eqn:reg-9}
|\pi_{V_j} - \pi_{V_{j+1}}| \leq ( c(1+\Lambda_{j+1}))^{1/2} E^{1/2} (r_j^+)^\alpha, 
\end{equation}
which with \eqref{eqn:reg-3} implies
\begin{equation}\label{eqn:reg-12}
|\pi_{V_j} - \pi_V| \leq ( c(1+\Lambda_{j+1}))^{1/2} E^{1/2}.
\end{equation}
On the other hand, we can estimate
\begin{align}
|p_j - p| 
&\leq |\pi_{V_j}^\perp(x - p_j)| + |\pi_{V}^\perp(x - p)| + |\pi_{V_j} - \pi_{V}| \nonumber \\
&\leq r_j^+ E(M, \Phi, p_j + V_j, \tilde x_j, r_j)^{1/2} + E(M, \Phi, p + V, 0, 1)^{1/2} \nonumber \\
&\quad + (c (1+\Lambda_{j+1}))^{1/2} E^{1/2}  \nonumber \\
&\leq (c (1+\Lambda_{j+1}))^{1/2} E^{1/2} .\label{eqn:reg-13}
\end{align}

Finally, we must show decay \eqref{eqn:reg-15} for $r \geq r_j^+$.  \eqref{eqn:reg-15} implies that we can find a $q_j \in p_j + V_j$ such that
\[
|x - q_j| \leq (c(1+\Lambda_{j+1}))^{1/2} E^{1/2} (r_j^+)^{1+\alpha} \leq r_j^+
\]
and hence
\begin{equation}\label{eqn:reg-10}
|q_j - q_j'| \leq (c(1+\Lambda_{j+1}))^{1/2} E^{1/2} (r_j^+)^{1+\alpha}, \quad \text{ and } \quad |\tilde x_j - q_j| \leq c r_j^+ .
\end{equation}
Therefore, using \eqref{eqn:reg-9}, \eqref{eqn:reg-10}, \eqref{eqn:reg-11}, \eqref{eqn:reg-8}, we have for $r_j^+ \leq r \leq (1/4) (r_j^+/r_{j+1}^-)$:
\begin{align*}
&E^\infty(\Phi, M, p_j + V_j \equiv q_j + V_j, \tilde x_j, r) \\
&= r^{-2} \sup_{z \in B_r(\tilde x_j) \cap \spt M} |\pi_{V_j}(z - q_j)|^2 \\
&\leq c |\pi_{V_j} - \pi_{V_{j+1}}|^2 + r^{-2} |q_j - q_j'| + c E^\infty(\Phi, M, q_j' + V_{j+1}, \tilde x_{j+1}, r r^-_{j+1}/r_j^+) \\
&\leq c(1+\Lambda_{j+1}) E r^{2\alpha}.
\end{align*}
On the other hand, when $(1/4) r_j^+ / r_{j+1}^- \leq r \leq 1/4$, then we have $r \geq 1/c(\OmegaRef)$, and hence we can estimate using \eqref{eqn:reg-12}, \eqref{eqn:reg-13}:
\begin{align*}
&E^\infty(\Phi, M, p_j  + V_j, \tilde x_j, r) \\
&\leq c |\pi_{V_j} - \pi_V|^2 + c |p_j - p|^2 + E^\infty(\Phi, M, p + V, 0, 1) \\
&\leq c(1+\Lambda_{j+1}) E \\
&\leq c(1+\Lambda_{j+1}) E r^{2\alpha}.
\end{align*}
Bounds on $E^W$ and $E$ follow by similar computations.

Provided we take $\Lambda_j \geq c(1+\Lambda_{j+1})$ sufficiently big, depending only on $c(\OmegaRef)$ and $\Lambda_{j+1}$, this proves $(\dagger_j)$.
\end{proof}

\begin{proof}[Proof of Theorem \ref{thm:main}]
Take $\delta(\OmegaRef)$ sufficiently small so that we can apply Theorem \ref{thm:total-decay}.  Given $x \in \spt M \cap B_{1/16}$, let $V_x \in \cP_{T_x\Omega}$ be as in Theorem \ref{thm:total-decay}.

Given $x, y \in \spt M \cap B_{1/16}$, using \eqref{eqn:lower-reg} and \eqref{eqn:total-decay-concl2} we can estimate
\begin{align*}
\frac{1}{c(\OmegaRef)} |\pi_{V_x} - \pi_{V_y}|^2
&\leq |x - y|^{-n} \mu_M(B_{|x - y|}(x)) |\pi_{V_x} - \pi_{V_y}|^2 \\
&\leq 2 |x - y|^{-n} \int_{B_{|x - y|}(x)} |\pi_{V_x} - \pi_{T_zM}|^2 + |\pi_{V_y} - \pi_{T_z M}|^2 d\mu_M   \\
&\leq 2 E^{tot}(\Phi, M, x + V_x, x, |x - y|) + 2 E^{tot}(\Phi, M, y + V_y, y, 2|x - y|)  \\
&\leq c(\OmegaRef) |x - y|^{2\alpha} E, 
\end{align*}
and hence
\begin{equation}\label{eqn:reg-phi-tilt}
|\pi_{V_x} - \pi_{V_y}| \leq c(\OmegaRef) E^{1/2} |x - y|^\alpha.
\end{equation}

This effectively shows that $\spt M$ lies inside some $C^{1,\alpha}$ graph.  We must show that $\spt M$ fills out this entire graph, and the same holds true for $(\Phi^{-1})_\sharp M$.  We first prove two auxillary claims.

\textbf{Claim 1:} For all $x \in \sint\Omega \cap \spt M \cap B_{1/16}$, we have
\[
\frac{1}{r} d_H(\spt M \cap B_r(x), (x + V_x) \cap B_r(x)) \to 0 \quad \text{ as } r \to 0.
\]

We prove this by contradiction.  Otherwise, suppose there is a sequence $r_i \to 0$ and $\eps > 0$ so that
\begin{align}
&\frac{1}{2r_i} d_H(\spt M \cap B_{r_i}(x) , (x+ V_x) \cap B_{r_i}(x)) \nonumber \\
&\equiv d_H( \spt M_i \cap B_{1/2} , V_x \cap B_{1/2} ) \nonumber \\
&\geq \eps \quad \forall i. \label{eqn:d_h-graph},
\end{align}
where $M_i =  (\eta_{x, r_i})_\sharp M$.  Since $E^{\infty}(M_i, V_x, 0, 1) \to 0$ as $r \to 0$, we can argue as in Proposition \ref{prop:graph}, using Theorem \ref{thm:first-var} and Corollary \ref{cor:mass-non-conc}, to deduce that (after passing to a subsequence)
\[
M_i \to [V_x]
\]
as varifolds in $B_1$.  The lower bound \eqref{eqn:lower-reg} then implies that $\spt M_i \cap B_{1/2} \to V_x \cap B_{1/2}$ in the Hausdorff distance, which is a contradiction for large $i$.  This proves Claim 1.
\vspace{5mm}

\textbf{Claim 2:} We have $\spt M \cap B_{1/512} \cap \sint\Omega \neq \emptyset$.

On the one hand, by \eqref{eqn:lower-reg} we have $\lim_{r \to 0} r^{-n} \mu_M(B_r(x)) > 0$ for every $x \in \spt M$.  On the other hand, for every $x \in \spt M \cap B_{1/16} \cap \del\Omega$ the decay bound \eqref{eqn:total-decay-concl2} and non-concentration estimate \eqref{eqn:mass-non-conc} imply
\[
\limsup_{r \to 0} r^{-n} \mu_M(B_r(x) \cap \del\Omega) = 0.
\]
Since $\spt M \cap B_{1/512} \neq \emptyset$, this proves Claim 2.

\vspace{5mm}

Let $M_\Phi = (\Phi^{-1})_\sharp M$.  Since $1/c \leq \theta_M(x) \leq c$ for every $x \in \spt M \cap B_\theta$, $c = c(\theta, \OmegaRef)$, we can write
\[
\mu_M = \haus^n \llcorner \spt M \llcorner \theta_M,
\]
where $\spt M$ is countably-$n$-rectifiable.  By the area formula we can therefore write
\[
\mu_{M_\Phi} = \haus^n \llcorner \Phi^{-1}(\spt M) \llcorner (\theta_M\circ \Phi).
\]
This implies (for $\delta(n)$ sufficiently small) that $\Phi(\spt M) \cap B_{1/32} = \spt M_\Phi \cap B_{1/32}$, 
\[
\Phi( \spt M_\Phi \cap B_{1/32} ) \subset \spt M \cap B_{1/16},
\]
and $T_{\Phi^{-1}(x)} M_\Phi = D\Phi^{-1}|_x T_{x} M$ for $\mu_M$-a.e. $x \in B_{1/16}$.

\vspace{5mm}

We aim to show decay estimates like \eqref{eqn:total-decay-concl2}, \eqref{eqn:reg-phi-tilt} for $\Phi^{-1}(\spt M) \equiv \spt M_\Phi$.  This is essentially a direct consequence of the fact that $\Phi$ is a $C^2$ diffeomorphism.  Define $q = \Phi^{-1}(p)$, $W = D\Phi^{-1}|_p V$ for $p$, $V$ as in our hypotheses.  For $z \in \spt M_\Phi \cap B_{1/32}$, define $W_z = D\Phi^{-1}|_{\Phi(z)} V_{\Phi(z)}$.

Let us recall that $\Phi$ is a bi-Lipschitz equivalence: for every $x, y \in B_1$ we have
\[
|\Phi(x) - \Phi(y) - (x - y)| \leq c(n) \delta |x - y|,
\]
and we have the bound
\[
|D\Phi^{-1}|_x - D\Phi^{-1}|_y| \leq c(n) |D^2 \Phi|_{C^0(B_1)} |x - y| .
\]

\textbf{Claim 3:} Take $a \in B_{1/16}$, $V$ an $n$-plane, $y \in \spt M \cap B_r(a)$ for $r < 1/4$, and $x \in (a + V) \cap B_r(a)$.  Then we can find a $\tilde x \in (\Phi^{-1}(a) + D\Phi^{-1}|_a(V)) \cap B_{(1+c(n)\delta) r}(\Phi^{-1}(a))$ such that
\[
|\Phi^{-1}(y) - \tilde x| \leq 2 |y - x| + c(n) |D^2\Phi|_{C^0(B_2)} r^2.
\]

To prove this, let $\tilde x = \Phi^{-1}(a) + D\Phi^{-1}|_a(x - a)$.  Then
\[
|\Phi^{-1}(x) - \tilde x| = |\Phi^{-1}(x) - \Phi^{-1}(a) - D\Phi^{-1}_a(x - a)| \leq c(n) |D^2\Phi|_{C^0(B_2)} | x - a|^2,
\]
and
\[
|\Phi^{-1}(y) - \Phi^{-1}(x)| \leq (1+c\delta) |y - x|.
\]
This proves Claim 3.

Claim 3 and our definition of $W_{\tilde a}$ implies that for every $\tilde a = \Phi(a) \in \spt M_\Phi \cap B_{1/32}$ and $r < 1/8$ we have
\begin{align}
\sup_{z \in \spt M_\Phi \cap B_r(\tilde a)} d(z, \tilde a + W_{\tilde a}) 
&\leq 2 \sup_{y \in \spt M \cap B_{2r}(a)} d(y, a + V_a) + c(n) |D^2 \Phi| r^2 \nonumber \\
&\leq c(n) E^\infty(\Phi, M, a + V_a, a, 2r)^{1/2} r \nonumber \\
&\leq c(\OmegaRef) E^{1/2} r^{1+\alpha} \label{eqn:reg-phi-5}
\end{align}
and by the same reasoning
\begin{equation}\label{eqn:reg-phi-6}
\sup_{z \in \spt M_\Phi \cap B_{1/32}} d(z, q + W) \leq c(\OmegaRef) E^{1/2}.
\end{equation}

Similarly, combining Claim 1 with Claim 3 we have that
\begin{equation}\label{eqn:reg-phi-1}
\lim_{r \to 0} \frac{1}{r} d_H( \spt M_\Phi \cap B_r(\tilde a), (\tilde a + W_{\tilde a}) \cap B_r(\tilde a)) = 0
\end{equation}
for every $\tilde a \in \spt M_\Phi \cap B_{1/32} \cap \sint\Omega$.

Like in the proof of Lemma \ref{lem:B}, we have the bounds
\[
|\pi_{A(V)} - \pi_{B(V)}| \leq c(n) |A - B|, \quad |\pi_{A(V)} - \pi_{A(W)}| \leq c(n) |\pi_V - \pi_W| 
\]
for linear maps $A, B$ satisfying $|A - Id| + |B - Id| \leq \eps(n)$.  Therefore for $\delta(n)$ sufficiently small, we can estimate
\begin{align}
|\pi_{W_{\tilde a}} - \pi_{W_{\tilde b}}|
&\leq c(n) |D\Phi^{-1}|_a - D\Phi^{-1}|_b| + c(n) |\pi_{V_a} - \pi_{V_b}| \nonumber \\
&\leq c(n) |D^2\Phi|_{C^0(B_1)} |a - b| + c(\OmegaRef) E^{1/2} |a - b|^\alpha \nonumber \\
&\leq c(\OmegaRef) E^{1/2} |\tilde a - \tilde b|^\alpha, \label{eqn:reg-phi-2}
\end{align}
where $\tilde a = \Phi(a)$, $\tilde b = \Phi(b)$ both lie in $\spt M_\Phi \cap B_{1/32}$.  The same proof gives us the bound
\begin{equation}\label{eqn:reg-phi-3}
|\pi_{W_{\tilde a}} - \pi_W| \leq c(\OmegaRef) E^{1/2}.
\end{equation}

\vspace{5mm}

Combining \eqref{eqn:reg-phi-5}, \eqref{eqn:reg-phi-2}, \eqref{eqn:reg-phi-3} we get
\begin{align}\label{eqn:reg-phi-4}
&|\pi_W^\perp(y - z)| \leq c(\OmegaRef) E^{1/2}|x - y|, \quad |\pi_{W_z}^\perp(y - z)| \leq c(\OmegaRef) E^{1/2} |y - z|^{1+\alpha}, \\
&\quad |\pi_{W_y} - \pi_{W_z}| \leq c(\OmegaRef) E^{1/2} |y - z|^\alpha \nonumber
\end{align}
for all $y, z \in \spt M_\Phi \cap B_{1/32}$.

Let $U_0 = \Omega' \cap (q + W) \cap \overline{B_{1/128}(q)}$, $\sint U_0 = \sint\Omega' \cap (q + W) \cap B_{1/128}(q)$, and $U = (U_0 \times W^\perp) \cap B_{1/64}(q)$.  For $\delta(\OmegaRef)$ sufficiently small we have $|q| \leq 1/300$, and hence
\begin{equation}\label{eqn:reg-phi-8}
B_{1/256} \subset U \subset B_{1/32}.
\end{equation}
Further, by \eqref{eqn:reg-phi-6} we can assume that
\begin{equation}\label{eqn:reg-phi-7}
\spt M_\Phi \cap B_{1/32} \subset B_{1/300}(q + W),
\end{equation}
and therefore $\spt M_\Phi \cap U$ is a closed subset of $\R^{n+1}$.

By \eqref{eqn:reg-phi-4}, \eqref{eqn:reg-phi-7} the projection mapping $F(x) = \pi_W(x) + \pi_W^\perp(q)$ induces a $(1+c \delta)$-bi-Lipschitz equivalence between $\spt M_\Phi \cap U$ and some closed subset $\tilde U_0 \subset U_0$, and in particular by \eqref{eqn:reg-phi-1} $\tilde U_0$ satisfies the property that
\[
\lim_{r \to 0} d_H(B_r(y) \cap \tilde U_0, B_r(y) \cap U_0) = 0 \quad \forall y \in \tilde U_0 \cap \sint U_0.
\]
It follows by an elementary argument that we must have either $\tilde U_0 \cap \sint U_0 = \sint U_0$ or $\tilde U_0 \cap \sint U_0 = \emptyset$.  By Claim 2 $\spt M_\Phi \cap \sint\Omega' \cap B_{1/256} \neq \emptyset$, so $\tilde U_0 \cap \sint U_0 \neq \emptyset$, and hence $\tilde U_0 = U_0$.

This proves $F$ is a $(1+c\delta)$-bi-Lipschitz equivalence between $\spt M \cap U$ and $U_0$.  We can therefore find a Lipschitz function $u : U_0 \to \R$ so that $\spt M_\Phi \cap U = \graph(u)$.  It follows easily from \eqref{eqn:reg-phi-4}, \eqref{eqn:reg-phi-6} that $u$ is $C^{1,\alpha}$ and satisfies the bound $|u|_{C^{1,\alpha}(U_0)} \leq c(\OmegaRef) E^{1/2}$.
\end{proof}

\section{Higher codimension and ambient manifolds}\label{sec:codim}

The only part of our regularity theorem that requires codimension-one is in obtaining the estimates \eqref{eqn:l2-case-mass}, \eqref{eqn:l3-case-mass}.  If $l = 1$, or one can otherwise verify these estimates, then we get corresponding regularity in higher codimension.  The most obvious situation in which \eqref{eqn:l2-case-mass}, \eqref{eqn:l3-case-mass} continue to hold is when we know a priori that $\spt M$ is contained in some closed $(n+1)$-submanifold with controlled geometry.  Beyond these estimates the proofs are verbatim to the codimension-one case.  To avoid excess notational clutter we outline this here, rather than integrate it into the original proof.

Suppose we are in $\R^{n+1+k}$.  In analogy to the original definitions, we now let $\OmegaRef = \OmegaRefN^l \times \R^{m+k}$ be a finite union of half-spaces, such that $\OmegaRefN^l$ is $0$-symmetric.  Define $\cP_\OmegaRef$ to be the collection of $n$-planes of the form $\R^l \times W^{n-l}$ for some $(n-l)$-plane in $\{0\}\times \R^{m+1+k}$.  Given $\Phi, \Psi : B_2 \to B_2$ as in Defintion \ref{def:omega}, we let $\Omega = \Phi(\Psi(\OmegaRef)) \cap B_1 \in \cD_\delta(\OmegaRef)$ analogously to the codimension-one case.

Then we have the following extension of our results to higher codimension.  We point out that since any complete $C^3$ $(n+1)$-manifold $N$ can be locally isometrically embedded into some $\R^{n+1+k}$ space, Theorem \ref{thm:codim} implies our regularity can be extended to codimension-one varifolds in general ambient manifolds.

\begin{theorem}\label{thm:codim}
Let $\OmegaRef = \OmegaRefN^l \times \R^{m+k}$ be a polyhedral cone domain in $\R^{n+1+k}$, and let $\Omega = \Phi(\Psi(\OmegaRef)) \cap B_1 \in \cD_\eps(\OmegaRef)$.  Write $N = \Phi(\Psi(\OmegaRef \times \R^m \times \{0^k\})) \cap B_1$, so that $N$ is a closed $C^2$ $(n+1)$-submanifold of $B_1$.  Let $M \in \cI_n(B_1)$.

If $l \geq 2$, assume that $M$ satisfies $\spt M \subset N$, and
\[
\delta M(X) = -\int H^{tan}_{M, N} \cdot X d\mu_M, \quad ||H^{tan}_{M,N}||_{L^\infty} \leq \eps,
\]
for all $X \in C^1_c(B_1)$ which are tangential to $\Omega$, and $X(x) \in T_x N$ for all $x \in N$.  If $l = 1$, it suffices to assume that $M \in \cIVT_n(\Omega, B_1)$, with $||H^{tan}_M||_{L^\infty} \leq \eps$.

Then provided $\eps(\OmegaRef)$ is sufficiently small, the Theorems of Sections \ref{sec:main} -- \ref{sec:reg} continue to hold for $M$.
\end{theorem}

\section{Minimizers}\label{sec:minz}

In this section we use our Theorems \ref{thm:main}, \ref{thm:codim} to prove the partial regularity Theorem \ref{thm:partial-reg} for codimension-one area-minimizing free-boundary currents in a manifold $N$.  We will use the Nash embedding theorem and reduce our problem to one in Euclidean space, as in Section \ref{sec:codim}.  We prove some technical Lemmas, establishing appropriate non-concentration (\ref{lem:minz-non-conc}) and compactness (\ref{lem:minz-compact}) for almost-area-minimizing currents.  We then classify low-dimensional minimizing cones with free-boundary (Lemmas \ref{lem:minz-planes}, \ref{lem:low-dim-cones}), and use a standard-dimension reducing argument to get our partial regularity bound.

First, we establish some notation.  Given an open set $U \subset \R^{n+1+k}$, we let $\cI_n(U)$ be the set of integer-multiplicity rectifiable $n$-currents in $U$ with locally finite mass.  Given $T \in \cI_n(U)$, we will write $||T||$ for the mass measure, and $\spt T \subset U$ for the support of $T$.  Note that associated to every such $T$, there is an integral varifold which we will often denote by $T$ also, such that $\mu_T = ||T||$.  Given an open set $E \subset U$, we write $[E]$ for the $(n+1)$-current in $U$ obtained by integrating the standard orientation over $E$.  We say $E$ is a set of locally finite perimeter in $U$ if $\del[E] \in \cI_n(U)$.

Given a domain $\Omega \subset \R^{n+1+k}$, and $T \in \cI_n(U)$, we say $T$ is minimizing with free-boundary in $\Omega$ if $T = T\llcorner \sint\Omega$, and satisfies
\begin{equation}\label{eqn:T-minz}
||T||(W) \leq ||T+S||(W)
\end{equation}
for all $W \subset\subset U$, and all $S \in \cI_n(U)$ satisfying $\spt S \subset \overline{\Omega} \cap W$, $\del S \llcorner \sint\Omega = 0$.  

Given $A \geq 0$, $\alpha \geq 0$, $\delta > 0$, we say $T$ is $(A, \alpha, \delta)$-almost-minimizing with free-boundary in $\Omega$ if $T = T \llcorner \sint\Omega$, and given $W$, $S$ as above then $\diam(W) \leq \delta$ and we have the inequality 
\[
||T||(W) \leq ||T+S||(W) + A r^{n+\alpha}, \quad \diam(W) \leq \min\{ \delta,  r \} .
\]
instead of \eqref{eqn:T-minz}.

Let $\OmegaRef = \Omega_0 \times \R^{m+k}$ be a polyhedral cone domain in $\R^{n+1+k}$, and $\Omega \in \cD_\eps(\OmegaRef)$.  Recall that $\Omega$ is closed in $B_1$, $\overline{\sint\Omega} = \Omega$, and we have the natural stratification
\[
\del_0\Omega \subset \del_1\Omega \subset \ldots \subset \del_{n+1+k}\Omega = \Omega, 
\]
where $\del_i\Omega$ consists of the points near which $\Omega$ is diffeomorphic to a polyhedral cone domain having at most $i$ dimensions of translational symmetry.

Given $T \in \cI_n(B_1)$ with $T = T \llcorner \sint\Omega$, we define $\reg T$ to be the set of points $x \in \spt T \subset \Omega$ with the following property: there is an $r > 0$, a $C^2$ diffeomorphism $\phi : B_r(x) \to B_r(x)$, a $V \in \cP_{T_x\Omega}$, and a $C^{1,\alpha}$ function $u : B_{2r}(x) \cap (x + V) \to V^\perp$, so that
\[
\phi(\spt T \cap B_r(x)) = \graph_{x + V}(u) \cap B_r(x).
\]
Note that this implies $\del_0\Omega \cap \reg T = \emptyset$ (but also recall that $\del_i\Omega = \emptyset$ for $i < m+k$).  Note further that if $x \in \reg T$, then the tangent plane of $T$ at $x$ lies in $\cP_{T_x\Omega}$.  We define $\sing T = \spt T \setminus \reg T$.

The global domains we are interested in are locally polyhedral domains, defined precisely here.
\begin{definition}\label{def:domain}
Let $N^{n+1}$ be a complete manifold.  A \emph{locally polyhedral domain} is a closed domain $\Omega \subset N$ with non-empty interior which is locally diffeomorphic to some polyhedral cone domain at every point.  Precisely, for every $x \in \Omega$ there is a radius $r > 0$, a polyhedral cone domain $\OmegaRef$, and a diffeomorphism $\phi : B_1(0^{n+1}) \to B_r(x \in N)$ so that $\phi(B_1 \cap \OmegaRef) = \Omega \cap B_r(x)$ and $D\phi|_0$ is an isometry.  We say $\Omega$ is $C^k$ is the associated diffeomorphism is $C^k$.

In the special case when every model polyhedral cone $\OmegaRef$ takes the form $L(\R^i \times [0, \infty)^{n-i})$, for some $i$ and linear isomorphism $L$, then $\Omega$ is said to be a \emph{domain with corners}.  We say $\Omega$ has dihedral angles $\leq \pi/2$ (resp. $=\pi/2$) if every model $\OmegaRef$ as above has dihedral angles $\leq \pi/2$ (resp. $=\pi/2$).  As before, we may call $\Omega$ with dihedral angles $\leq \pi/2$ \emph{non-obtuse}.
\end{definition}

\vspace{5mm}

The following two Lemmas prove the required compactness and closure theorems for our (almost-)area minimizing currents.
\begin{lemma}[Boundary non-concentration]\label{lem:minz-non-conc}
Let $\OmegaRef$ be a polyhedral cone domain in $\R^{n+1+k}$.  Let $\Omega \in \cD_\eps(\OmegaRef)$, and $T \in \cI_n(B_1)$ be $(A, 0, 1)$-almost-minimizing with free-boundary in $\Omega$, such that $\del T \llcorner \sint\Omega = 0$.

Then for $\eps \leq \eps(\OmegaRef)$, there is a continuous function $\eta(\tau) : [0, 1] \to [0,1]$ depending only on $\theta, \OmegaRef$, satisfying $\eta(0) = 0$, such that
\[
||T||(B_\tau(\del\Omega) \cap B_\theta) \leq ||T||(B_1) \eta(\tau) + 2 A.
\]
\end{lemma}

\begin{remark}
It should in fact be true for minimizers that $||\del T||(B_\theta) \leq C(\theta, \OmegaRef, ||T||(B_1))$ (c.f. \cite{Gr}).  However the above weaker statement is easier to prove in our more singular setting setting, and suffices for our purposes.
\end{remark}

\begin{proof}
Fix any unit vector $\tau \in \sint\Omega$, and let $H$ be the half-space with outer normal given by $-\tau$.  Then for $\eps(\OmegaRef)$ sufficiently small $\Omega \subset H$, and we can write $\del\Omega$ as a Lipschitz graph over $\del H$.  Given $R \geq 2$, $\eps(\OmegaRef)$ small, and $x \in B_1 \cap \Omega$, there is a unique $\zeta(x) \in \del\Omega$ such that $\zeta(x) - R\tau = \lambda_x (x - R \tau)$ for some $\lambda_x \geq 1$.  Provided we fix $R(\theta, \OmegaRef)$ to be sufficiently large, then $\zeta$ is a Lipschitz function satisfying
\begin{equation}\label{eqn:T-non-conc-1}
\zeta(B_{\theta}) \subset B_{(1+\theta)/2}, \quad |x - \zeta(x)| \leq c(\OmegaRef) d(x, \del\Omega), \quad ||D\zeta||_{L^\infty(B_1)} \leq c(\OmegaRef) . 
\end{equation}

Since $T = T \llcorner \sint\Omega$, let us view $T \in \cI_n(\sint\Omega \cap B_1)$, in which case $\del T = 0$.  Define $f(x) = d(x, \del\Omega) + g(|x|)$, where $g = 0$ on $B_\theta$, $g$ is increasing, $g = 1$ on $B_{(1+\theta)/2}$, and $|Dg| \leq 4/(1-\theta)$.  Then for all $h \leq 1$, $T \llcorner \{ f < h \} \in \cI_n(\sint\Omega \cap B_{(1+\theta)/2})$.

Let $T_h = \langle T, f, h\rangle$ be the slice (see \cite[section 28]{simon:gmt}) of $T$ at $f = h$, defined for a.e. $h$, and let $m(h) = ||T \llcorner \{ f < h\}||(B_1)$.  By the coarea formula, we have
\begin{equation}\label{eqn:T-non-conc-2}
||T_h||(B_1) \leq c(\theta) m' 
\end{equation}
for a.e. $h < 1$.  On the other hand, again for a.e. $h$, by slicing we have
\[
T_h = \del (T \llcorner \{ f < t\}).
\]
Therefore if we define 
\[
F(t, x) = t x + (1-t) \zeta(x), \quad t \in [0, 1], x \in B_\theta,
\]
then $\del F_\sharp ([0, 1] \times T_h) \llcorner \sint\Omega = T_h$, $\spt F_\sharp ([0, 1] \times T_h) \subset \Omega \cap B_{(1+\theta)/2}$, and so by area comparison and the homotopy formula (\cite[26.22]{simon:gmt}) we have
\begin{equation}\label{eqn:T-non-conc-3}
m(h) \leq ||F_\sharp T_h||(B_1) + A \leq c(\OmegaRef) h ||T_h||(B_1) + A,
\end{equation}
having used \eqref{eqn:T-non-conc-1}, and the fact that $|x - \zeta(x)| \leq c d(x, \del\Omega) \leq c f(x)$.  Together, \eqref{eqn:T-non-conc-2} and \eqref{eqn:T-non-conc-3} imply that for a.e. $h< 1$ we have
\[
m(h) \leq c(\OmegaRef, \theta) h m' + A,
\]
and therefore since $m$ is increasing there is an $\alpha > 0$ so that $h^{-\alpha} (m(h) - A)$ is increasing for $h \in (0, 1)$.  Letting $\eta(h) = h^\alpha$ proves the Lemma.
\end{proof}

\begin{lemma}\label{lem:minz-compact}
Let $\Omega = \Omega_0^l \times \R^{m+k}$ be a polyhedral cone domain in $\R^{n+1+k}$.  Let $\eps_i \to 0$, $A_i \to 0$, $\Omega_i \in \cD_{\eps_i}(\Omega)$, and $T_i \in \cI_n(B_1)$.  Suppose $T_i$ is $(A_i, 0, 1)$-area-minimizing with free-boundary in $\Omega_i$, and satisfies
\begin{equation}\label{eqn:compactness-hyp1}
\del T_i \llcorner \sint\Omega = 0, \quad \sup_i ||T_i||(B_1) < \infty.
\end{equation}
Then after passing to a subsequence, there is a $T \in \cI_n(B_1)$ which is area-minimizing with free-boundary in $\Omega$, such that $T_i \to T$ as currents in $B_1$ and $||T_i|| \to ||T||$ as Radon measures on $B_1$.

If the $T_i$ as varifolds lie in $\cIVT(\Omega_i, B_1)$ with $||H^{tan}_{T_i}||_{L^\infty(B_1)} \to 0$, then $T \in \cIVT(\Omega, B_1)$, $T_i \to T$ as varifolds, $||H^{tan}_T||_{L^\infty(B_1)} = 0$.

Write $\Omega_i = \Phi_i(\Psi_i(\Omega)) \cap B_1$, and let $N_i = \Phi_i(\Psi_i(\Omega_0 \times \R^m \times \{0^k\})) \cap B_1$.  Suppose there are relatively open sets $E_i \subset N_i \cap B_1$ so that $T_i = \del[E_i] \llcorner \sint\Omega_i$.  Then $T$ is multiplicity-one $||T||$-a.e. and for every $x \in \reg T \cap B_1$, there is a neighborhood $B_r(x)$ such that $\sing T_i \cap B_r(x) = \emptyset$ for all $i >> 1$.
\end{lemma}

\begin{proof}
By assumption we have $|\Phi_i - \mathrm{id}|_{C^2(B_2)} \to 0$, $|\Psi_i - \mathrm{Id}| \to 0$.  We can find $C^2$ functions $F_i : B_2 \to B_2$ functions satisfying
\[
\Omega_i = F_i(\Omega) \cap B_1, \quad |F_i - \mathrm{id}|_{C^2(B_2)} \to 0,
\]
such that for every $U \subset\subset \sint\Omega \cap B_1$, we have $F_i|_U = \mathrm{id}$ for all $i>>1$.

By a diagonalization argument, after passing to a subsequence, we can find a $T \in \cI_n(\sint\Omega)$ so that $T_i(\omega) \to T(\omega)$ for any smooth $n$-form compactly supported in $B_1 \cap \sint\Omega$.  We have $||T||(\sint\Omega \cap B_1) < \infty$, so extend $T$ to be an element of $\cI_n(B_1)$ by restriction $T := T \llcorner \sint\Omega$.  Because of our definition of $T$, and convergence $T_i \to T$ as currents on compact subsets of $\sint\Omega \cap B_1$, we can also assume that the $||T_i||$ limit to some Radon measure on $B_1$, which is $\geq ||T||$.

By slicing theory and \eqref{eqn:compactness-hyp1}, and after passing to a further subsequence, we can find $d_j \to 0$, $r_k \to 1$ so that if $D_{jk} = B_{r_k} \cap \{ d_{\del\Omega} > d_j\}$, then 
\begin{equation}\label{eqn:compactness-1}
||\del (T_i \llcorner D_{jk} )||(B_1) \leq C_{jk}
\end{equation}
for $C_{jk}$ independent of $i$.  Since, for any fixed $j, k$, $T_i \llcorner D_{jk} \to T \llcorner D_{jk}$, we have $\del (T_i \llcorner D_{jk}) \to \del (T \llcorner D_{jk})$ and hence
\begin{equation}\label{eqn:compactness-2}
||\del (T \llcorner D_{jk} )||(B_1) \leq C_{jk}
\end{equation}
also.

By Lemma \ref{lem:minz-non-conc}, we have
\begin{equation}\label{eqn:compactness-4}
||T_i ||(B_{r_k} \setminus D_{jk}) \leq \eta_k(d_j) + 2A_i
\end{equation}
for some continuous function $\eta_k$ independent of $i, j$, satisfying $\eta_k(0) = 0$.  Lower semi-continuity and our hypothesis $A_i \to 0$ implies
\begin{equation}\label{eqn:compactness-5}
||T||(B_{r_k} \setminus D_{jk}) \leq \eta_k(d_j).
\end{equation}
Inequalities \eqref{eqn:compactness-4} and \eqref{eqn:compactness-5} imply that $T_i \to T$ as currents in $B_1$.

Fix $S \in \cI_n(B_1)$ with $\spt S \subset \Omega \cap B_1$., and $\del S \llcorner \sint\Omega = 0$.  We wish to prove that
\[
||T||(B_1) \leq ||T+S||(B_1).
\]
Since $T = T\llcorner \sint\Omega$, there is no loss in assuming $S = S \llcorner \sint\Omega$ also.  Define $S_i = (F_i)_\sharp S$, so that for $i >> 1$: $\spt S_i \subset \Omega_i \cap B_1$, $\del S_i \llcorner \sint \Omega_i = 0$, $S_i = S_i\llcorner \sint\Omega_i$ and $||S - S_i||(B_1) \to 0$.

Fix a $k$ such that $\spt S \subset B_{r_k}$ for all $i$ large, and fix any $j$ arbitrary.  By \eqref{eqn:compactness-1}, \eqref{eqn:compactness-2}, slicing theory and the deformation theorem (see e.g. \cite[Theorem 7.2.4]{simon:gmt}), we can find $P_i \in \cI_{n+1}(B_1)$, $R_i \in \cI_n(B_1)$ such that
\[
(T_i - T) \llcorner D_{jk} = \del P_i + R_i
\]
and
\[
\spt P_i, \spt R_i \subset \overline{D_{j+1,k}}, \quad ||P_i||(B_1) + ||R_i||(B_1) + ||\del P_i||(\del B_{r_k}) \to 0.
\]
Since $T_i$ is $(A_i, 0, 1)$-almost-minimizing with free-boundary in $\Omega_i$, and $\del P_i$ is an admissible competitor, we have
\[
||T_i||(B_r) \leq ||T_i + \del P_i + S_i||(B_r) + A_i \quad \forall r_k < r < 1,
\]
and hence taking $r \to r_k$, we get
\begin{align}
||T_i||(B_{r_k}) 
&\leq ||T_i + \del P_i + S_i||(B_{r_k}) + ||\del P_i||(\del B_{r_k}) + A_i \\
&\leq \left( ||T \llcorner D_{jk} + S|| + ||S - S_i|| + ||R_i|| \right)(B_{r_k}) + ||\del P_i||(\del B_{r_k}) + A_i
\end{align}
Let $i \to \infty$, and then by lower-semi-continuity we get
\begin{align}
||T||(B_{r_k}) 
&\leq ||T \llcorner D_{jk} + S||(B_{r_k}) \\
&\leq ||T +S||(B_{r_k}) + ||T ||(B_{r_k} \setminus D_{jk}) \\
&\leq ||T+S||(B_{r_k}) + \eta_k(d_j) \label{eqn:compactness-3}
\end{align}
Since \eqref{eqn:compactness-3} holds for every $j$ ($k$ fixed), we can take $j \to \infty$ to get
\[
||T||(B_{r_k}) \leq ||T+S||(B_{r_k}),
\]
which proves that $T$ is minimizing.  If we apply the same argument to $S = 0$, then we get $||T||(B_{r_k}) = \lim_i||T_i||(B_{r_k})$, which by lower-semi-continuity implies $||T_i|| \to ||T||$.  This proves the first part of the Lemma.

\vspace{5mm}

Abusing notation, let us write $T_i$, $T$ for the underlying varifolds associated to the $n$-currents, and assume as in our hypothesis that $T_i \in \cIVT(\Omega_i, B_1)$.  By Theorem \ref{thm:first-var}, the $T_i$ have uniformly bounded first variation, and hence after passing to a further subsequence we can assume $T_i$ converge as varifolds in $B_1$.  Since $T_i$, $T$ are both integral varifolds, and $||T_i|| \to ||T||$, then we must have $T_i \to T$ as varifolds.

Fix any $X \in C^1_c(B_1)$ tangential to $\Omega$.  Let $X_i(x) = DF_i|_{F_i^{-1}(x)} X(F_i^{-1}(x))$.  Then $X_i \in C^1_c(B_1)$, and $X_i$ is tangential to $\Omega_i$, and $|X_i - X|_{C^0(B_1)} \to 0$.  Pick $\theta < 1$ such that $\spt X_i, \spt X \subset B_\theta$.  We have
\[
|\delta T_i(X)| \leq C(\Omega)|X|_{C^0(B_1)} ||H^{tan}_{T_i}||_{L^\infty(B_1)} ||T_i||(B_\theta) + C(\Omega) |X_i - X|_{C^0(B_1)} ||\delta T_i||(B_\theta),
\]
and hence $\delta T(X) = 0$.  This implies $T \in \cIVT(\Omega, B_1)$, and $H^{tan}_T = 0$.

\vspace{5mm}

Suppose $T_i = \del[E_i] \llcorner \sint\Omega_i$ for some relatively open sets $E_i \subset N_i \cap B_1$.  Then, after passing to a further subsequence as necessary, we can assume $[E_i] \to E \subset \Omega_0 \times \R^m \times \{0^k\}$, and hence $T = \del[E] \llcorner \Omega_0 \times \R^m$ has multiplicity-one $||T||$-a.e.  If $x \in \reg T$, then at sufficiently small scales $T \llcorner B_r(x)$ is varifold-close to a multiplicity-one plane in $\cP_{T_x\Omega}$, and hence the $T_i \llcorner B_r(x)$ lie close to this plane also.  For $r$ sufficiently small, and $i$ sufficiently large, we deduce by Theorems \ref{thm:main}, \ref{thm:codim} that $\sing T_i \cap B_r(x) = \emptyset$.
\end{proof}


\vspace{5mm}

We now work towards classifying low-dimensional tangent cones.
\begin{lemma}\label{lem:minz-planes}
Let $\Omega^{n+1}$ be a $0$-symmetric polyhedral cone domain.  If $n \geq 2$ assume $\Omega$ is non-obtuse.  Then $\Omega$ is a domain-with-corners, and given any $T \in \cI_n(\R^{n+1})$ a free-boundary minimizing cone in $\Omega$, such that $\spt T$ is contained in a plane, then $T = 0$.
\end{lemma}

It is plausible Lemma \ref{lem:minz-planes} fails in higher dimensions when $\Omega$ does not satisfy the dihedral angle condition.
\begin{example}\label{ex:possible-minz-plane}
Let $\Omega^3$ be the intersection of the half-spaces with outer normals given by
\[
(1, 0, 1), \quad (-1, 0, 1), \quad (0, \eps, 1).
\]
Then when $\eps > 0$ is small-ish, it may be the case that the plane $y = 0$ is not minimizing with free boundary in $\Omega$.
\end{example}

\begin{proof}[Proof of Lemma \ref{lem:minz-planes}]
By the constancy theorem $T = k[P \cap \sint\Omega]$ for some integer $k \geq 0$, and some oriented plane $P$.

We perform induction on $n$. First assume $n = 1$: $\Omega = W^2$ is a $2$-dimensional wedge with interior angle $\beta < \pi$, and $P \cap W$ is a ray.  If we write $\del W = L_1 \cup L_2$ where $L_i$ are the two rays meeting at angle $\beta$, then after relabeling as necessary the rays $P$, $L_1$ meet at some angle $< \pi/2$.  Take any point $q\in P$, and let $q_1=\proj_{L_1}q$. Then $|qq_1|<|oq|$, so $P$ is not length-minimizing. We therefore must have $k=0$ and hence $T=0$.

Assume the statement holds for all positive integers less than $n$. Since $\Omega^{n+1}$ is $0$-symmetric and non-obtuse, by \cite[Theorem 1.1]{Coxter1934discrete}, $\Omega$ is simplicial. In other words, there exists $n+1$ half spaces $H_1,\cdots,H_{n+1}$ such that $\Omega=\cap_{i=1}^{n+1}H_i$, and if $F_i=\partial\Omega\cap \partial H_i$, then each $F_i$ and $F_j$ meet along some $(n-1)$-dimensional set.  Let $\nu_i$ denote the outer unit normal vector of $\partial H_i$. Since $\Omega$ is non-obtuse, $\nu_i\cdot \nu_j\le 0$, and
\[\Omega=\{x\in \R^n: x\cdot \nu_i\le 0, i=1,\cdots,n+1\}.\]

Suppose, for the sake of contradiction, that $k\ne 0$. We first observe that $\spt T\cap \partial_{n-1}\Omega=\{0\}$. Otherwise, let $q\in \spt T\cap \partial_{n-1}\Omega$. Then $Tan_q T=T'\times \R^{n+1-j}$, $T_q\Omega=\Omega'^j\times \R^{n+1-j}$ for some $2\le j\le n$, $\Omega'$ is $0$-symmetric, and $T'$ is free-boundary minimizing in $\Omega'$, contradicting the induction hypothesis.

Therefore $\spt T\cap \partial \Omega$ only in the smooth part of $\partial \Omega$, and thus $\spt T$ meets $\partial \Omega$ orthogonally. Denote $\Omega_0=P\cap \Omega$. Note that $\Omega_0$ is a $0$-symmetric $n$-dimensional polyhedral cone. Thus by dimension counting, $P$ intersects exactly $n$ faces, say $F_1,\cdots,F_n$ orthogonally. After a rotation if necessary, we may assume that $P=\{x_{n+1}=0\}$, and $e_n\in P\cap \sint \Omega$. Since $e_n\in \sint\Omega\subset H_{n+1}$, we have $e_{n}\cdot \nu_{n+1}\le 0$. On the other hand, regarded as a convex polyhedron in $P\cong \R^n$, $\Omega_0$ has outer unit normals $\nu_1,\cdots, \nu_n$, and $e_n\in \sint(\Omega_0)$. Therefore there exists a linear combination
\[e_n=\sum_{j=1}^n \lambda_j \nu_j\]
with $\lambda_j\le 0$. Hence
\[e_n\cdot \nu_{n+1}=\sum_{j=1}^n \lambda_j \nu_j\cdot \nu_{n+1}\ge 0.\]

As a result, $e_n\cdot \nu_{n+1}=0$, and hence $e_n\in \partial H_{n+1}$. This implies that $P\cap F_{n+1}\ne \{0\}$, contradiction. 
\end{proof}

\begin{lemma}\label{lem:low-dim-cones}
Let $T \in \cI_n(\R^{n+1})$ be a dilation-invariant minimizing current with free-boundary in a polyhedral cone domain $\Omega^{n+1}$.  Then $T$ is entirely regular if:
\begin{enumerate}
\item $n = 1$;
\item $n = 2$ and $\Omega$ is non-obtuse;
\item $n \leq 6$ and $\Omega = [0, \infty)^l \times \R^{n+1-l}$.
\end{enumerate}
\end{lemma}

\begin{proof}
Suppose $n \leq 6$ and $\Omega = [0, \infty)^l \times \R^{n+1-l}$.  When $l = 0$ then $T$ is minimizing without boundary in $\R^{n+1}$, and hence $T$ is planar by Simons' theorem \cite{Simons}.  For general $l$, we can assume by induction that $\sing T \subset \{0\}$.  Otherwise, if $x \in \sing T \setminus \{0\}$, then by \eqref{eqn:lower-reg} and Lemma \ref{lem:minz-compact} we would be able to find a non-zero singular tangent cone $T' \times [\R]$ at $x$ which is minimizing with free-boundary in some $[0, \infty)^{l-1} \times \R^{n-l}$.

Now by reflection we can obtain a cone $\tilde T$ in $\R^{n+1}$ which is $C^{1,\alpha}$ away from $0$, and smooth and stable away from finitely-many $(n-2)$-planes.  By standard interior elliptic regularity and a cutoff argument it follows that $\tilde T$ is smooth and stable on all of $\R^{n+1} \setminus \{0\}$, and hence by Simons' theorem $\tilde T$ is planar.

Suppose $n = 1$.  By the previous characterization, we can WLOG assume $\Omega = W^2$ is a non-obtuse wedge.  Then $T$ is a cone over finitely-many points, and hence is a union of minimizing rays (with possible multiplicity) in $W^2$.  But then by Lemma \ref{lem:minz-planes} $T = 0$.  This completes the $n = 1$ case.

Suppose $n = 2$ and $\Omega$ is non-obtuse.  We claim that $\spt T \cap (\del_1\Omega \setminus \{0\}) = \emptyset$.  Otherwise, if there were an $x \in \spt T \cap (\del_1\Omega \setminus \{0\})$ then by \eqref{eqn:lower-reg}, Lemma \ref{lem:minz-compact} (and a rotation as necessary) we could obtain a non-zero tangent cone $T' \times [\R]$ in some cone $W^2 \times \R$, where $W^2$ is a wedge.  But $T'$ is minimizing in $W^2$, and hence by our $n = 1$ case we would have $T' = 0$, which is a contradiction.

Therefore $\spt T \setminus \{0\}$ meets $\del\Omega$ only where $\del\Omega$ is planar, and hence by reflection and interior regularity $\spt T$ is smooth away from $0$.  Moreover, by the free-boundary condition $\spt T$ satisfies the usual stability inequality
\[
\int_{\spt T} |A|^2 \zeta^2 \leq \int_{\spt T} |\nabla \zeta|^2 \quad \forall \zeta \in C^1_c(\R^{n+1} \setminus \{0\}),
\]
and $\del_n |A| = 0$ along $\spt T \cap \del\Omega \setminus \{0\}$.  (Here $A$ is the second fundamental form of $\spt T$.)  It then follows by the usual proof of Simons' theorem that $\spt T$ is a finite union of disjoint free-boundary planes in $\sint\Omega$.  

As before we can assume WLOG that $\Omega$ is either $W^2 \times \R$ or $0$-symmetric.  If $\Omega = W^2 \times \R$ then since $\spt T$ is planar and $\spt T \cap (\{0\}\times \R) \subset \{0\}$ we must have $\spt T = W^2 \times \{0\}$, and hence $T$ is regular.  If $\Omega$ is $0$-symmetric then by Lemma \ref{lem:minz-planes} we must have $T = 0$.
\end{proof}

\vspace{5mm}

Using the previous results, and the Nash embedding theorem, partial regularity is now a standard argument.

\begin{proof}[Proof of Theorem \ref{thm:partial-reg}]
Since $N$ is a complete, Riemannian $C^3$ $(n+1)$-manifold we can by isometrically embed it in some $\R^{n+1+k}$ space (for $k = k(n)$).  WLOG assume $0 \in \del\Omega \cap \spt T$, $T_0N = \R^{n+1} \times \{0^k\}$, and by dilating $N$ as necessary we can assume that the map $\exp_{T^\perp N}(x, v)$ that takes $x \in N$, $v \in T^\perp_x N$ to $x + v$ is a diffeomorphism onto its image for $x \in B_2$, $|v| \leq 2$, and $\exp_{T^\perp N}( B_2, B_2) \supset B_1$.  Let $\Omega' = \exp_{T^\perp N}(\Omega, B_2) \cap B_1$, and then $\Omega' \in \cD_\eps(T_0\Omega \times \R^k)$, where $\eps$ can be made arbitrarily small by dilating $N$.

Suppose $T$ is area-minimizing with free-boundary in $\Omega$.  Then $T$ is $(A, 1, 1)$-almost-area-minimizing in $B_1$ with free-boundary in $\Omega'$, where $A$ depends only on the curvature of $N$.  Since $T$ has zero tangential mean curvature in $\Omega \subset N$, then $T$ has bounded tangential mean curvature in $\Omega' \subset B_1$ (again controlled by the curvature of $N$).  By standard codimension-one theory we can reduce to the case when $T = \del[E] \llcorner \sint\Omega'$, for $E \subset N$ relatively open.

Suppose $T = \del[E] \llcorner \sint\Omega$ is an isoperimetric region.  Then $T$ is $(A, 1, \delta)$-almost-area-minimizing in $N$ with free-boundary in $\Omega$ for some constants $A, \delta$ (depending on the volume $|E|$; see \cite[Example 21.3]{maggi_2012}), and hence as above $T$ is $(A', 1, \delta/2)$-almost-area-minimizing in $B_1$ with free-boundary in $\Omega'$.  Similarly, since $T$ has bounded tangential mean curvature in $\Omega \subset N$, then $T$ has bounded tangential mean curvature in $\Omega' \subset B_1$ also.

In either case, if we consider a sequence of dilations $T_i = (\eta_{0, r_i})_\sharp T$ then by \eqref{eqn:lower-reg}, \eqref{eqn:monotonicity} and Lemma \ref{lem:minz-compact}, the $T_i$ will subsequentially converge as currents and varifolds to some non-zero area-minimizing cone $T''$ with free-boundary in $\Omega''$, where $\Omega'' = T_0\Omega \subset T_0N$.  Moreover, if $x \in \reg T''$, then $\sing T_i \cap B_r(x) = \emptyset$ for all $i$ large.
	
Lemma \ref{lem:low-dim-cones} implies that:
\begin{enumerate}
\item if $T''$ is $(n-1)$-symmetric then $T''$ is regular;
\item if $T''$ is $(n-2)$-symmetric and $\Omega$ is non-obtuse, then $T''$ is regular;
\item if $T''$ is $(n-6)$-symmetric and $\Omega$ has dihedral angles $=\pi/2$, then $T''$ is regular.
\end{enumerate}
The partial regularity of Theorem \ref{thm:partial-reg} then follows from a standard dimension reducing argument.
\end{proof}

\section{Appendix: Tangential first variation}

Let $Q$ be a $C^2$, closed $p$-submanifold in $\R^{n+k}$.  Suppose that the nearest point projection $\zeta(x) : B_1 \to Q$ is smooth in $B_1$.  Define the subspaces
\[
\tau(x) = T_{\zeta(x)} M, \quad \iota(x) = \mathrm{span} \{ x - \zeta(x) \} , \quad \sigma(x) = \tau(x) \oplus \iota(x).
\]
Write $r = |x - \zeta(x)|$.

We consider here a rectifiable $n$-varifold $M \in \cIV(B_1 \subset \R^{n+k})$, such that $\mu_M(Q) = 0$ and $\mu_M(B_1) < \infty$.

\subsection{Free-boundary}

Let $Q$ be a hypersurface, so that $p = n+k-1$.  Let $\delta M^{tan}$ be $\delta M$ restricted to vector fields $X$ which are tangential to $Q$, i.e. those $X \in C^1_c$ for which $X(x) \in T_xQ$ for all $x \in Q$.

Assume $||\delta M^{tan}||(B_1) < \infty$, so in particular we can write
\[
\delta M^{tan}(X) \equiv \delta M(X) = \int X \cdot \mu^{tan} d||\delta M^{tan}||, \quad |\mu^{tan}| = 1 \text{ $||\delta M^{tan}||$-a.e.}
\]
for all tangential $X$, and some $||\delta M^{tan}||$-integrable, unit-vector-valued function $\mu^{tan}$.  If $||\delta M^{tan}|| << \mu_M$, then let us write
\[
\delta M^{tan}(X) = -\int H^{tan}_M \cdot X d\mu_M.
\]

\begin{theorem}[\cite{GrJo} or \cite{Ed:fb-notes}] \label{thm:first-var-free-boundary}
Assuming the above setup on $Q, M$, then we can conclude the following:
\begin{enumerate}
\item For any non-negative $h \in C^1_c(B_1)$, we have that
\begin{align*}
\Gamma_1(h) &:= \lim_{\rho \to 0} \frac{1}{\rho} \int_{B_\rho(Q)} (M \cdot \iota) h \\
&= \int h Dr \cdot \mu^{tan} d||\delta M^{tan}|| + \int (M \cdot D^2 r) h - \nabla h \cdot \nabla r d\mu_M
\end{align*}
is a Radon measure on $B_1$, and for any $X \in C^1_c(B_1, \R^{n+1})$ (not necessarily tangential) we have
\[
\delta M(X) = \int X \cdot \mu^{tan} d||\delta M^{tan}|| - \Gamma_1(X \cdot Dr).
\]

\item In particular, $\delta M$ is a Radon measure in $B_1$, and if we write
\[
\delta M(X) = \int X \cdot \mu d||\delta M||, \quad |\mu| = 1 \text{ $||\delta M||$-a.e.},
\]
then $||\delta M^{tan}|| = |\mu^{TB}| ||\delta M||$, where
\[
\mu^{TB}(x) = \left\{ \begin{array}{l l} \mu(x) & x \not \in B \\ \pi_{T_x B}(\mu(x)) & x \in B \end{array} \right. .
\]

\item If $||\delta M^{tan}|| << \mu_M$, then we can write
\[
\delta M(X) = -\int H^{tan}_M \cdot X d\mu_M + \int \eta \cdot X d\sigma,
\]
where $\sigma \perp \mu_M$ is a non-negative Radon measure supported in $Q$, and for $\sigma$-a.e. $x$ we have $|\eta(x)| = 1$, $\eta(x) \in (T_x Q)^\perp$.
\end{enumerate}
\end{theorem}

\subsection{Prescribed boundary}

Let $Q$ be an $(n-1)$-manifold.  Assume that $||\delta M||(B_1 \setminus Q) < \infty$, so that for every $X \in C^1_c(B_1 \setminus Q \R^{n+k})$, we can write
\[
\delta M(X) = \int X \cdot \mu d||\delta M||, \quad |\mu(x)| = 1 \text{ $||\delta M||$-a.e. x}.
\]
If $\delta M \llcorner Q^C << \mu_M$, then let us write
\[
\delta M(X) = - \int H \cdot X d\mu_M.
\]

\begin{theorem}[\cite{All:boundary}]\label{thm:first-var-boundary}
Assuming the above on $Q, M$, then we can conclude the following.
\begin{enumerate}
\item For any non-negative $h \in C^1_c(B_1)$, we have that
\begin{align*}
\Gamma_2(h) &:= \lim_{\rho \to 0} \frac{1}{\rho} \int_{B_\rho(Q)} h \\
&= \int_{B^C} h Dr \cdot \mu d||\delta M|| \\
&\quad - \int \nabla h \cdot \nabla r + (M \cdot \sigma^\perp) h/r + (M \cdot (\tau \circ \zeta - D\zeta)) h/r d\mu_M
\end{align*}
is a Radon measure on $B_1$, and for any $X \in C^1_c(B_1, \R^{n+1})$, we have
\[
\delta M(X) = \int_{B^C} X \cdot \mu d||\delta M|| - \Gamma_2(X \cdot Dr) .
\]
In particular, $\delta M$ is a Radon measure on $B_1$.

\item We have
\[
\lim_{\rho \to 0} \frac{1}{\rho} \int_{B_\rho(B)} |\nabla r - Dr|^2 d\mu_M = 0.
\]

\item If $\delta M \llcorner B^C << \mu_M$, then for any $X$ we can write
\[
\delta M(X) = -\int H \cdot X d\mu_M + \int \eta \cdot X d\sigma,
\]
where $\sigma \perp \mu_M$ is a non-negative Radon measure supported in $Q$, and for $\sigma$-a.e. $x$ we have $|\eta(x)| = 1$, $\eta(x) \in (T_x Q)^\perp$.

\end{enumerate}
\end{theorem}

\subsection{Higher codimension boundary}

Assume now $p \leq n-2$.  Assume $||\delta M||(B_1 \setminus Q) < \infty$, and additionally, assume that there is some constant $C$ such that
\[
\mu_M(B_\rho(x)) \leq C \rho^n \quad \forall x \in Q \cap B_1.
\]

\begin{theorem}\label{thm:first-var-high-codim}
Assuming the above on $Q, M$, then $||\delta M||$ is a Radon measure on $B_1$, and $||\delta M||(Q) = 0$.
\end{theorem}

\begin{proof}
Follows directly from the fact that, for any $\theta < 1$, we have 
\[
\frac{1}{\rho} \mu_M(B_\rho(Q) \cap B_\theta) \leq c'(\theta, C) \rho \to 0 \quad \text{ as } \rho \to 0. \qedhere
\]
\end{proof}

\section{Appendix: First variation and Sobolev inequalities}\label{sec:sobolev}

Here we sketch a proof of the inequalities \eqref{eqn:sobolev-1}, \eqref{eqn:sobolev}, for a rectifiable $n$-varifold $M$ satisfying the first variation bound \eqref{eqn:deltaM-bounds} and the condition $\theta_M \geq 1$ $\mu_M$-a.e.
\begin{proof}[Proof of \eqref{eqn:sobolev-1}, \eqref{eqn:sobolev}]
Combining the isoperimetric bound of \cite[Theorem 7.1]{All} (see also the \cite[Lemma 2.3]{MiSi}) with \eqref{eqn:deltaM-bounds} we get
\begin{equation}\label{eqn:first-var-sobolev-0}
\int_{ h \geq 1 } h d\mu_M \leq c(\OmegaRef)\left( \int h d\mu_M \right)^{1/n} \int |H^{tan}_M| h + |\nabla h| d\mu_M
\end{equation}
for all non-negative $h \in C^1_c(B_1)$.  Take $\gamma_\eps : \R \to \R$ a $C^\infty$ function which is $0$ on $(-\infty, 0]$, and $1$ on $[\eps, \infty)$, and then plug in $\gamma_\eps(h - t)$ into \eqref{eqn:first-var-sobolev-0} in place of $h$, to obtain
\begin{equation}\label{eqn:first-var-sobolev-1}
\mu_M( h > t + \eps) \leq c(\OmegaRef) \mu_M(h > t)^{1/n} \left( \int_{h > t} |H^{tan}_M| - \frac{d}{dt} \int \gamma_\eps(h - t)|\nabla h| d\mu_M \right).
\end{equation}

If $n = 1$, then we can integrate \eqref{eqn:first-var-sobolev-1} to obtain
\[
\int_0^{\sup h} \frac{\mu_M(h > t + \eps)}{\mu_M(h > t)} dt \leq c(\OmegaRef) \int |H^{tan}_M|h  + |\nabla h| d\mu_M.
\]
Since for a.e. $t$ the integrand $\mu_M(h > t + \eps) / \mu_M(h > t) \to 1$ as $\eps \to 0$ we obtain \eqref{eqn:sobolev-1} by the dominated convergence theorem.

If $n \geq 2$, then we multiply \eqref{eqn:first-var-sobolev-1} by $(t+\eps)^{1/(n-1)}$ to get
\begin{align*}
&\mu_M( h > t+\eps) (t+\eps)^{1/(n-1)} \\
&\quad \leq c \left( \int_{h > 0} (h+\eps)^{n/(n-1)} d\mu_M \right)^{1/n} \left( \int_{h > t} |H^{tan}_M| d\mu_M - \frac{d}{dt} \int \gamma |\nabla h| d\mu_M \right) .
\end{align*}
Now integrate in $t \in [0, \infty)$:
\[
\int (h+\eps)^{n/(n-1)} - \eps^{n/(n-1)} d\mu_M \leq c \left( \int_{h > 0} (h+\eps)^{n/(n-1)} d\mu_M \right)^{1/n} \left( \int |H^{tan}_M| h + |\nabla h| d\mu \right) ,
\]
and let $\eps \to 0$ to get \eqref{eqn:sobolev}.
\end{proof}

\section{Appendix: Fourier expansion in cones}\label{sec:fourier}

\begin{proof}[Proof of Lemma \ref{lem:expansion}]
First, since $u \in W^{1,2}(CD \cap B_1)$, we have that $\omega \mapsto u(r\omega) \in L^2(D)$ for every $0 < r < 1$, and hence for each such $r$ we can expand in $L^2(D)$:
\begin{equation}\label{eqn:exp-1}
u(r\omega) = \sum_i c(r) \phi_i(\omega), \quad c_i(r) = \int_D u(r\omega) \phi_i(\omega).
\end{equation}
By Fatou's lemma, this expansion holds in $L^2(CD \cap B_1)$.

It's easy to check that
\[
c'_i(r) = \int_D (\del_r u) \phi_i
\]
weakly, and so $c_i(r)\phi_i(\omega) \in W^{1,2}_{loc}(CD \cap B_1 \setminus \{0\})$.  Using that the $\phi_i$ are Neumann eigenfunctions, we can bound
\[
\int_{CD \cap B_s \setminus B_r} \left| \sum_{i=0}^N D(c_i \phi_i) \right|^2 \leq \int_{CD \cap B_1} |Du|^2 \quad \forall 0 < r < s < 1, \forall N,
\]
and hence the expansion \eqref{eqn:exp-1} holds in $W^{1,2}(CD \cap B_1)$ also.

Using the equation \eqref{eqn:lem-exp-hyp}, and the definition the $\phi_i$, then one can verify that
\[
0 = \int_0^1 c_i(r) r^{n-1} (\eta'' + (n-1) \eta'/r - \mu_i \eta/r^2) dr \quad \forall \eta \in C^2_c(0, 1).
\]
Setting $f(t) = c_i(e^t) e^{(n-2)t}$, then this implies that $f$ solves the linear equation $f'' - (n-2) f' - \mu_i f = 0$ in the weak sense, and hence in the strong sense.  Therefore, when $n \geq 3$, we have
\begin{equation}\label{eqn:exp-2}
c_i(r) = A_i r^{\gamma_i^+} + B_i r^{\gamma_i^-}, \quad \gamma_i^\pm = -((n-2)/2) \pm \sqrt{((n-2)/2)^2 + \mu_i},
\end{equation}
for some constants $A_i, B_i$.  If $n = 2$, then \eqref{eqn:exp-2} holds for $i \geq 1$, but when $i = 0$ then
\[
c_0(r) = A_0 + B_0 \log(r).
\]

We just need to show each $B_i = 0$.  Suppose otherwise, that $B_i \neq 0$ for some $i$.  Since $\lambda^+_i \geq 0$ and $\lambda^-_i \leq -(n-2)$, we can find a radius $0 < r_0 < 1$ so that for $r < r_0$ we have
\[
|A_i| \lambda^+_i r^{\lambda^+_i} \leq \frac{1}{4} |B_i| |\lambda^-_i| r^{\lambda^-_i}.
\]
Therefore we have
\begin{align}
\int_{B_1} |Du|^2 \geq \int_0^1 (c_i')^2 r^{n-1} dr &\geq \frac{1}{4} |B_i|^2 \int_0^{r_0} r^{2\gamma_i^- + n- 3} dr \\
&\geq \frac{1}{4} |B_i|^2 \int_0^{r_0} r^{-n+1} dr =\infty
\end{align}
which is a contradiction.  If $n = 2$ and $i = 0$, then we have the similar contradiction
\[
\int_{B_1} |Du|^2 \geq |B_0|^2 \int_0^1 dr/r = \infty.
\]
This shows every $B_i = 0$, and hence proves Lemma \ref{lem:expansion} for $n \geq 2$.  The $n = 1$ case is trivial.
\end{proof}

\bibliographystyle{plain}
\bibliography{refs}

\end{document}